\documentclass[a4paper, 11pt]{article}
\usepackage[utf8]{inputenc}
\usepackage{hyperref}
\usepackage{amssymb,amsmath,amsfonts,mathdots,amsthm,textcomp,enumerate,mathrsfs,eufrak,version,mathtools} 
\usepackage{upgreek}
\usepackage{caption}
\usepackage{array}
\usepackage{lipsum}
\usepackage{multirow}
\usepackage{vmargin}
\usepackage{titlesec}
\usepackage{eurosym}
\usepackage[english]{babel}
\usepackage[nottoc,notlot,notlof]{tocbibind}
\usepackage[titletoc,toc,title]{appendix}
\titleformat{\subsubsection}{\Large\scshape\raggedright}{}{0em}{}[\titlerule]
\usepackage[final]{pdfpages}
\usepackage{colortbl}
\usepackage{url}
\usepackage[T1]{fontenc}
\usepackage{color}
\usepackage[all,cmtip]{xy}
\usepackage{tikz}
\usetikzlibrary{arrows.meta}
\tikzset{%
  >={Latex[width=2mm,length=2mm]},
            base/.style = {rectangle, rounded corners, draw=black,
                           minimum width=4cm, minimum height=1cm,
                           text centered, font=\sffamily},
}
\newtheorem{theo}{Theorem}[section]

\newtheorem{prop}[theo]{Proposition}
\newtheorem{lem}[theo]{Lemma}

\newcommand{\R}{\mathbb{R}}

\newcommand{\C}{\mathbb{C}}

\newcommand{\etalchar}[1]{$^{#1}$}

\title{Micro-packets for real groups of type $\mathrm{G}_{2}$}

\author{ Nicolas Arancibia Robert, Leticia Barchini, Paul Mezo}

\begin{document}

\maketitle

\begin{abstract}
In their study of Arthur's conjectures for real groups, Adams,
Barbasch, and Vogan introduced the notion of
\emph{micro-packets}. Micro-packets are finite sets of irreducible
representations defined using microlocal geometric methods and
characteristic cycles. We explore an action of the Weyl group on
characteristic cycles to compute all micro-packets of real groups of
type $\mathrm{G}_2$. 
\end{abstract}

\tableofcontents

\section{Introduction}

Langlands
introduced \emph{L-packets} for connected real algebraic groups in
\cite{Langlands}.  L-packets are finite sets (of equivalence
classes) of irreducible admissible representations which were motivated
by number-theoretic correspondences.  

Arthur used the trace formula to  explore such correspondences and
was led to conjecture the existence of
complementary packets for some unitary representations
\cite{Arthur84}, \cite{Arthur89}.  These are commonly called \emph{Arthur
packets} or \emph{A-packets}.  Every A-packet contains a specific
L-packet, but not every L-packet is obtained in this manner. A further
conjecture of Arthur's is the existence of a stable distribution, in
the sense of Langlands and Shelstad \cite{Shel82}, which is associated to each
Arthur packet.  
       
In their seminal work \cite{ABV}, Adams, Barbasch and Vogan, gave a
definition of A-packets for connected real reductive groups.  In fact, they
generalized the notion of an A-packet to that of a
\emph{micro-packet}, and gave precise definitions for these too.  They also
defined a stable distribution associated to each micro-packet. Every 
micro-packet contains a specific L-packet 
and all L-packets are obtained in this manner.  In particular, there
is a natural bijection between L-packets
and  micro-packets.

Our principal goal is to explicitly compute the
micro-packets and the associated stable distributions for the real forms of the
complex algebraic group of type $\mathrm{G}_{2}$. Some analogues of
micro-packets for $p$-adic groups of type $\mathrm{G}_{2}$ are computed in
\cite{Cunningham21, Cunningham22}.

To give an outline of how micro-packets are defined, we need to
fix a few objects.   First, we denote the
complex group of type $\mathrm{G}_{2}$ simply by $\mathrm{G}_{2}$, and
a real form of it by
$\mathrm{G}_{2}(\mathbb{R})$.  The real form is either split or
compact \cite[Table  II]{Tits}.  
The Langlands dual group
${^\vee}\mathrm{G}_{2}$ of $\mathrm{G}_{2}$ is isomorphic to
$\mathrm{G}_{2}$ itself.  Nevertheless, it will be helpful in keeping
track of the roots of one group and the coroots of the other to
preserve the distinction in notation.  Fix  Borel subgroups $B
\subset \mathrm{G}_{2}$, ${^\vee}B \subset {^\vee}\mathrm{G}_{2}$ and
maximal tori  $T \subset B$, ${^\vee}T \subset {^\vee}B$.  Denote
the root system of $\mathrm{G}_{2}$ relative to $T$ by
$R(\mathrm{G}_{2},T)$, and the positive 
roots determined by $B$ as $R(B,T)$.  We may identify
$R(\mathrm{G}_{2},T)$ with the roots $R(\mathfrak{g}_{2},
\mathfrak{t})$ of the respective complex Lie algebras, and this
notation carries over to the dual objects in the obvious fashion.  

Fix $\lambda \in {^\vee}\mathfrak{t}$. The element $\lambda$ should be
regarded as the infinitesimal character 
of a representation and we will refer to $\lambda$ as an infinitesimal
character.  An infinitesimal
character is  actually the conjugacy class of a semisimple element.
For this reason the
Weyl group action makes it harmless to 
assume that $\lambda$ is \emph{integrally dominant} with respect to 
${^\vee}B$, \emph{i.e.}
\begin{equation}
\label{intdom}
{^\vee}\alpha(\lambda) \in \mathbb{Z} \Longrightarrow
{^\vee}\alpha(\lambda) \geq 0,  \quad {^\vee}\alpha \in R({^\vee}B,
{^\vee}T).
\end{equation}
This is all we need to proceed with our outline, but the presentation
is greatly simplified with the two additional assumptions that
$\lambda$ is \emph{regular} and \emph{integral}, \emph{i.e.}
\begin{equation}
\label{regint}
{^\vee}\alpha(\lambda) \in \{1,2,3,\ldots\}, \quad {^\vee}\alpha
\in R({^\vee}\mathrm{G}_{2}, {^\vee}T).
\end{equation}
We continue with our outline under assumption (\ref{regint}) and will
say a few words about how the removal of the regularity and
integrality assumptions affects the presentation afterwards.

The framework of \cite{ABV}, specialized to the group
$\mathrm{G}_{2}$, is summarized in 
the following diagram.
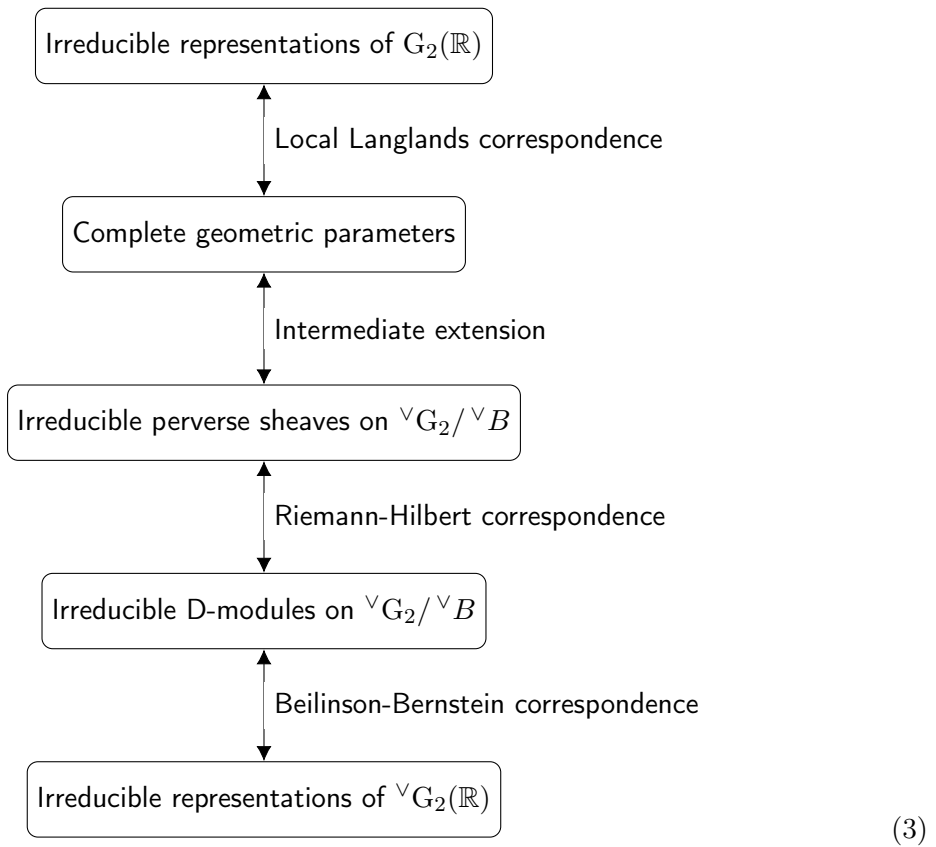
\begin{equation}
\label{landscape}
\begin{tikzpicture}[node distance=2.5cm,
    every node/.style={fill=white, font=\sffamily}, align=center]
  % Specification of nodes (position, etc.)
  \node (repG2)             [base]              {Irreducible
    representations of $\mathrm{G}_{2}(\mathbb{R})$};
  \node (geoparam)     [base, below of=repG2]          {Complete
    geometric parameters};
  \node (sheaves)      [base, below of=geoparam]   {Irreducible
    perverse sheaves on ${^\vee}\mathrm{G}_{2}/ \,{^\vee}B$};
  \node (dmod)     [base, below of=sheaves]   {Irreducible D-modules
    on  ${^\vee}\mathrm{G}_{2}/ \,{^\vee}B$};
  \node (repG2dual)      [base, below of=dmod] {Irreducible representations
 of ${^\vee}\mathrm{G}_{2}(\mathbb{R})$};  

% Arrows
\draw[<->]             (repG2) --  (geoparam) node[midway, right]
                                   {Local Langlands correspondence};
  \draw[<->]     (geoparam) --(sheaves)  node[midway, right]
                                   {Intermediate extension};
  \draw[<->]      (sheaves) -- (dmod) node[midway, right]
                                   {Riemann-Hilbert correspondence};
  \draw[< ->]     (dmod) -- (repG2dual) node[midway, right]
                                   {Beilinson-Bernstein
                                     correspondence}; 
 
 \end{tikzpicture}
\end{equation}

It is impossible to do justice to the immense amount of mathematics
that is hinted at in this diagram.  We must be content with a meagre
description, emphasizing those points which appear in our computations
later on.  All arrows in the diagram indicate bijections. 

The top bijection is \cite[Theorem 10.4]{ABV}.  It is an extension of
Langlands' original correspondence \cite{Langlands}, which we
assume the reader understands.  On the one side of the
bijection we have
(infinitesimal equivalence classes of) irreducible admissible
representations of both the split and compact forms of
$\mathrm{G}_{2}$.  On the other side of the bijection we have
\emph{complete geometric parameters} \cite[Definition 7.6]{ABV}, whose
description requires more effort.  The underlying geometric object is
the flag variety ${^\vee}\mathrm{G}_{2}/ \, {^\vee}B$
\cite[Proposition 6.16]{ABV}.  The flag
variety is acted upon by a symmetric  subgroup ${^\vee}K \subset
{^\vee}\mathrm{G}_{2}$, and each ${^\vee}K$-orbit $S \subset
{^\vee}\mathrm{G}_{2}/ \, {^\vee}B$ corresponds to a
unique (equivalence class of an) L-parameter $\varphi$
\cite[Proposition 6.17]{ABV}.  Conversely, every L-parameter (for the
split form) corresponds to a unique ${^\vee}K$-orbit, for an
appropriate choice of ${^\vee}K$.  This yields a correspondence $\varphi
\leftrightarrow S$
between L-parameters and symmetric orbits on the flag variety. The
correspondence is enhanced by including an irreducible
${^\vee}K$-equivariant local system
$\mathcal{V}$ of complex vector spaces on $S$.  The monodromy
representation of $\mathcal{V}$ is an irreducible representation of the
Langlands component group for $\varphi$ \cite[Lemma 7.5]{ABV}.  A
complete geometric parameter is  a pair $\xi = (S, \mathcal{V})$,
and the enhanced local Langlands correspondence carries $\xi$ to a unique
(infinitesimal equivalence class of an) irreducible admissible
representation $\pi(\xi)$ of a real form of $\mathrm{G}_{2}$.  Every
such representation is obtained from a unique complete geometric
parameter in this way.  This completes our description of the
bijection of the local Langlands correspondence.

We move to the second bijection in the diagram.  Let
$(S, \mathcal{V})$ be a complete geometric parameter and $j: S
\hookrightarrow \overline{S}$ be the inclusion of $S$ in its closure.  
To pass from $(S, \mathcal{V})$ to an
irreducible perverse sheaf on ${^\vee}\mathrm{G}_{2}/ \, {^\vee}B$ one
applies the intermediate extension functor $j_{!*}$ to the shifted
complex $\mathcal{V}[\dim S]$ followed by the direct image functor of
the closed embedding $\overline{S} \hookrightarrow
{^\vee}\mathrm{G}_{2}/ \, {^\vee}B$  \cite[(7.10)]{ABV}.
The resulting irreducible perverse sheaf is an
intersection cohomology complex of $(S,\mathcal{V})$ \cite[Definition
  3.3.9]{Achar}.  Every irreducible perverse sheaf on
${^\vee}\mathrm{G}_{2}/ \, {^\vee}B$ is obtained in this manner
\cite[Theorem 3.4.5]{Achar}.
% Exercise 3.3.1 Achar, and for closed embeddings the intermediate
% extension is the same as the direct image
%  In the reverse direction one applies restriction to $S$ \cite[Lemma
%  3.3.11]{Achar}. 
The intermediate extension functor is actually being applied to
${^\vee}K$-equivariant perverse sheaves.  We refer to
\cite[I.5.2]{Lunts} and \cite[Section 6.5]{Achar} for the equivariant
versions of this functor.   This completes our
description of the second bijection. 

We have little to say about the remaining two bijections.  The
D-modules in the diagram are ${^\vee}K$-equivariant
regular holonomic sheaves of $\mathcal{D}$-modules on
${^\vee}\mathrm{G}_{2}/ \, {^\vee}B$, where $\mathcal{D}$ is the sheaf
of algebraic differential operators on ${^\vee}\mathrm{G}_{2}/ \,
{^\vee}B$ \cite[Definition 7.7]{ABV}, \cite[Chapter 6]{Hotta}.  The
Riemann-Hilbert correspondence 
is proven without reference to equivariance in \cite[Theorem
  VIII.14.4]{Boreletal} and is asserted in the equivariant setting in
\cite[I.4.2]{Lunts}.  

 The final bijection is the Beilinson-Bernstein correspondence
\cite[Theorem 8.3]{ABV}, \cite[Sections 11.5-11.6]{Hotta}.  Its image is the set
of irreducible $({^\vee}\mathfrak{g}_{2}, {^\vee}K)$-modules whose
infinitesimal character is $\rho$, the half-sum of the roots in $R({^\vee}B,
{^\vee}T)$. These are the underlying modules of the irreducible
admissible representations of ${^\vee}\mathrm{G}_{2}(\mathbb{R})$ with
infinitesimal character $\rho$.
% modules annihilated by the ideal generated by z-\chi(z) where z is
% in the centre of the universal enveloping algebra and \chi is the
% Harish-Chandra map.
%
% See also (16e) and section 6.1 of Adams-Vogan 15, where g=rho

Having seen the general landscape in which we will be working, 
it is convenient to streamline the notation. Set
$$X = {^\vee}\mathrm{G}_{2}/ \, {^\vee}B$$
and denote the set of ${^\vee}K$-orbits by ${^\vee}K \backslash X$. 
We use $\xi = (S,
\mathcal{V})$ to denote a complete geometric parameter and the
irreducible objects of (\ref{landscape}) are denoted by
\begin{equation}
\label{landscape1}
\xymatrix@-1pc{\pi(\xi)  \ar@{<->}[d] & \mbox{representation of
  }\mathrm{G}_{2}(\mathbb{R}) 
\\ \xi  \ar@{<->}[d]    & \mbox{complete geometric parameter}  \\ P(\xi)
  \ar@{<->}[d] & \mbox{perverse sheaf}  \\ D(\xi)  \ar@{<->}[d] &
  \mbox{D-module}  \\ \pi({^\vee}\xi) & \mbox{representation of }
       {^\vee}\mathrm{G}_{2}(\mathbb{R}) }
\end{equation}

Continuing on our path to the definition of micro-packets, we 
focus on the D-modules.  The definition of a micro-packet hinges on two
invariants attached to an 
irreducible D-module $D = D(\xi)$.   
The first invariant is
the \emph{characteristic variety} $\mathrm{Ch}(D)$ \cite[Section
  19]{ABV}, \cite[Sections 2.1-2.2]{Hotta}.  It is a closed algebraic
subvariety of the complex cotangent bundle $T^{*}X$.  In fact, it is
contained in the conormal bundle $T^{*}_{{^\vee}K}X$ to the
${^\vee}K$-action on $T^{*}X$ \cite[(19.1), Proposition 19.12
  (b)]{ABV}.  

To every irreducible component $C$ of $\mathrm{Ch}(D)$
there is an associated local module whose length we denote by
$m_{C}(M)$ \cite[Definition 19.7]{ABV}.  The ${^\vee}K$-orbit of $C$
is of the form 
$\overline{T^{*}_{S}X}$ for some $S \in {^\vee}K \backslash X$
\cite[Lemma 19.2 (b)]{ABV}, and every irreducible component in this
orbit has length $m_{C}(D)$.  This leads to the definition of the
\emph{characteristic cycle} of $D$ as a formal sum of irreducible components
\begin{equation}
\label{cc}
CC(D) = \sum_{C} m_{C}(D) \, C = \sum_{S \in {^\vee}K \backslash X}
\chi_{S}^{\mathrm{mic}}(D) \ \overline{T^{*}_{S}X},
\end{equation}
where $\chi_{S}^{\mathrm{mic}}(D) = m_{C}(D)$ for any $C$ in the
${^\vee}K$-orbit corresponding to $\overline{T^{*}_{S}X}$.  The
non-negative integer 
$
\chi_{S}^{\mathrm{mic}}(D)
$
 is called the
\emph{microlocal multiplicity along} $S$.  

We are now able to define the micro-packet attached to an L-parameter $\varphi$
with infinitesimal character $\lambda$  \cite[Definition 5.2,
  Proposition 5.6]{ABV}.
%The L-parameter $\varphi$ is a homomorphism
%$\varphi:W_{\mathbb{R}} \rightarrow 
%{^\vee}\mathrm{G}_{2} \times \Gamma$, where $W_{\mathbb{R}} =
%\mathbb{C}^{\times} \cup j \mathbb{C}^{\times}$ is the
%real Weil group, and $\Gamma$ is the Galois group of
%$\mathbb{C}/\mathbb{R}$.
Recall that  an L-parameter $\varphi$ corresponds
to a unique $S \in {^\vee}K \backslash X$.  The micro-packet of
$\varphi$ is defined as
\begin{equation}\label{mpacket}
\Pi^{\mathrm{mic}}_{\varphi} = \{ \pi(\xi') :
\chi_{S}^{\mathrm{mic}}(D(\xi')) \neq 0 \}
\end{equation}  
\cite[Definition 19.13]{ABV}.  It is a finite set of representations
of possibly both the split and compact real forms of
$\mathrm{G}_{2}$.  The
associated stable distribution is 
\begin{equation}
\label{etamic}
\eta^{\mathrm{mic}}_{\varphi} = \sum_{\pi(\xi') \in
  \Pi^{\mathrm{mic}}_{\varphi}} e(\xi') \, (-1)^{\dim S' - \dim S}\, 
\chi^{\mathrm{mic}}_{S} (D(\xi')) \, \pi(\xi'),
\end{equation}
in which $\pi(\xi')$, for $\xi' = (S', \mathcal{V}')$, is identified with
its distribution character, 
and $e(\xi')$ is the Kottwitz sign attached to the real form of which
$\pi(\xi')$ is a representation \cite[(17.24)(g), Corollary 19.16]{ABV}.  
It is obvious from these definitions that the computation of the
micro-packets is equivalent to the computation of the microlocal
multiplicities, or equivalently, the characteristic
cycles.  For this reason, we return to (\ref{cc}).

Taking a characteristic cycle defines a $\mathbb{Z}$-linear map
\begin{equation}
\label{eq:CC-map}
CC: \mathscr{K}(X, {^\vee}K) \rightarrow \mathscr{L}(X,
{^\vee}K)
\end{equation}
\cite[Theorem 2.2.3]{Hotta}.  The domain of the map is
the Grothendieck group $\mathscr{K}(X,{^\vee}K)$ of the category of
${^\vee}K$-equivariant regular holonomic sheaves of
$\mathcal{D}$-modules on $X$.  The codomain of the map is
\begin{equation}
\label{LXK} 
\mathscr{L}(X,{^\vee}K)=\left\{\sum_{S\in {^\vee}K\setminus
  X}m_S\ \overline{T^{\ast}_{S}X}:m_S\in\mathbb{Z}\right\}
\end{equation}
the free
$\mathbb{Z}$-module generated by $\overline{T^{*}_{S}X}$, $S \in
{^\vee}K \backslash X$.  Through the bijections of (\ref{landscape1})
we have isomorphisms between $\mathscr{K}(X, {^\vee}K)$ and the
Grothendieck groups of the category of
${^\vee}K$-equivariant perverse sheaves on $X$, and the
category of admissible representations of
${^\vee}\mathrm{G}_{2}(\mathbb{R})$ with infinitesimal character
$\rho$.  We identify $\mathscr{K}(X, {^\vee}K)$ with the other two
isomorphic Grothendieck groups,  so that
\begin{align} \label{eq:characteristicvariety}
\begin{aligned}
 CC(P(\xi)) &= CC(\pi({^\vee}\xi)) = CC(D(\xi))\\
 \chi^{\mathrm{mic}}_{S} (P(\xi)) &= \chi^{\mathrm{mic}}_{S}
(\pi({^\vee}\xi)) = \chi^{\mathrm{mic}}_{S} (D(\xi))\\ 
 \mathrm{Ch}(P(\xi)) &= \mathrm{Ch}(\pi({^\vee}\xi)) =
  \mathrm{Ch}(D(\xi))
  \end{aligned}
\end{align} 

One of the main tools in our computation of the micro-packets is the
interaction of the Weyl group $W = W({^\vee}\mathrm{G}_{2}, {^\vee}T)$
with $CC$.  In the setting of representations, there is a
$W$-action on $\mathscr{K}(X, {^\vee}K)$ commonly called the \emph{coherent
continuation representation} \cite[Definition 7.2.28]{greenbook}.  
% Equivalent to 3.1 in Tanisaki, or section 13.1 in Hotta's book (see
% page 310).
This $W$-action may be explicitly computed using the Atlas of Lie Groups
and Representation software \cite{atlas}.  On the other hand, there is 
a $W$-action on $\mathscr{L}(X, {^\vee}K)$ given in
\cite{Hotta85}, \cite{Rossmann95}.  By \cite[Theorem 1]{Tanisaki} the
map $CC$ is $W$-equivariant.  In Section \ref{section:G2computations} we use the
$W$-equivariance of $CC$  to 
generate a number of linear equations whose coefficients involve microlocal
multiplicities.  We can then solve for some of the microlocal
multiplicities by making elementary substitutions.  

Another means of simplifying the equations  is
through the use of the \emph{associated variety}.  We define the
associated variety of an irreducible object in $\mathscr{K}(X,
{^\vee}K)$ by
\begin{equation}
\label{eq:AVdef}
\mathrm{AV}(P(\xi))  = \mu \left(
\mathrm{Ch}(P(\xi)) \right),
\end{equation}
where $\mu : T^{*}X \rightarrow \mathfrak{g}_{2}^{*}$ is the
\emph{moment map} \cite[(20.3)]{ABV}.
Since $\mathrm{Ch}(P(\xi))$
is the support of $CC(P(\xi))$, information about $\mathrm{AV}(P(\xi))$
gives us information about which microlocal multiplicities can be non-zero.
The image of the moment map
lies in the cone of nilpotent elements in $\mathfrak{g}_{2}^{*}$ and
is a union of finitely many known ${^\vee}K$-orbits
\cite{Dokovic}. It follows 
from the ${^\vee}K$-equivariance of $P(\xi)$ that
$\mathrm{AV}(P(\xi))$ is a finite union of ${^\vee}K$-orbits
\cite[Lemma 20.16]{ABV}.  These  are
known in some cases \cite{Samples}.  A compatibility property between
the coherent continuation representation and the associated varieties,
provides additional information about  $\mathrm{AV}(P(\xi))$ in
Propositions \ref{prop:Cell-AV0} and \ref{prop:Cells-AV}.   This
additional information allows us to compute 
all microlocal multiplicities. The micro-packets and associated
stable distributions are given in Theorem \ref{theo:Micropackets}.

This concludes our overview of micro-packets
$\Pi^{\mathrm{mic}}_{\varphi}$ when the infinitesimal character
$\lambda$ of $\varphi$ is regular and integral (\ref{regint}).  The
framework of (\ref{landscape}) does not fundamentally change if we
drop the assumption of integrality on $\lambda$.  Consider the setting
in which the infinitesimal character $\lambda$ is merely integrally dominant
(\ref{intdom}) and  regular, \emph{i.e.}
\begin{equation}
\label{reg}
{^\vee}\alpha(\lambda) \neq 0, \quad {^\vee}\alpha \in
R({^\vee}\mathrm{G}_{2}, {^\vee}T).
\end{equation}
In this broader
setting we replace ${^\vee}\mathrm{G}_{2}$ with the subgroup
${^\vee}\mathrm{G}_{2}(\lambda)$, which is the centralizer of
$\exp(2 \uppi i \lambda)$ in ${^\vee}\mathrm{G}_{2}$ \cite[(6.2)(b),
  Proposition 6.16]{ABV}.  We determine
the possibilities for  ${^\vee}\mathrm{G}_{2}(\lambda)$, along with
the possibilities for its symmetric subgroups ${^\vee}K(\lambda)$, in
Section \ref{section:FlagNonIntegral}.   The regular integrally
dominant element $\lambda$ fixes 
a Borel subgroup ${^\vee}B(\lambda) \subset
{^\vee}\mathrm{G}_{2}(\lambda)$.  To compute micro-packets with
non-integral regular infinitesimal character $\lambda$, all we need 
to do is replace the groups ${^\vee}\mathrm{G}_{2}$, ${^\vee}K$ and
${^\vee}B$ with 
${^\vee}\mathrm{G}_{2}(\lambda)$, ${^\vee}K(\lambda)$ and
${^\vee}B(\lambda)$ respectively, and apply the same methods.  As
${^\vee}\mathrm{G}_{2}(\lambda)$ is a smaller group this is a much
easier task.  We compute these micro-packets and stable distributions in
 Section \ref{section:G2computationsNon-integral}.

Removing the further assumption of regularity (\ref{reg}) has the effect of
replacing the 
Borel subgroup ${^\vee}B(\lambda)$ with a parabolic subgroup
${^\vee}P(\lambda) \subset {^\vee}\mathrm{G}_{2}(\lambda)$
\cite[(6.2)(e), Proposition 6.16]{ABV}.   In this way the full flag
variety ${^\vee}\mathrm{G}_{2}(\lambda)/ \, {^\vee}B(\lambda)$ is
replaced by a partial flag variety ${^\vee}\mathrm{G}_{2}(\lambda)/ \,
{^\vee}P(\lambda)$.  One is able to obtain information in the singular
setting from the regular setting using the \emph{translation functors}
of Jantzen, Vogan and Zuckerman \cite[Propositions 8.8 and 16.6]{ABV}.
In particular, the microlocal multiplicities, and so the micro-packets
and stable distributions, are obtained through the translation functors
\cite[Proposition 20.1 (e)]{ABV}.  

Having ascertained the microlocal multiplicities,
we finish by computing  \emph{Kashiwara's local index formula} for
$\mathrm{G}_{2}$.  In essence, this formula provides a change of basis
between the microlocal multiplicities and another set of maps called
\emph{local multiplicities}.  This is a geometric relationship that may be
motivated in the following representation-theoretic language.
The stable distributions $\eta_{\varphi}^{\mathrm{mic}}$ in
\eqref{etamic} form a basis for the 
lattice of stable virtual representations  
of the pure real forms of $\mathrm{G}_2$ with infinitesimal
character $\lambda$ \cite[Corollary 19.16]{ABV}.
There is another basis for this space of stable virtual
representations, namely the basis consisting of  standard representations
$\eta_\varphi^{\mathrm{loc}}$  
%attached to $L$-parameters $\varphi$ of $\mathrm{G}_2$
\cite[Definition~18.9, Lemma 18.11]{ABV}.  
In \cite[Lemma~18.15]{ABV}, the stable distribution
$\eta_{\varphi}^{\mathrm{loc}}$ is expressed as a linear combination of
representations whose coefficients are given by $\mathbb{Z}$-linear
functionals
$$\chi_{S}^{\mathrm{loc}} : 
  \mathscr{K}(X(\lambda),{}^{\vee}K(\lambda)) \longrightarrow \mathbb{Z},$$
called \emph{local multiplicities} \cite[Definition~1.28]{ABV}.   
The change of basis matrix between the two bases
$\{\eta_{\varphi}^{\mathrm{mic}}\}$ and 
$\{\eta_{\varphi}^{\mathrm{loc}}\}$ entails a ``change of basis''
\begin{equation}
\label{Kashiwaraformula}
\chi_{S}^{\mathrm{loc}} = \sum_{S'} (-1)^{\mathrm{dim}(S')} a(S,
S') \, \chi_{S'}^{\mathrm{mic}} 
\end{equation}
between the sets of $\mathbb{Z}$-linear functionals $\{
\chi_{S}^{\mathrm{mic}}\}$ and $\{\chi_{S}^{\mathrm{loc}}\}$.
This equation is Kashiwara's local index formula \cite[Section~2]{Kashiwara73}. 
In fact, the coefficients  $a(S, S')$ are integers.
%for every ${}^{\vee}K(\lambda)$-orbit $S'$
%$ \subset \overline{S}$, 
The integer coefficients $a(S, S')$ are constructed by different
methods in \cite[Section~3]{MacPherson74} and are
individually called
the \emph{local Euler obstruction} of $S$ at $S'$.  
%The coincidence of Kashiwara’s coefficients $a(S, S')$ with
%MacPherson’s local Euler obstruction was established by Dubson. 
The local Euler obstructions have been extensively studied from various
perspectives  
\cite[\S 11.7]{Ginsburg86}, \cite{Gonzalez-Sprinberg}.  Nevertheless,
to our knowledge, no method is known for computing them in full
generality.  By our work in the setting of ${^\vee}\mathrm{G}_{2}(\lambda)$,
%However, Theorems~\ref{theo:CCs} and~\ref{theo:CCsNonintegral} and
%equation \eqref{eq:translationofchiV1},
we know the values of $\chi_{S'}^{\mathrm{mic}}$ on the right of
(\ref{Kashiwaraformula}).  
% for any ${}^{\vee}K(\lambda)$-orbit $S'$.  
In Section \ref{KashiwaraSection} we explain how the Atlas of Lie Groups and
Representations software allows us to compute the values of
$\chi_{S}^{\mathrm{loc}}$  on the left of (\ref{Kashiwaraformula}).  
Having these values on both sides of (\ref{Kashiwaraformula}) allows
us to compute  the local Euler
obstructions $a(S,S')$.

The first author thanks the Fields Institute for Research in
Mathematical Sciences for supporting a Fields Research Fellowship
visit at Carleton University, where this work began. The second author
thanks the organizers of the special program ``Representation Theory
and Noncommutative Geometry'' held at the Institut Henri Poincar\'{e} in
2025. She also thanks the the Institut Henri Poincar\'{e} for their warm
hospitality.  The third author was supported by NSERC grant
RGPIN-06361.

\section{${^\vee}K$-orbits of the flag variety of ${^\vee}\mathrm{G}_{2}$}
\label{section:Korbitsflag}

We begin this section by determining all possibilities for the symmetric
groups ${^\vee}K \subset {^\vee}\mathrm{G}_{2}$ of the introduction,
and concluding that there is   
only one such group of any interest. In Section \ref{Korbits} we
parameterize the ${^\vee}K$-orbits of the flag variety,
and describe the partial order on them given by a closure relation.
This amounts to an application of \cite{Adams-Fokko} and \cite{RS90}
to ${^\vee}\mathrm{G}_{2}$.  In Section \ref{compgroupsec}, we examine
the correspondence between ${^\vee}K$-orbits and L-parameters, and
compute the Langlands component group for each L-parameter.  These
computations give us a concrete picture of the geometric objects for
regular and integral infinitesimal character $\lambda$.  We close by
repeating the computations when  $\lambda$ is regular and
non-integral, that is finding the objects
${^\vee}\mathrm{G}_{2}(\lambda)$,  ${^\vee}K(\lambda)$,
\emph{etc.}~alluded to in the introduction.

There are some additional objects to fix beyond the Borel pairs
$B \supset T$ and 
${^\vee}B \supset {^\vee}T$ of the introduction.
We identify the set of
positive roots $R(B,T)$ with the set of positive coroots
${^\vee}R({^\vee}B, {^\vee}T)$.  Similarly, the characters $X^{*}(T)$
are identified with the cocharacters $X_{*}({^\vee}T)$ of ${^\vee}T$.
We identify the complex Lie algebra
${^\vee}\mathfrak{t}$ of ${^\vee}T$ with $X^{*}(T) \otimes_{\mathbb{Z}}
\mathbb{C}$.  Define
the surjection
\begin{equation}
\label{emap}
{^\vee}\mathfrak{t} \stackrel{e}{\rightarrow} {^\vee}T
\end{equation}
by $e(X) = \exp(2 \uppi i X)$.  Clearly $\ker e = X^{*}(T)$, which is
also the root lattice spanned by $R(\mathrm{G}_{2},T)$.  
%The roots  $R({^\vee}\mathrm{G}_{2},
%{^\vee}T)$ are identified with roots  $R({^\vee}\mathfrak{g}_{2},
%{^\vee}\mathfrak{t})$ of the Lie algebras through the exponential map.  In the
%same manner,  the coroots of the group are identified with the coroots of
%the Lie algebra.  
Let $\alpha_{1}$ and $\alpha_{2}$ be the simple roots in $R(B,T)$, with
$\alpha_{1}$ being the long root.  With this, $\rho = 5\alpha_{2} +
3\alpha_{1}$ is the half-sum of the positive roots.

%We fix $\lambda \in
%{^\vee}\mathfrak{t}$ to be 
%\emph{regular} and \emph{integral}, \emph{i.e.}
%$$\alpha(\lambda) \in \mathbb{Z} \mbox{ and } \alpha(\lambda) >0 $$
%for all $\alpha \in R({^\vee}\mathfrak{b}, {^\vee}\mathfrak{t})$.
%This is equivalent to $e(\lambda) = 1$  or
%${^\vee}\mathrm{G}_{2}(\lambda) = {^\vee}\mathrm{G}_{2}$.

\subsection{The group ${^\vee}K$}
\label{groupK}

Let us
review the definition of ${^\vee}K$ and specialize to the case at
hand.  Let ${^\vee}\mathrm{G}_{2}^{\Gamma} = {^\vee}\mathrm{G}_{2} \times
\Gamma$ be the Galois form of the L-group of $\mathrm{G}_{2}$,
and set
$$\mathcal{I} = \{y \in {^\vee}\mathrm{G}_{2}^{\Gamma} -
{^\vee}\mathrm{G}_{2} : y^{2} =1\}.$$
Since the L-group is a direct product, we may identify $\mathcal{I}$
with the set 
$$ \{y \in {^\vee}\mathrm{G}_{2}:  \ y^{2} =1\}.$$
The group ${^\vee}\mathrm{G}_{2}$ acts by conjugation on $\mathcal{I}$
with finitely many orbits \cite[Lemma 6.12]{ABV}.  By definition, a
symmetric group
${^\vee}K$ is taken to be the centralizer in 
${^\vee}\mathrm{G}_{2}$ of a representative $y$ of a
${^\vee}\mathrm{G}_{2}$-orbit of $\mathcal{I}$ \cite[Proposition 6.16]{ABV}. 
\begin{lem}
\label{Iorbits}
There are exactly two ${^\vee}\mathrm{G}_{2}$-orbits of $\mathcal{I}$, and
the two orbits are represented by the elements $1$ and
$e(\rho/2) = \exp(\uppi i \rho)$.   
\end{lem}
\begin{proof}
 As the elements in $\mathcal{I}$ are
semisimple they lie in maximal tori. Since  ${^\vee}\mathrm{G}_{2}$ acts
transitively on the set of maximal tori under conjugation, it follows that
the ${^\vee}\mathrm{G}_{2}$-orbits of $\mathcal{I}$ are in bijection
with the Weyl group orbits of
$$\{ y \in {^\vee}T: y^{2} = 1\}.$$
 The condition $y^{2} = 1$ is equivalent to $y^{2}$ being in the centre of
${^\vee}\mathrm{G}_{2}$, and the latter is equivalent to $y^{2}$ being
in the kernel of
${^\vee}\alpha_{1}$ and ${^\vee}\alpha_{2}$. Using (\ref{emap}), we
may write  
$$y = e(c_{1} \alpha_{1} + c_{2} 
\alpha_{2}) = \exp(2 \uppi i (c_{1} \alpha_{1} + c_{2} 
\alpha_{2}))$$
 for some $c_{1}, c_{2} \in \mathbb{C}$.  The kernel
property  for ${^\vee}\alpha_{1}$ and ${^\vee}\alpha_{2}$ translates into 
$$c_{1} \langle \, {^\vee}\alpha_{j}, \alpha_{1} \rangle + c_{2} \langle \,
{^\vee}\alpha_{j}, \alpha_{2} \rangle \in \frac{1}{2} \mathbb{Z},
\quad j = 1,2.$$
As pairings between roots and coroots are integers, we see that
$y^{2} = 1$ if and only if 
$c_{1}, c_{2} \in \frac{1}{2} \mathbb{Z}$.  If $c_{1}$ or $c_{2}$ are
integers then they may be removed from $e(c_{1} \alpha_{1} + c_{2}
\alpha_{2})$.  Thus, we are left with four possibilities for $y$
$$1, e((\alpha_{1}+\alpha_{2})/2), e(\alpha_{1}/2), e(\alpha_{2}/2).$$
However, $\alpha_{2}/2$ and $(\alpha_{1} + \alpha_{2})/2$ are Weyl group
conjugates. In addition, $\alpha_{1}/2$ is a Weyl group conjugate of $(3
\alpha_{2} + 2 \alpha_{1})/2$, and 
$$e((3\alpha_{2} + 2 \alpha_{1}) /2) = e(\alpha_{2}/2 + 2(\alpha_{2} +
\alpha_{1})/2) = e(\alpha_{2}/2).$$
This proves that there are at most two ${^\vee}\mathrm{G}_{2}$-orbits
of $\mathcal{I}$, with respective representatives $1$ and $e(\alpha_{2}/2)$.  It
is easily verified that 
${^\vee}\alpha_{2}(e(\alpha_{2}/2)) \neq 1$, so there are exactly two
orbits.  Finally,
$$e(\rho/2) = e((5\alpha_{2} + 3\alpha_{1})/2) = e((\alpha_{1} + \alpha_{2})/2)$$
is also a representative of the non-trivial orbit.
\end{proof}

 By Lemma
\ref{Iorbits}, the group ${^\vee}K$ may be taken to equal either
${^\vee}\mathrm{G}_{2}$, when $y=1$, or may be taken to equal the
centralizer of $e(\rho/2)$.  In the former case, there is only a
single ${^\vee}K$-orbit of ${^\vee}\mathrm{G}_{2}/\, {^\vee}B$,
namely  ${^\vee}\mathrm{G}_{2}/\, {^\vee}B$ itself.  For the rest of
this section we will therefore assume that ${^\vee}K$ is the
centralizer of $e(\rho/2)$.  

The group ${^\vee}\mathrm{G}_{2}$ is simply connected and $e(\rho/2)$
is semisimple, so by Steinberg's Theorem, the centralizer ${^\vee}K$
is connected  \cite[Theorem   2.11]{HumphCC}. By \cite[Theorem
  2.2]{HumphCC}, the group ${^\vee}K$ is generated by ${^\vee}T$ and the root
subgroups $U_{{^\vee}\alpha} \subset {^\vee}\mathrm{G}_{2}$ such that
${^\vee}\alpha(e(\rho/2)) = 1$.  These are given by the roots ${^\vee}\alpha \in
R({^\vee}\mathrm{G}_{2},{^\vee}T)$ such that $\langle {^\vee}\alpha,\rho \rangle
\in 2 \mathbb{Z}$.  Using \cite[Lemma 13.3.A]{Humphreys} one may prove
that there are exactly two positive roots satisfying this property,
${^\vee}\alpha_{1}+ {^\vee}\alpha_{2}$ and $3 {^\vee}\alpha_{1}+
{^\vee}\alpha_{2}$.  (Note that ${^\vee}\alpha_{1}$ is a short root.) These two 
roots are orthogonal.  Consequently, 
\begin{equation}
\label{checkK}
{^\vee}K = \langle {^\vee}T, U_{\pm ({^\vee}\alpha_{1}+
 {^\vee}\alpha_{2})},  U_{\pm (3{^\vee}\alpha_{1}+ {^\vee}\alpha_{2})}
\rangle 
=\langle  U_{\pm ({^\vee}\alpha_{1}+
 {^\vee}\alpha_{2})} \rangle \ \langle U_{\pm (3{^\vee}\alpha_{1}+
  {^\vee}\alpha_{2})} \rangle  
\end{equation}
% Theorem 8.1.5 in Springer's Linear algebraic groups
is the product of two commuting subgroups, each of which being
isomorphic to $\mathrm{SL}_{2}$.  An element in the intersection of
these two commuting subgroups must  lie in the centre of each
subgroup.  The non-trivial central elements of $\langle  U_{\pm
  ({^\vee}\alpha_{1}+ 
 {^\vee}\alpha_{2})} \rangle$ and $\langle U_{\pm
  (3{^\vee}\alpha_{1}+    {^\vee}\alpha_{2})} \rangle$ are $\exp(\uppi
i (3\alpha_{2} +  \alpha_{1}))$ and  $\exp(\uppi i
(\alpha_{1} + \alpha_{2}))$ respectively. 
% These are SL_2 computations together with finding the coroots of the
% given roots
 Since
$$\exp(\uppi i (3\alpha_{2} +  \alpha_{1}))  = \exp( \uppi i
(2\alpha_{2} + (\alpha_{1}+ \alpha_{2}))) = \exp(\uppi i
(\alpha_{1} + \alpha_{2}))$$
we see that the intersection of the two subgroups is a group of order
two  generated by $\exp(\uppi i
(\alpha_{1} + \alpha_{2}))$.
The intersection is isomorphic to $\upmu_{2} =
\{\pm 1 \}$ and so ${^\vee}K$  is isomorphic to the fibre
product $\mathrm{SL}_{2} \times_{\upmu_{2}} \mathrm{SL}_{2}.$

\subsection{${^\vee}K$-orbits of
  ${^\vee}\mathrm{G}_{2}/ \,{^\vee}B$}
\label{Korbits}

The algebraic group ${^\vee}K$ acts by left multiplication on
${^\vee}\mathrm{G}_{2}/ {^\vee}B$.  The closure of each
${^\vee}K$-orbit is a union of ${^\vee}K$-orbits. 
%  Humphreys "Linear algebraic groups"  Proposition 8.3
 This closure
relation gives rise to the \emph{Bruhat order} on the
${^\vee}K$-orbits \cite{RS90}. The
goal of this section is to  describe the ${^\vee}K$-orbits and the
Bruhat order.

We begin by giving a complete set of representatives
$y_{0}, \ldots, y_{\ell} \in {^\vee}\mathrm{G}_{2}$ for
the ${^\vee}K$-orbits 
$${^\vee}K \backslash {^\vee}\mathrm{G}_{2} / \,{^\vee}B  = \{
{^\vee}K y_{0}  {^\vee}B, \ldots ,  {^\vee}K
y_{\ell}  {^\vee}B\}$$ 
by following \cite{Adams-Fokko}.
We shall also employ the Atlas of Lie Groups and Representations
software \cite{atlas}, which is an outgrowth of  \cite{Adams-Fokko}.  
There are a number of sets in \cite{Adams-Fokko} which are pertinent
to the description of  ${^\vee}K \backslash {^\vee}\mathrm{G}_{2} /
{^\vee}B$.  We present them here, taking the simplifications of the
special case of ${^\vee}\mathrm{G}_{2}$ into consideration.  Set
$$\widetilde{\mathcal{X}} = \{ \xi \in
\mathrm{Norm}_{{^\vee}\mathrm{G}_{2}} ({^\vee}T) : \xi^{2} = 1\}.$$
The torus ${^\vee}T$ acts on $\widetilde{\mathcal{X}}$ by conjugation.
Let $\mathcal{X}$ be the set of ${^\vee}T$-conjugacy classes.  The
element  $e(\rho/2) = \exp(\uppi i \rho)$ defining ${^\vee}K$ belongs to
$\widetilde{\mathcal{X}}$.  Let $S_{0} = \{e(\rho/2)\} 
\in \mathcal{X}$ be its 
${^\vee}T$-conjugacy class.  
According to \cite[(9.8)]{Adams-Fokko}, there is a bijection between
${^\vee}K \backslash {^\vee}\mathrm{G}_{2} / {^\vee}B$ and the subset
$\mathcal{X}[S_{0}] \subset \mathcal{X}$ of elements whose pre-images in
$\widetilde{\mathcal{X}}$ are ${^\vee}\mathrm{G}_{2}$-conjugate 
to $e(\rho/2)$.  By \cite[(8.1), (9.1)]{Adams-Fokko}, each element $S
\in \mathcal{X}[S_{0}]$ is the  
 ${^\vee}T$-conjugacy class of $y e(\rho/2) y^{-1}$ for
some $y \in {^\vee}\mathrm{G}_{2}$, and the element $S \in
\mathcal{X}[S_{0}]$ corresponds to the ${^\vee}K$-orbit ${^\vee}K
y^{-1} {^\vee}B$ 
$$S \longleftrightarrow {^\vee}K y^{-1} \, {^\vee}B.$$
Thus, the problem of determining ${^\vee}K \backslash
{^\vee}\mathrm{G}_{2} / \,{^\vee}B$ is equivalent to determining the
elements $S_{0}, \ldots, S_{\ell} \in \mathcal{X}[S_{0}]$ and  $y_{0},
\ldots ,y_{\ell} \in {^\vee}\mathrm{G}_{2}$ as indicated.  Using
the results \cite[Proposition 11.2, Lemma 14.2 and Lemma
  14.11]{Adams-Fokko},  the elements $S_{0},
\ldots ,S_{\ell} \in  \mathcal{X}[S_{0}]$ may be obtained by
taking \emph{cross-actions} \cite[(9.11  f)]{Adams-Fokko} and
\emph{Cayley transforms} \cite[Definition 14.1]{Adams-Fokko}
from $S_{0}$.  The
command \texttt{KGB} in the Atlas of Lie Groups and Representations
software implements  this approach and produces ten elements,
$y_{0} = 1, y_{1}, \ldots , y_{9} \in
{^\vee}\mathrm{G}_{2}$ which parameterize the  ${^\vee}K$-orbits.

Let us provide some more detail.  From now on we fix a  pair of simple
root vectors   $X_{{^\vee}\alpha_{1}}, X_{{^\vee}\alpha_{2}}
\in {^\vee}\mathfrak{g}_{2}$ which then fixes the pinning
$$({^\vee}B, {^\vee}T, \{ X_{{^\vee}\alpha_{1}}, X_{{^\vee}\alpha_{2}}\})$$
of ${^\vee}\mathrm{G}_{2}$.  
Every
element of $\widetilde{\mathcal{X}}$ belongs to a coset in the Weyl group
$$W = W({^\vee}\mathrm{G}_{2}, {^\vee}T) =
\mathrm{Norm}_{{^\vee}\mathrm{G}_{2}}({^\vee}T)/ {^\vee}T$$ 
and this correspondence determines a map
%\label{eq:pmap}
$$p: \mathcal{X}[S_{0}] \rightarrow W$$
\cite[(9.11i)]{Adams-Fokko}. It is clear that the image of $p$ lies in
the subset of elements in $W$ which square to the identity.   In
consequence, the fibres of the map 
$p$ partition $\mathcal{X}[S_{0}]$ into subsets
$p^{-1}(w) = \mathcal{X}_{w}$, where $w \in W$ satisfies $w^{2} = 1$. Every
simple reflection $s_{j} = s_{{^\vee}\alpha_{j}} 
\in W$, $j=1,2$,  has a Tits representative 
\begin{equation}
\label{titsrep}
\sigma_{j} =  \exp\left( \frac{\uppi}{2} (X_{{^\vee}\alpha_{j}} -
X_{-{^\vee}\alpha_{j}} )  \right)\in
\widetilde{\mathcal{X}}
\end{equation}
%  \exp(X_{{^\vee}\alpha_{j}}) \, \exp(-X_{{^\vee}\alpha_{j}}^{\intercal}) \,
%  \exp(X_{{^\vee}\alpha_{j}})\cite[Lemma 8.1.4]{Springer}
%which is a matrix in $\mathrm{SL}_{8}$ in our realization of
%${^\vee}\mathrm{G}_{2}$.  
\cite[Section 12]{AVParameters}.  It follows that every element $S \in \mathcal{X}_{w}$ is a
${^\vee}T$-conjugacy class of an 
element of the form $t \sigma_{w}$, where $t \in {^\vee}T$,
$\sigma_{w}$ is a product of Tits representatives, and the
element $t \sigma_{w}$ is 
${^\vee}\mathrm{G}_{2}$-conjugate to $e(\rho/2)$.

We can describe $\mathcal{X}[S_{0}]$ by starting with
$\mathcal{X}_{1}$, where the $1$ in subscript is the identity element in
$W$.  The obvious choice of basepoint in $\mathcal{X}_{1}$ is $S_{0}$, the
${^\vee}T$-conjugacy class of 
$e(\rho/2)$.  The remaining elements in $\mathcal{X}_{1}$ are obtained by
cross-action of (imaginary) elements in $W$
\cite[(12.17)]{Adams-Fokko}.  This 
amounts to repeatedly conjugating $e(\rho/2)$ by the Tits representatives
(\ref{titsrep}) of the simple reflections.  The Atlas software tells
us that these cross actions produce exactly three elements in
$\mathcal{X}_{1}$: one 
corresponding to the base point $S_{0}$, and two further elements
$S_{1}$ and $S_{2}$ equal to the ${^\vee}T$-conjugacy classes of  
$\sigma_{2} e(\rho/2) \sigma_{2}^{-1}$ and  $\sigma_{1} e(\rho/2)
\sigma_{1}^{-1}$ respectively.  

We can pass from the basepoint $S_{0}$ for $\mathcal{X}_{1}$ to a
basepoint in $\mathcal{X}_{s_{j}}$  by a Cayley transform with
respect to ${^\vee}\alpha_{j}$, $j=1,2$.  \cite[Equation (14.5)]{Adams-Fokko}
expresses these Cayley transforms in terms of conjugation by an element
$$g_{j}  = \exp\left( \frac{\uppi}{4} ( X_{-{^\vee}\alpha_{j}} -
X_{{^\vee}\alpha_{j}}) \right), \quad j = 1,2.$$
(\emph{cf.}~\cite[(6.65a)]{Beyond}).  In this manner, the basepoint of
$\mathcal{X}_{s_{j}}$ is the ${^\vee}T$-conjugacy class
of $g_{j} e(\rho/2) g_{j}^{-1}$.  As with
$\mathcal{X}_{1}$, the remaining elements of
$\mathcal{X}_{s_{j}}$ are obtained from $g_{j}
e(\rho/2) g_{j}^{-1}$ through cross-actions of Weyl group elements
(which are imaginary with respect to $s_{j}$).  There happen to be no
such cross actions, and so 
$\mathcal{X}_{s_{j}}$ is a singleton.  The elements of
the remaining $\mathcal{X}_{w}$ are obtained through the similar
iterations of Cayley transforms and cross-actions.  A summary is given
in the following table.
\begin{center}
\begin{tabular}{|l|l|l|l|}
\hline
$j$ & Representative of  $S_{j} \in \mathcal{X}[S_{0}]$ &
$y_{j}$ & $p(S_{j})$ \\ 
\hline
\hline
$0$ & $e(\rho/2)$ & $1$ & $1$\\
\hline
$1$ &   $\sigma_{2} e(\rho/2) \sigma_{2}^{-1}$ &
$\sigma_{2}^{-1}$ & $1$\\
\hline
$2$ & $\sigma_{1} e(\rho/2) \sigma_{1}^{-1}$ &
$\sigma_{1}^{-1}$ & $1$\\ 
\hline 
$3$ &   $g_{2} e(\rho/2) g_{2}^{-1}$ &
$g_{2}^{-1}$ & $s_{2}$ \\
\hline
$4$ &  $g_{1} e(\rho/2) g_{1}^{-1}$ &
$g_{1}^{-1}$ & $s_{1}$\\
\hline
$5$ &  $\sigma_{2}g_{1} e(\rho/2)
(\sigma_{2}g_{1})^{-1}$ & 
$(\sigma_{2}g_{1})^{-1}$ & $s_{2} s_{1} s_{2}$ \\
\hline
$6$ &  $\sigma_{1}g_{2} e(\rho/2)
(\sigma_{1}g_{2})^{-1}$ & 
$(\sigma_{1}g_{2})^{-1}$ & $s_{1}s_{2} s_{1}$ \\
\hline
$7$ &  $\sigma_{2}\sigma_{1}g_{2} e(\rho/2)
(\sigma_{2}\sigma_{1}g_{2})^{-1}$ & 
$(\sigma_{2}\sigma_{1}g_{2})^{-1}$ & $s_{2}s_{1}s_{2} s_{1}s_{2}  $\\
\hline
$8$ &  $\sigma_{1}\sigma_{2}g_{1} e(\rho/2)
(\sigma_{1}\sigma_{2}g_{1})^{-1}$ & 
$(\sigma_{1}\sigma_{2}g_{1})^{-1}$  & $s_{1}s_{2} s_{1} s_{2}s_{1}$\\
\hline
$9$ &  $g_{2}\sigma_{1}\sigma_{2}g_{1} e(\rho/2)
(g_{2}\sigma_{1}\sigma_{2}g_{1})^{-1}$ & 
$(g_{2}\sigma_{1}\sigma_{2}g_{1})^{-1}$ & $s_{2}s_{1}s_{2} s_{1}
s_{2}s_{1}$ \\  
\hline
\end{tabular}
\captionof{table}{${^\vee}K$-orbit data}\label{orbitdata}
\end{center}

We now turn to the description of the Bruhat order on these ten orbits.
The minimal orbits in the Bruhat order are characterized as the closed orbits
\cite[Proposition 8.3]{HumphLAG}.  
\begin{lem}
\label{closedorb}
The closed $^{\vee}K$-orbits are $S_{0}$, $S_{1}$ and $S_{2}$.
\end{lem}
\begin{proof}
From (\ref{checkK}) we see that
${^\vee}B \cap  {^\vee}K = \langle {^\vee}T, U_{{^\vee}\alpha_{1}+
 {^\vee}\alpha_{2}},  U_{3{^\vee}\alpha_{1}+ {^\vee}\alpha_{2}}
\rangle $ is a Borel subgroup of ${^\vee}K$.  Hence, ${^\vee}K /
({^\vee}B \cap {^\vee}K)$ is a flag variety.  As a flag variety, it is
a projective variety.  There is an obvious injection
$${^\vee}K /({^\vee}B \cap {^\vee}K) \rightarrow
{^\vee}\mathrm{G}_{2}/{^\vee}B$$
and its  image is ${^\vee}K\, {^\vee}B =   {^\vee}K
y_{0}{^\vee}B$, \emph{i.e.}~the orbit $S_{0}$.  Since it is the
image of a projective 
variety, the orbit $S_{0}$ is closed.  Similarly, the orbit $S_{2}$ is of the
form 
\begin{equation}
\label{orbit1}
{^\vee}K y_{2} {^\vee}B = {^\vee}K \sigma_{1}^{-1} {^\vee}B =
\sigma_{1}^{-1} \cdot 
(\sigma_{1} {^\vee}K \sigma_{1}^{-1}) \,{^\vee}B.
\end{equation} 
From (\ref{checkK}) we deduce that 
$$\sigma_{1} {^\vee}K \sigma_{1}^{-1} = \langle {^\vee}T, U_{\pm
  (2{^\vee}\alpha_{1}+ 
 {^\vee}\alpha_{2})},  U_{\pm {^\vee}\alpha_{2}} \rangle.$$
Clearly, we may replace ${^\vee}K$ with $\sigma_{1} {^\vee}K
\sigma_{1}^{-1}$ in the  argument for $S_{0}$ to conclude that
$(\sigma_{1} {^\vee}K \sigma_{1}^{-1}) \,{^\vee}B$ is closed in
${^\vee}G_{2}/ {^\vee}B$.  The orbit of $S_{2}$ is then closed by
(\ref{orbit1}).  The same reasoning applies to prove that the orbit of
$S_{1}$ is closed.
\end{proof}

By Lemma \ref{closedorb}, $S_{0}, S_{1}$ and $S_{2}$ minimal in
the Bruhat order.   With the aim of placing the remaining orbits
$S_{3}, \ldots , S_{9}$ 
in the Bruhat order, we use an action defined by Richardson and
Springer  of simple root reflections on the orbits 
 \cite[4.7]{RS90}.  Given a simple root
${^\vee}\alpha_{j}$ and an orbit ${^\vee}Ky_{\ell} {^\vee}B$, they
produce a union of orbits $P_{s_{j}}({^\vee}Ky_{\ell} {^\vee}B)$ 
containing a unique open orbit $m(s_{j})({^\vee}Ky_{\ell}
{^\vee}B)$.  
%  $V$ in [RS90] is defined in 1.2.  It is equal to $\{y_{0}^{-1},
%  \dldots, y_{9}^{-1}\}.   The Weyl group action on $V$ is
%  defined in 2.  This action appears in 4.3.4.
This defines an action $m$ of the simple reflections on the set of
${^\vee}K$-orbits.   
By \cite[Lemma 4.6]{RS90}, every orbit is obtained by iterating this
action.  More precisely, every orbit ${^\vee}Ky_{j} {^\vee}B$
is of the form  
\begin{equation}
\label{RSorbit}
{^\vee}Ky_{j} {^\vee}B = m(w_{k})m(w_{k-1}) \cdots m(w_{1})
({^\vee}Ky_{\ell} {^\vee}B),
\end{equation} 
for some choice of simple root reflections $w_{1}, \ldots, w_{k} \in W$  
and some  closed orbit ${^\vee}Ky_{\ell} {^\vee}B$.  Furthermore, the
closure of the orbit ${^\vee}Ky_{j} {^\vee}B$ is 
\begin{equation}
\label{RSclosure}
\overline{{^\vee}Ky_{j} {^\vee}B} = P_{w_{k}} \cdots
P_{w_{1}}({^\vee}Ky_{\ell} {^\vee}B).
\end{equation} 
An elementary and explicit description of the terms on the right is given in
\cite[4.3]{RS90}.
% complex/real/imaginary for $v$ defined in 1.6.  To see that this definition is
% equivalent to the one in IC4, observe that $\alpha$ is imginary for
% $\dot(y)$ in the sense of [RS90] iff $Ad(e(\rho/2) \dot(y))(\alpha)
% = Ad(y) (\alpha)$ iff $Ad(y^{-1} e(\rho/2) \dot(y))
% \alpha = \alpha$  iff $Ad (y) \alpha = \alpha$.  The last equation
% is the definition of $\alpha$ to be imaginary in IC4.    
 It is  therefore not
difficult to compute that all of the remaining orbits  are actually
obtained as in (\ref{RSorbit}) with $\ell=0$, \emph{i.e.}~from 
the closed orbit ${^\vee}K y_{0} {^\vee}B = {^\vee}K\,
{^\vee}B$.  The orbit closures are then easily computed from
(\ref{RSclosure}) with $\ell = 0$. 
%\begin{center}
%\begin{tabular}{|l|l|}
%\hline
%$j$ & ${^\vee}Ky_{j} {^\vee}B$ \\
%\hline
%\hline
%3 & $m(s_{2}) ({^\vee}K \, {^\vee}B)$ \\
%\hline
%4 &  $m(s_{1}) ({^\vee}K \, {^\vee}B)$ \\
%\hline
%5 &  $m(s_{2}) m(s_{1}) ({^\vee}K \, {^\vee}B)$\\
%\hline
%6 &  $m(s_{1})m(s_{2}) ({^\vee}K \, {^\vee}B)$\\
%\hline
%7  & $m(s_{2}) m(s_{1})m(s_{2}) ({^\vee}K \, {^\vee}B)$ \\
%\hline
%8 &  $m(s_{1})m(s_{2}) m(s_{1}) ({^\vee}K \, {^\vee}B)$\\
%\hline
%9 & $m(s_{2})m(s_{1})m(s_{2}) m(s_{1}) ({^\vee}K \, {^\vee}B)$ \\
%\hline
%\end{tabular}
%\captionof{table}{${^\vee}K$-orbits obtained from $m$-action}\label{orbitdata2}
%\end{center}

For example, the simple root ${^\vee}\alpha_{1}$ is  imaginary  for $p(S_{0})
= 1 \in W$ (\cite[1.6]{RS90}, where $\varphi$ there is replaced by $p$
here). 
% there are bijections S_{j} \mapsto {^\vee}K y_{j}^{-1} B
% \mapsto {^\vee}B S_{j} {\vee}K \maps to TS_{j}K, where the second
% map is induced by inversion and the last map is Theorem 1.3
% \cite{RS90}
\cite[4.3.4]{RS90} tells us that for some $3\leq j \leq 9$
\begin{align*}
\overline{{^\vee}Ky_{j} {^\vee}B} &= \overline{m(s_{1})({^\vee}K \,
     {^\vee}B)} \\
&= P_{s_{1}}({^\vee}K\, {^\vee}B)\\
& = m(s_{1})({^\vee}K \,
     {^\vee}B) \cup ({^\vee}K\, {^\vee}B) \cup ({^\vee}K y_{2} {^\vee}B)\\
\end{align*}
We then use  \cite[Lemma 7.4 (i)]{RS90}
% $s \circ a$ is defined in 3.1 [RS90]
 to conclude that
$p(S_{j}) = s_{1}$.  Table \ref{orbitdata} identifies the orbit number
as $j=4$, \emph{i.e.}~${^\vee}Ky_{4} {^\vee}B = m(s_{1})({^\vee}K \,
     {^\vee}B)$.  
Identical reasoning with the simple root $\alpha_{2}$ shows that
${^\vee}Ky_{3} {^\vee}B = m(s_{2})({^\vee}K \, {^\vee}B)$.

To go one step further, we argue from ${^\vee}K y_{4}
{^\vee}B$. The simple root 
$\alpha_{2}$ is complex for $p(S_{4}) = s_{1}$.  The orbit ${^\vee}K
y_{j} {^\vee}B = m(s_{2})({^\vee}K y_{4} {^\vee}B )$
satisfies $p(S_{j}) = s_{2}s_{1}s_{2}$ and so $j=5$ from Table
\ref{orbitdata}.  Thus, ${^\vee}K
y_{5} {^\vee}B = m(s_{2})({^\vee}K y_{4} {^\vee}B )$, and
the closure of ${^\vee}K y_{5} {^\vee}B$ is  
$P_{s_{2}}P_{s_{1}}({^\vee}K\, {^\vee}B)$. 
Our work on the earlier orbits and
\cite[4.3.1]{RS90} allow us to compute the closure as 
\begin{align*}
&P_{s_{2}} \left( \, ({^\vee}Ky_{4} {^\vee}B )\cup ({^\vee}Ky_{0}
{^\vee}B) \cup ({^\vee}Ky_{1} {^\vee}B) \, \right)\\
&  =  ({^\vee}Ky_{5}
{^\vee}B) \cup ({^\vee}Ky_{4} {^\vee}B) \cup  ({^\vee}Ky_{3}
{^\vee}B) \cup  ({^\vee}Ky_{2} {^\vee}B) \cup ({^\vee}Ky_{1}
{^\vee}B) \cup ({^\vee}Ky_{0} {^\vee}B).
\end{align*}
% $m(s_{1}) {^\vee}K y_{1} {^\vee}B$ is a 

A summary of the closure relations
between the orbits is given in the following Hasse diagram of the
Bruhat order.
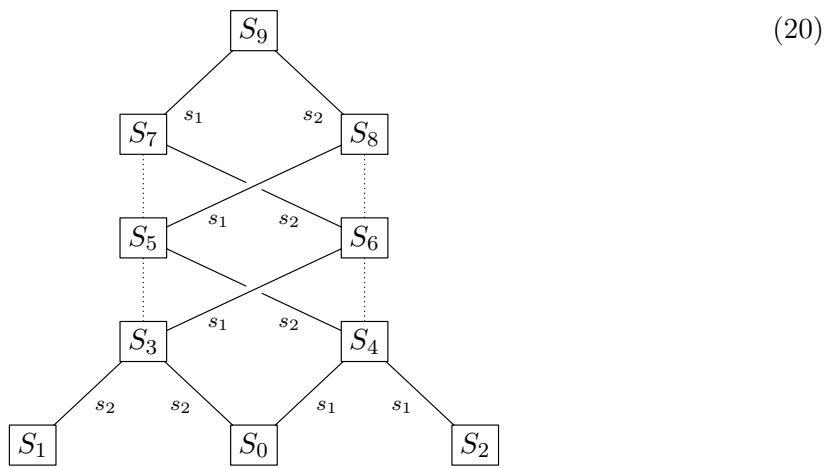
\begin{align}\label{eq:BruhatOrbits}
\xymatrix{
& & *+[F]{S_{9}} \ar@{-}[dr]_(.7){s_{2}}  \ar@{-}[dl]^(.7){s_{1}}& &\\
&  *+[F] {S_{7}} \ar@{.}[d] \ar@{-}[drr]_(.7){s_{2}}|\hole &  &
  *+[F]{S_{8}}  \ar@{.}[d] 
  \ar@{-}[dll]^(.7){s_{1}}&\\ 
& *+[F]{S_{5}}  \ar@{.}[d] \ar@{-}[drr]_(.7){s_{2}}|\hole & & *+[F]{S_{6}} \ar@{.}[d]
  \ar@{-}[dll]^(.7){s_{1}} & \\
& *+[F]{S_{3}}  \ar@{-}[dl]^{s_{2}} \ar@{-}[dr]_{s_{2}} & &  *+[F]{S_{4}}
  \ar@{-}[dl]^{s_{1}} 
  \ar@{-}[dr]_{s_{1}} & \\
 *+[F]{S_{1}} & &  *+[F]{S_{0}} &  & *+[F]{S_{2}}}
 \end{align}
In this diagram a solid line indicates that the higher orbit is
obtained from the lower one by applying $m$ of the labelled simple
reflection.  For example $S_{4} = m(s_{1}) S_{0}$.  If one orbit is
contained in the closure of the other, but is not related through the
action of $m$  then the line is dotted. 

In the foregoing discussion we have not gone into any detail about what
it means for a root to be (compact/non-compact) imaginary, real or
complex relative to $S_{j}$.   The reader may
wish to review the equivalent definitions given in \cite[1.6]{RS90}
and \cite[12]{Adams-Fokko}.  Apart from the computations summarized above,
in Section \ref{cohcont} it is important  to know 
the nature of the simple roots relative to $S_{j}$.  We have used the
Atlas software to do the elementary computations for the following table. 
\begin{center}
\begin{tabular}{|l|l|l|}
\hline
$S_{j}$ & ${^\vee}\alpha_{1}$ & ${^\vee}\alpha_{2}$\\
\hline \hline 
$S_{0}$ & imaginary non-compact & imaginary non-compact \\
\hline
$S_{1}$ &   imaginary compact & imaginary non-compact\\
\hline
$S_{2}$ & imaginary non-compact & imaginary compact\\ 
\hline 
$S_{3}$ &  complex & real \\
\hline
$S_{4}$ &  real & complex \\
\hline
$S_{5}$ & complex & complex \\
\hline
$S_{6}$ &  complex & complex \\
\hline
$S_{7}$ & imaginary non-compact & complex \\
\hline
$S_{8}$ & complex & imaginary non-compact \\
\hline
$S_{9}$ &  real & real \\  
\hline
\end{tabular}
\captionof{table}{The nature of the simple roots relative to
  $S_{j}$}\label{orbitdata3} 
\end{center}

\subsection{Langlands Component groups}
\label{compgroupsec}

To each L-parameter there is attached a Langlands component
group. Recall from the introduction that each L-parameter
corresponds to a ${^\vee}K$-orbit $S$.  The Langlands component group is
important as its irreducible characters  parameterize  the
irreducible local systems $\mathcal{V}$ on $S$  \cite[Lemma 7.3
  (e)]{ABV}.  These local systems are the ones appearing in the
complete geometric parameters $\xi = (S,\mathcal{V})$.

We review the definitions and compute the Langlands component groups
in the context of $\mathrm{G}_{2}$.   
In the present context, an L-parameter is a homomorphism
$\varphi: W_{\mathbb{R}} \rightarrow {^\vee}\mathrm{G}_{2}$, where
$W_{\mathbb{R}} = \mathbb{C}^{\times} \cup j \, \mathbb{C}^{\times}$
is the Weil group of $\mathbb{C}/\mathbb{R}$.  
Without loss of generality, we may assume that
\begin{equation}
\label{lparam}
\varphi(z) = z^{\lambda} \bar{z}^{\mathrm{Ad}(y) \lambda}, \quad z
\in \mathbb{C}^{\times},
\end{equation}
where $y = \exp(\uppi i \lambda)  \varphi(j)$ and $\lambda \in
{^\vee}\mathfrak{t}$ is dominant relative to ${^\vee}B$
(\emph{cf.}~\cite[(5.7)]{ABV}). 
Clearly, the L-parameter 
$\varphi$ is uniquely determined by the pair $(y,\lambda)$ and so we
may write $\varphi = \varphi(y, \lambda)$. 

In this section we assume 
that the infinitesimal character $\lambda$ is dominant, regular and
integral (\ref{regint}).
Then $y^{2} = e(\lambda) = 1$.  According to \cite[Proposition 6.16
  and Proposition 6.17]{ABV}, the equivalence classes of L-parameters
with fixed regular integral infinitesimal character $\lambda$ are in
bijection with the union  of the ${^\vee}K$-orbits of
${^\vee}\mathrm{G}_{2}/ {^\vee}B$ and the
$^{\vee}\mathrm{G}_{2}$-orbit of ${^\vee}\mathrm{G}_{2}/ {^\vee}B$ (a
singleton).  We shall ignore the ${^\vee}\mathrm{G}_{2}$-orbit.  The 
${^\vee}K$-orbits are given explicitly in Table \ref{orbitdata}
and using the data listed there the bijection is induced by
\begin{equation}
\label{lparambij}
\varphi(S_{j}, \lambda) \longmapsto {^\vee}K y_{j} {^\vee}B.
\end{equation}
Set 
$$\varphi_{j} = \varphi(S_{j}, \lambda), \quad  0 \leq j \leq 9$$
so that $\{ \varphi_{0}, \ldots , \varphi_{9}\}$. Apart from the
L-parameter corresponding to the ${^\vee}\mathrm{G}_{2}$-orbit, this
is a complete set of 
representatives of the equivalence classes of L-parameters with
infinitesimal character $\lambda$.  

The Langlands component group of
$\varphi_{j}$ is defined as the component group of the 
centralizer of the image of $\varphi_{j}$ in ${^\vee}\mathrm{G}_{2}$.  By
\cite[Corollary 5.9 (c) and Lemma 12.10]{ABV}, the Langlands component
group of $\varphi_{j}$ is isomorphic to the component group of the
fixed-point subgroup 
${^\vee}T^{S_{j}}$ of ${^\vee}T$ under conjugation by $S_{j}$
(\emph{cf.}~\cite[(12.11)]{ABV}).  Evidently, the subgroup ${^\vee}T^{S_{j}}$
is equal to the fixed-point subgroup ${^\vee}T^{p(S_{j})}$ under the
Weyl group action of $p(S_{j})$. 

Let us find the Langlands component group of $\varphi_{j}$ by
identifying it with the component group of ${^\vee}T^{p(S_{j})}$ and
referring to  Table \ref{orbitdata}.  For $j = 0,1,2$ the Weyl group
element $p(S_{j})$ is trivial and so ${^\vee}T^{p(S_{j})} = {^\vee}T$.
Since the torus ${^\vee}T$ is connected, the Langlands component groups are
trivial. 

For the other extreme, consider $\varphi_{9}$.  In this case
$p(S_{9})$ is the long Weyl group element which negates all roots.
As in the proof of Lemma \ref{Iorbits}, we write an  element in ${^\vee}T$ as
$e(c_{1} \alpha_{1} + c_{2} \alpha_{2})$ where $c_{1},
c_{2} \in \mathbb{C}$.  Then $e(c_{1} \alpha_{1} + c_{2}
\alpha_{2})$ belongs to ${^\vee}T^{p(S_{9})}$ if and only if 
\begin{align*}
& p(S_{9}) \cdot (c_{1} \alpha_{1} + c_{2}  \alpha_{2})
- (c_{1} \alpha_{1} + c_{2} \alpha_{2}) \in \ker e =
X^{*}(T) \\
& \Longleftrightarrow -2c_{1}\alpha_{1} - 2c_{2} \alpha_{2} \in
X^{*}(T)\\
& \Longleftrightarrow c_{1}, c_{2} \in \frac{1}{2}\mathbb{Z}.
\end{align*}
It follows that ${^\vee}T^{p(S_{9})}$ is generated by
$e(\alpha_{1}/2)$ and $e(\alpha_{2}/2)$--a Klein 4-group.  We conclude
that the Langlands component group of $\varphi_{9}$ is isomorphic to
$\mathbb{Z}/ 2 \mathbb{Z} \times \mathbb{Z}/2 \mathbb{Z}$.  

Now consider $\varphi_{4}$, for which $p(S_{4}) = s_{2}$.  Arguing as
we did for $\varphi_{9}$,  we see that   $e(c_{1} \alpha_{1} + c_{2}
\alpha_{2})$ belongs to ${^\vee}T^{s_{2}}$ if and only if 
\begin{align*}
& s_{2} \cdot (c_{1} \alpha_{1} + c_{2}  \alpha_{2})
- (c_{1} \alpha_{1} + c_{2} \alpha_{2}) \in \
X^{*}(T) \\
& \Longleftrightarrow c_{1}(\alpha_{1}+\alpha_{2}) - c_{2} \alpha_{2}
- (c_{1} \alpha_{1} + c_{2} \alpha_{2}) \in X^{*}(T)\\
& \Longleftrightarrow c_{1}-2c_{2} \in \mathbb{Z}.
\end{align*}
This implies that 
\begin{align*}
{^\vee}T^{s_{2}} &= \{ e((2c_{2} + \ell) \alpha_{1} + c_{2}
\alpha_{2}) : c_{2} \in \mathbb{C}, \, \ell \in \mathbb{Z}\}\\
& = \{ e(c_{2}(2 \alpha_{1} + 
\alpha_{2})) : c_{2} \in \mathbb{C}\}.
\end{align*}
The group ${^\vee}T^{s_{2}}$ is the image of the connected
space $\mathbb{C}$, and so is itself connected.  We conclude that the Langlands
component group of $\varphi_{4}$ is trivial.

The remaining cases are $\varphi_{5}, \varphi_{6}, \varphi_{7}$ and
$\varphi_{8}$.  In each 
of these cases $p(S_{j}) \in W$ is a reflection.  One may argue as we did
for ${^\vee}T^{s_{2}}$ to prove that ${^\vee}T^{p(S_{j})}$ is
connected.  In summary, the only non-trivial Langlands component group
is the Langlands component group of $\varphi_{9}$, which is a Klein
4-group. 

\subsection{Orbits on flag varieties for regular non-integral infinitesimal
  character}\label{section:FlagNonIntegral} 
 
In this section we consider L-parameters $\varphi = 
\varphi(y,\lambda)$ as in (\ref{lparam}), but remove the assumption
of integrality on $\lambda \in {^\vee}\mathfrak{t}$.  We merely assume
that $\lambda$ is regular \ref{reg}.  In this more general setting,
\cite[Section 6]{ABV} tells us that  the
equivalence class of $\varphi$ corresponds to a certain orbit of the 
the flag variety ${^\vee}\mathrm{G}_{2}(\lambda)/ \,{^\vee}B(\lambda)$. Here,
${^\vee}\mathrm{G}_{2}(\lambda)$ is the centralizer of $e(\lambda)$ in
$\mathrm{G}_{2}$, and  ${^\vee}B(\lambda)$ is the unique Borel subgroup of
${^\vee}\mathrm{G}_{2}(\lambda)$ determined by the regular element
$\lambda$.  Our first task is to determine all of the possible groups
${^\vee}\mathrm{G}_{2}(\lambda)$.  Once this is complete, we will
find the analogues of ${^\vee}K$ for each
${^\vee}\mathrm{G}_{2}(\lambda)$ and describe their orbits on
${^\vee}\mathrm{G}_{2}(\lambda)/ \,{^\vee}B(\lambda)$.  

\begin{prop}
\label{glambda}
The group ${^\vee}\mathrm{G}_{2}(\lambda)$ is isomorphic to one of the following
complex algebraic  groups:
$${^\vee}T, \ \mathrm{GL}_{2}, \
\mathrm{SL}_{2} \times_{\upmu_{2}} \mathrm{SL}_{2}, \ \mathrm{SL}_{3}, \
\mathrm{G}_{2}.$$
\end{prop}
\begin{proof}
 By \cite[Theorem 2.2 and Theorem
  2.11]{HumphCC}  the centralizer ${^\vee}\mathrm{G}_{2}(\lambda)$ of
 $e(\lambda)$ 
is connected, reductive and equal to
$$\langle {^\vee}T, U_{\pm {^\vee}\alpha} : {^\vee}\alpha(e(\lambda))
= 1 \rangle = \langle {^\vee}T, U_{\pm {^\vee}\alpha} : \langle
{^\vee}\alpha, \lambda \rangle
\in \mathbb{Z} \rangle.$$
The root system of ${^\vee}\mathrm{G}_{2}(\lambda)$ is 
$$R({^\vee}\mathrm{G}_{2}(\lambda), {^\vee}T) = \{ {^\vee}\alpha \in
R({^\vee}\mathrm{G}_{2}, {^\vee}T) : \langle {^\vee}\alpha, \lambda \rangle \in
\mathbb{Z} \}.$$
As a subsystem of $R({^\vee}\mathrm{G}_{2}, {^\vee}T)$, there are only three
nonempty possibilities for $R({^\vee}\mathrm{G}_{2}(\lambda),
{^\vee}T)$, namely those of type $\mathrm{A}_{1}$, 
$\mathrm{A}_{1} \times \mathrm{A}_{1}$, $\mathrm{A}_{2}$ and $\mathrm{G}_{2}$.
Writing $\lambda = c_{1} \alpha_{1} + c_{2} \alpha_{2}$ for
$c_{1}, c_{2} \in \mathbb{C}$, we see that a positive root ${^\vee}\alpha \in
R({^\vee}\mathrm{G}_{2},{^\vee}T)$ belongs to
$R({^\vee}\mathrm{G}_{2}(\lambda),{^\vee}T)$ if and only if 
$$c_{1} \langle \, {^\vee}\alpha, \alpha_{1} \rangle + c_{2} \langle
\, {^\vee}\alpha, 
\alpha_{2} \rangle =m \in \mathbb{Z}.$$
This complex equation is equivalent to  the two families of real lines
\begin{align*} 
&  a_{1} \langle \, {^\vee}\alpha, \alpha_{1}
  \rangle + a_{2} \langle \, {^\vee}\alpha, \alpha_{2} \rangle = m \in
  \mathbb{Z},\\
&  b_{1} \langle \, {^\vee}\alpha, \alpha_{1}
  \rangle + b_{2} \langle \, {^\vee}\alpha, \alpha_{2} \rangle = 0,
\end{align*}
where $c_{1} = a_{1} + i b_{1}$ and $c_{2} = a_{2}+ib_{2}$ are in standard
form.  One may always take $b_{1} =
b_{2}= 0$ to obtain a point on the second line. We therefore focus on
the first family of lines.  As $m$ runs 
over all integers and ${^\vee}\alpha$ runs over all positive roots the
first equation tessellates the real plane.  By choosing $(a_{1},
a_{2}) \in \mathbb{R}^{2}$ so that it does not lie on any of the lines we obtain
$\lambda$ such that $R({^\vee}\mathrm{G}_{2}(\lambda),{^\vee}T) =
\emptyset $ and ${^\vee}\mathrm{G}_{2}(\lambda) =  {^\vee}T \cong
\mathrm{GL}_{1} \times \mathrm{GL}_{1}$.
A nonempty root system is given by choosing $(a_{1}, a_{2})$ to lie on a
unique line for a fixed $m \in \mathbb{Z}$ and positive root
${^\vee}\alpha$.  In this case
$R({^\vee}\mathrm{G}_{2}(\lambda),{^\vee}T) = \{ \pm {^\vee}\alpha \}$
and  ${^\vee}\mathrm{G}_{2}(\lambda) =
\langle {^\vee}T, U_{\pm {^\vee}\alpha} \rangle$.  Setting
${^\vee}\alpha_{\perp}$ to be the positive root orthogonal to
${^\vee}\alpha$, we see that $\alpha, \alpha_{\perp} \in
X_{*}({^\vee}T)$ generate ${^\vee}T$.  Consequently,  
$${^\vee}\mathrm{G}_{2}(\lambda) =
\langle {^\vee}T, U_{\pm {^\vee}\alpha} \rangle = \langle
\alpha_{\perp}(\mathbb{C}^{\times}) \rangle \ \langle
\alpha(\mathbb{C}^{\times}), U_{\pm {^\vee}\alpha} \rangle,$$
a product of commuting groups.
Computations similar to those made for ${^\vee}K$ at the end of
Section \ref{groupK} reveal that $\langle
\alpha_{\perp}(\mathbb{C}^{\times}) \rangle \cap \langle
\alpha(\mathbb{C}^{\times}), U_{\pm {^\vee}\alpha} \rangle$ equals
$\langle \alpha(-1) \rangle \cong \upmu_{2}$.  Hence,
$${^\vee}\mathrm{G}_{2}(\lambda) \cong \mathrm{GL}_{1}
\times_{\upmu_{2}} \mathrm{SL}_{2} \cong \mathrm{GL}_{2}. $$

Finally, distinct positive roots
${^\vee}\alpha, {^\vee}\beta \in R(^{\vee}\mathrm{G}_{2}, {^\vee}T)$
produce non-parallel lines and so we may 
choose $(a_{1}, a_{2})$ so that
\begin{align*}
& a_{1} \langle \, {^\vee}\alpha, \alpha_{1}
  \rangle + a_{2} \langle \, {^\vee}\alpha, \alpha_{2} \rangle = m\\
& a_{1} \langle \, {^\vee}\beta, \alpha_{1}
  \rangle + a_{2} \langle \, {^\vee}\beta, \alpha_{2} \rangle = \ell
\end{align*}
for some $\ell, m \in \mathbb{Z}$.  This implies that ${^\vee}\alpha,
{^\vee}\beta \in R({^\vee}\mathrm{G}_{2}(\lambda), {^\vee}T)$.
Clearly, 
$$ a_{1} \langle \, {^\vee}\alpha \pm {^\vee}\beta, \alpha_{1}
  \rangle + a_{2} \langle \, {^\vee}\alpha \pm ^{\vee}\beta, \alpha_{2}
  \rangle = m \pm \ell \in \mathbb{Z},$$ 
which implies that the root subsystem of $R(^{\vee}\mathrm{G}_{2}, {^\vee}T)$
generated by ${^\vee}\alpha$ and 
${^\vee}\beta$ lies in $R({^\vee}\mathrm{G}_{2}(\lambda), {^\vee}T)$
\cite[9.4]{Humphreys}.   
The resulting root subsystem is of type $\mathrm{A}_{1}
\times \mathrm{A}_{1}, \mathrm{A}_{2}$ or $\mathrm{G}_{2}$.  As explained
at the end of  Section \ref{groupK},  if the root
subsystem is of type $\mathrm{A}_{1}
\times \mathrm{A}_{1}$, the group ${^\vee}\mathrm{G}_{2}(\lambda)$ is isomorphic
to ${^\vee}K \cong \mathrm{SL}_{2} \times_{\upmu_{2}} \mathrm{SL}_{2}$.  If the root
subsystem is of type $\mathrm{A}_{2}$ then
$${^\vee}\mathrm{G}_{2}(\lambda) = \langle {^\vee}T, {^\vee} U_{\pm
  {^\vee}\alpha_{2}} , U_{\pm (3{^\vee}\alpha_{1}+
  {^\vee}\alpha_{2})} \rangle  = \langle {^\vee} U_{\pm
  {^\vee}\alpha_{2}} , U_{\pm (3{^\vee}\alpha_{1}+
  {^\vee}\alpha_{2})} \rangle \cong \mathrm{SL}_{3}.$$
% ${^\vee}T$ is a two-dimensional torus generated by the tori of the
% two SL_{2}-subgroups
\end{proof}

Having identified the groups ${^\vee}\mathrm{G}_{2}(\lambda)$, we may
determine the analogues of the group ${^\vee}K$ for 
${^\vee}\mathrm{G}_{2}(\lambda)$.  We denote these analogues by
${^\vee}K(\lambda)$.  Each ${^\vee}K(\lambda)$ is defined as the
centralizer ${^\vee}\mathrm{G}_{2}(\lambda)^{y}$ of an
element $y \in {^\vee}\mathrm{G}_{2}(\lambda)$ satisfying
$y^{2} = e(\lambda)$ \cite[Propositions 6.13 and 6.16]{ABV}.
Furthermore, the
elements $y$  are only of interest up to conjugation by
${^\vee}\mathrm{G}_{2}(\lambda)$, and there are finitely many such
conjugacy classes.  To put it more clearly, let
$$\mathcal{I}(\lambda) = \{y \in  {^\vee}\mathrm{G}_{2}(\lambda) :
y^{2} = e(\lambda)\}.$$
% This matches (6.10)(f) ABV by the proof of Proposition 6.13 ABV
The group ${^\vee}\mathrm{G}_{2}(\lambda)$ acts by conjugation on
$\mathcal{I}(\lambda)$ with finitely many orbits, and the relevant groups
${^\vee}K(\lambda)$ are given by taking the centralizer of a representative
for each orbit.  

\begin{prop}
\label{Ilambdaorbits}
The possibilities for ${^\vee}K(\lambda)$, up to
isomorphism, are given in the following
table
\begin{center}
\begin{tabular}{|l|l|}
\hline
${^\vee}\mathrm{G}_{2}(\lambda)$ & ${^\vee}K(\lambda)$ \\
\hline \hline
${^\vee}T$ & ${^\vee}T$  \\ 
\hline 
$\mathrm{GL}_{2}$ & ${^\vee}T$, $\mathrm{GL}_{2}$ \\
\hline
$\mathrm{SL}_{2} \times_{\upmu_{2}} \mathrm{SL}_{2}$ & $\mathrm{GL}_{2}$ \\
\hline
$\mathrm{SL}_{3}$ & $\mathrm{GL}_{2}$, $\mathrm{SL}_{3}$ \\
\hline
${^\vee}\mathrm{G}_{2}$ & $\mathrm{SL}_{2} \times_{\upmu_{2}} \mathrm{SL}_{2}$,
${^\vee}\mathrm{G}_{2}$ \\
\hline
\end{tabular}
\end{center}
In the row for $\mathrm{SL}_{2} \times_{\upmu_{2}} \mathrm{SL}_{2}$,
the group $\mathrm{GL}_{2}$ is isomorphic to
$\mathrm{SL}_{2} \times_{\upmu_{2}} {^\vee}T$. In the row for $\mathrm{SL}_{3}$,
the group $\mathrm{GL}_{2}$ is identified with a block diagonal
subgroup of $\mathrm{SL}_{3}$.  In all cases, ${^\vee}T \subset
{^\vee}K(\lambda)$. 
\end{prop}
\begin{proof}
The final row of the table was established in Lemma \ref{Iorbits}.  We
therefore assume ${^\vee}\mathrm{G}_{2}(\lambda) \ncong
\mathrm{G}_{2}$ and offer a sketch of the elementary computations
involved in the remaining cases.
Let $W(\lambda)$ be the Weyl group of
$({^\vee}\mathrm{G}_{2}(\lambda), {^\vee}T)$.  As in the proof of
Lemma \ref{Iorbits}, we have a natural bijection between 
the ${^\vee}\mathrm{G}_{2}(\lambda)$-orbits of $\mathcal{I}(\lambda)$
and the $W(\lambda)$-orbits of
\begin{equation}
\label{eset}
\{y \in {^\vee}T : y^{2} = e(\lambda)\}.
\end{equation}
The groups ${^\vee}K(\lambda)$ we are seeking are the centralizers of
representatives of the $W(\lambda)$-orbits.  As we are determining
${^\vee}K(\lambda)$ only up to isomorphism, we may 
ignore the $W(\lambda)$-action.
Setting $y' = y \,e(-\lambda/2)$, we may
identify (\ref{eset}) with
\begin{equation}
\label{1set}
\{y' \in {^\vee}T : (y')^{2} = 1\}.
\end{equation}
For each instance of
${^\vee}\mathrm{G}_{2}(\lambda)$ it is easy to show that the elements
that square to the identity may be taken to be diagonal with $\pm 1$
as diagonal entries.  This determines the set (\ref{1set}), and
therefore the set (\ref{eset}), on a case-by-case basis.  The
centralizers of the resulting elements are ascertained using
\cite[Theorem 2.2 and Theorem 
  2.11]{HumphCC}.  We omit the computations. 
% For SL_2 \times SL_2, \lambda/2=\rho/4.  Since e(\rho/4) is
% semisimple the root space of a positive root
% \beta=a\alpha_1+b\alpha_2 belongs to the centralizer if and only if
% a<\alpha_1^\vee, \rho/4>+b<alpha_2^\vee, \rho/4> is an integer if
% and only if a<\alpha_1^\vee, \rho/4>+b<alpha_2^\vee, \rho/4> is
% divisible by 4 if and only if a=1 and b=3
%
% For SL_3, \lambda= (3/2)\alpha_2+(2/3)\alpha_1. The centralizer of
% e(\lambda/2) contains only the root subgroups for
% \alpha_{1}+3\alpha_2, which gives us a GL_2 block.  The centralizer
% of e(\lambda/2+ (\alpha_{1} +3\alpha_{2})/2) contains the root
% subgroups of all long roots and so is equal to all of SL_3
\end{proof}

Our last objective is to describe the ${^\vee}K(\lambda)$-orbits of
the  flag variety
${^\vee}\mathrm{G}_{2}(\lambda)/ {^\vee}B(\lambda)$ for each row in
the table of Proposition \ref{Ilambdaorbits}.   We disregard
all cases in which ${^\vee}K(\lambda) =
{^\vee}\mathrm{G}_{2}(\lambda)$, as there is only a single orbit in
these cases.   In view of this and Section
\ref{groupK}, we need only consider rows 2-4 of the table.

 If ${^\vee}\mathrm{G}_{2}(\lambda) \cong \mathrm{GL}_{2}$ then
${^\vee}\mathrm{G}_{2}(\lambda)/ {^\vee}B(\lambda)$ is isomorphic to
  $\mathrm{GL}_{2}/B$ where $B \subset \mathrm{GL}_{2}$ is the
  upper-triangular subgroup. The flag variety $\mathrm{GL}_{2}/B$ is
  isomorphic to the complex projective line $\mathbb{P}^{1}$ via the
  map that sends $gB$ to the line spanned by $g \tiny \begin{bmatrix} 1
    \\ 0 \end{bmatrix} \normalsize$.  If
  ${^\vee}K(\lambda)$ is isomorphic to the diagonal subgroup $^{\vee}T
  \subset \mathrm{GL}_{2}$ then there are three orbits:  two orbits
  correspond to the points $0, \infty \in \mathbb{P}^{1}$ and the
  remaining orbit corresponds to $\mathbb{P}^{1} - \{ 0, \infty \}$.

 If ${^\vee}\mathrm{G}_{2}(\lambda) \cong \mathrm{SL}_{2}
  \times_{\upmu_{2}} \mathrm{SL}_{2}$, the flag variety is isomorphic to
  $$(\mathrm{GL}_{2} \times \mathrm{GL}_{2})/ (B \times B) \cong
  \mathrm{GL}_{2}/B \times  \mathrm{GL}_{2}/B$$ 
as central  
  subgroups of reductive groups are contained in Borel subgroups.
 The ${^\vee}K(\lambda)$-orbits
  may be recovered as in the previous paragraph.  There are three orbits
  for ${^\vee}K(\lambda) \cong  \mathrm{GL}_{2} \cong \mathrm{SL}_{2}
  \times_{\upmu_{2}} {^\vee}T$.  

If ${^\vee}\mathrm{G}_{2}(\lambda) \cong \mathrm{SL}_{3}$ then we
may argue as in the previous paragraph to take the flag variety equal
to $\mathrm{GL}_{3}/B$, where $B \subset \mathrm{GL}_{3}$ is the
upper-triangular subgroup.  We may replace the block diagonal subgroup
of $\mathrm{SL}_{3}$ which is isomorphic to $\mathrm{GL}_{2}$, with
the block diagonal subgroup of $\mathrm{GL}_{3}$ which is isomorphic
to $\mathrm{GL}_{2} \times \mathrm{GL}_{1}$.  With these superficial
changes, the ${^\vee}K(\lambda)$-orbits of the flag variety may be
identified with the
$\mathrm{GL}_{2} \times \mathrm{GL}_{1}$-orbits of
$\mathrm{GL}_{3}/B$.  The orbits in this example, together with their
Bruhat order, are a special case of \cite[Section 2.4]{Yamamoto} (see
\cite[Figure 1]{Yamamoto}).  Alternatively, we may argue using
\cite{Adams-Fokko}, as we did in Section \ref{Korbits}.  The Hasse
diagram for the Bruhat order is
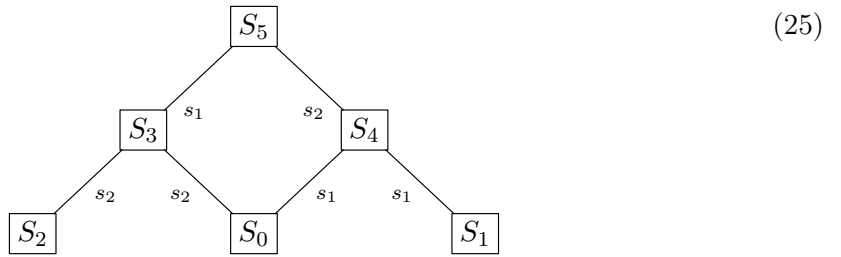
\begin{equation}
\label{sl3hasse}
\xymatrix{
& & *+[F]{S_{5}} \ar@{-}[dr]_(.7){s_{2}}  \ar@{-}[dl]^(.7){s_{1}}& &\\
& *+[F]{S_{3}}  \ar@{-}[dl]^{s_{2}} \ar@{-}[dr]_{s_{2}} & &  *+[F]{S_{4}}
  \ar@{-}[dl]^{s_{1}} 
  \ar@{-}[dr]_{s_{1}} & \\
 *+[F]{S_{2}} & &  *+[F]{S_{0}} &  & *+[F]{S_{1}}}
\end{equation}
In this diagram for ${^\vee}\mathrm{G}_{2}(\lambda) \cong
\mathrm{SL}_{3}$ we have again labelled the orbits with $S_{j}$ 
and the simple reflections with $s_{j}$, as we did for the Hasse
diagram for ${^\vee}\mathrm{G}_{2}$.  This abuse of notation should
not cause any confusion as the orbits for the different groups will be
treated in different sections.

We conclude by making some general observations about the orbits
introduced in this section. The reason the groups
${^\vee}\mathrm{G}_{2}(\lambda)$ are of interest 
is that the 
${^\vee}K(\lambda)$-orbits on the flag variety of ${^\vee}\mathrm{G}_{2}(\lambda)$
are in bijection with the equivalence classes of L-parameters for
${^\vee}\mathrm{G}_{2}$ with infinitesimal character $\lambda$
\cite[Proposition 6.16   and Proposition 6.17]{ABV}.   The Langlands
component group of a particular L-parameter is the component group of
a centralizer in ${^\vee}K(\lambda)$ of the semisimple element
$\lambda \in {^\vee}\mathfrak{t}$ \cite[Corollary 5.9 (c)]{ABV}.  The case of
${^\vee}K(\lambda) \cong \mathrm{SL}_{2} \times_{\upmu_{2}}
\mathrm{SL}_{2}$ was treated in Section \ref{compgroupsec}.  In all
remaining cases, the derived subgroup of ${^\vee}K(\lambda)$ is simply
connected.  By a theorem of Steinberg, the centralizer of $\lambda$ is
connected and the Langlands component group is trivial.

\section{Complete geometric parameters and the moment map}
\label{sec:momap}

In this section we return to the general framework of (\ref{landscape})
and (\ref{landscape1}), establishing some more notation and fleshing
out some details concerning the moment map.  We conclude by computing
the characteristic cycle (\ref{cc}) of a particularly simple perverse
sheaf on $X$.

Define the set of complete geometric parameters 
\begin{equation}
\label{geoparam}
\Xi\left(X, {^\vee}K\right) = \{ (S, \mathcal{V}) : S \in {^\vee}K
\backslash X \}
\end{equation}
 to be the set of pairs $\xi = (S,\mathcal{V})$
with $S$ a ${^\vee}K$-orbit on $X$ and $\mathcal{V}$ an irreducible
${^\vee}K$-equivariant 
local system on $S$. We are identifying the sets
$$\left\{P(\xi):\xi\in \Xi\left(X, {^\vee}K\right)\right\},\quad
\left\{D(\xi):\xi\in \Xi\left(X, {^\vee}K\right)\right\}$$
through the Riemann-Hilbert correspondence.  These sets  
are bases of the Grothendieck group $\mathscr{K}(X,{^\vee}K)$.  

As we saw in Section \ref{Korbits}, there are ten
${}^{\vee}K$-orbits on $X$, 
which we continue to denote somewhat abusively by $S_{i}$, $0\leq i\leq 9$.  
According to the description of the Langlands component groups
given in Section \ref{compgroupsec} and \cite[Lemma 7.3 (e)]{ABV}, the
trivial local system 
$\underline{\C}_{S_i}$ is the unique irreducible local system supported in orbits
$S_{i}$, $0 \leq i \leq 8$. For orbit $S_{9}$ the Langlands component
group is isomorphic to $\mathbb{Z}/2\mathbb{Z} \times
\mathbb{Z}/2\mathbb{Z}$.  The trivial character of the Langlands
component group corresponds to the
constant sheaf $\underline{\C}_{S_9}$.  There are three remaining
irreducible non-constant
local systems that we denote by $\mathcal{L}_i,\, i=10,11,12$.  
Therefore, $\Xi(X,{}^{\vee}K)$ consists of 13
complete geometric parameters
\begin{equation}\label{eq:IrrObj1}
\Xi(X,{}^{\vee}K)\ =\ \{\xi_i=(S_i,\underline{\C}_{S_i}):0\leq
i\leq 9\}\cup\{\xi_i=(S_9,\mathcal{L}_{i}):i=10,11,12\} 
\end{equation}
and $\mathscr{K}(X,{}^{\vee}K)$ has a basis given by
$\{P(\xi_i):0\leq i\leq 12\}$. 

We now specialize the general framework (\ref{landscape}) to
the setting of Section \ref{section:FlagNonIntegral}. 
Let $\lambda \in {^\vee}\mathfrak{t}$ be a non-integral regular element. 
Denote the flag variety by
$X(\lambda)={}^{\vee}\mathrm{G}_{2}(\lambda)/\,{}^{\vee}B(\lambda)$ and
take the group ${}^{\vee}K(\lambda)$ to be one of those described in
Proposition~\ref{Ilambdaorbits}. 
Under these conditions the situation is slightly simpler. For
any of the pairs $({}^{\vee}\mathrm{G}_{2}(\lambda),
{}^{\vee}K(\lambda))$ described in Proposition~\ref{Ilambdaorbits},
and any  $^{\vee}K(\lambda)$-orbit $S$ on $X(\lambda)$ the only
irreducible local system is the trivial local system
$\underline{\mathbb{C}}_S$. Consequently, 
the Grothendieck group $\mathscr{K}(X(\lambda),{}^{\vee}K(\lambda))$
has a basis given by  
\begin{equation}\label{eq:IrrObj2}
\{P(\xi_i):0\leq i\leq n\},
\end{equation} 
where $n$ is the number of ${}^{\vee}K(\lambda)$-orbits on
$X(\lambda)$, and $\xi_i = (S_i, \underline{\mathbb{C}}_{S_i})$. 

Our goal is to compute the characteristic cycles (\ref{cc}) of the
basis elements (\ref{eq:IrrObj1}) and (\ref{eq:IrrObj2}).  Some
well-known properties of these characteristic cycles which narrow the
computations are given in the following lemma.
\begin{lem}[Lemma 19.14 \cite{ABV}]\label{lem:lemorbit}
Let $S,\, S'$ be a pair of ${^\vee}K(\lambda)$-orbits in
$X(\lambda)$. Fix an irreducible local system $\mathcal{V}$  
on $S$ and write $\xi=(S,\mathcal{V})$. Then
$\chi_{S}^{\mathrm{mic}}(P(\xi))$ is equal to $1$ (the
rank of the local system $\mathcal{V}$). Moreover, if
$\chi_{S'}^{\mathrm{mic}}(P(\xi))\neq 0$, then $$S'\subset
\overline{S}.$$  
\end{lem}

For the remainder of this section we survey the moment map and use it
to help compute $CC(P(\xi_{9}))$  from (\ref{eq:IrrObj1}).
To simplify the 
notation we work in the context of an arbitrary connected complex
reductive group $G$ and Borel subgroup $B$.  We will take $X = G/B$
and the symmetric subgroup $K \subset G$ to be the fixed-point
subgroup of an involutive automorphism of $G$.  Once we have completed
our survey, we will specialize to $G = 
{^\vee}\mathrm{G}_{2}$, \emph{etc.}~as in Section \ref{Korbits}.  

At each point $x \in X$, 
the differential of the action of $\mathrm{G}$ on $X$, defines a map
$$
a_x:\mathfrak{g} \longrightarrow T_x X.
$$
Since $X$ is a homogeneous space, the map $a_{x}$ extends to a
surjective morphism 
%\label{eq:maptoTX}
$$a:\mathfrak{g}\times X \longrightarrow TX.$$
The dual of this map
$$
a^{\ast}:T^\ast X \longrightarrow \mathfrak{g}^{\ast}\times X, 
$$
is a closed immersion. %, and hence projective. 
Composing $a^{\ast}$ with the projection onto the first factor yields
the {\it{moment map}} 
$$
\mu:T^\ast X \longrightarrow \mathfrak{g}^{\ast}.
$$
We have an alternative description of the cotangent bundle $T^\ast X$,
and therefore have an alternative definition of the moment map.
Indeed, let $\mathcal{N}(\mathfrak{g}^\ast)$ 
be the nilpotent cone in $\mathfrak{g}^{\ast}$.
%consisting of all nilpotent elements. 
We may identify
$X= G/ B$ with the variety of Borel subalgebras of
$\mathfrak{b}' \subset\mathfrak{g}$ under the map $gB \mapsto
\mathrm{Ad}(g)\mathfrak{b}$. 
As shown in \cite[{Proposition 1.4.11 and Lemma
    3.2.2}]{Chriss-Ginzburg}, we have isomorphisms    
\begin{align}
\nonumber T^{\ast}X &\cong  \{(\mathfrak{b}',\nu):\mathfrak{b}'\in X,\, \nu\in
(\mathfrak{g}/\mathfrak{b}')^{\ast}\}\\ 
\label{cotiso}&\cong \{( \mathfrak{b}' , \nu )\in X\times
\mathcal{N}(\mathfrak{g}^\ast):\nu\in(\mathfrak{b}')^{\ast} \}. 
\end{align}
Under this identification the moment map 
$\mu$ is simply the projection 
$$( \mathfrak{b}' , \nu ) \quad \longmapsto \quad \nu,$$  
from $T^{\ast}X$ to $\mathcal{N}(\mathfrak{g}^\ast)$ \cite[{Lemma
    3.2.3}]{Chriss-Ginzburg}. 

Let $\mathfrak{k}$ be the Lie algebra of the symmetric subgroup $K$. 
%Then as the morphism in \eqref{eq:maptoTX} defined through the
%differential of the action of $G$ on $X$, the differential of the
%action of $K$ on $X$ induces a morphism $a_K:\mathfrak{k}\times X
%\longrightarrow TX$.  
The conormal bundle to the $K$-action on $X$ is defined as 
%\label{eq:conormalbundle}
$$T^{\ast}_{K}X\, =\, \left\{ (\mathfrak{b}',\lambda):\, \lambda \in
T_{\mathfrak{b}'}^{\ast}X,\ \lambda(a_{\mathfrak{b}'}(\mathfrak{k})) =
0\right\}.$$
For any $K$-orbit $S$ in $X$, the fibre of $T^{\ast}_{K}X$ at
$y = \mathfrak{b}' \in S$ is $T_{S,y}^{*}X$, the conormal bundle to $S$ at  
$y$, or equivalently,  the annihilator of $T_{y}S$ in
$T^{\ast}_{y}X$. 
It follows that
$$
T^{\ast}_{K}X\, =\, \bigcup_{S\in K\setminus X} T_{S}^{\ast}X,
$$
where $T_{S}^{*}X = \cup_{y \in S} T_{S,y}^{*}X$.
Under the identification (\ref{cotiso})
\begin{equation}\label{eq:TSX}
T_S^{\ast}X\, =\, \{
(\mathfrak{b}',\nu):\mathfrak{b}'\in S,\, \nu\in
(\mathfrak{g}/(\mathfrak{k}+\mathfrak{b}'))^{\ast} \}.
\end{equation}
The image of 
$T^{\ast}_{K}X$
under the moment map $\mu$ is $\mathcal{N}_{X,K}(\mathfrak{g}^\ast)$,
%and denote by $\mathcal{N}_{X,K}(\mathfrak{g}^\ast)$  
the cone of nilpotent elements in the intersection
$$
\left(G\cdot
(\mathfrak{g}/\mathfrak{b})^{\ast}\right)\cap(\mathfrak{g}/\mathfrak{k})^{\ast}. 
$$
By \cite{Kostant-Rallis} (see also \cite[Lemma 20.14]{ABV}), the group
$K$ has finitely many orbits on  
$\mathcal{N}_{X,K}(\mathfrak{g}^\ast)$, and for any $\nu\in
\mathcal{N}_{X,K}(\mathfrak{g}^\ast)$ 
$$\dim \left(K\cdot \nu\right)\, =\, \frac{1}{2}\dim \left(G\cdot
\nu\right).$$ 
Additionally, by \cite[{Lemma 20.16}]{ABV}, the image
$\mu\left(\overline{T^{\ast}_S X}\right)$ is the closure of a single
orbit of $K$ on  $\mathcal{N}_{X,K}(\mathfrak{g}^\ast)$.
When $X$ is the flag variety of ${}^\vee \mathrm{G}_{2}$, we will provide
a detailed description of the conormal bundles in $\mu^{-1}(\mathcal{O})$,  
for each nilpotent ${^\vee}K$-orbit $\mathcal{O}$ on
$\mathcal{N}_{X, {^\vee}K}(\mathfrak{g}_{2}^\ast)$ (Proposition
\ref{prop:momentmapimage}).  
Understanding the image of $\mu$ will play a central role in our
computations of characteristic cycles.

Another object that will play a key role in the computation of these
characteristic cycles is the associated variety $\mathrm{AV}(P) =
\mu(\mathrm{Ch}(P)))$  of an
irreducible perverse sheaf $P$ (\ref{eq:AVdef}). (We could have
defined $\mathrm{AV}(P)$ independently of $\mathrm{Ch}(P)$, as in 
\cite{Vogan-Avariety}. That the two definitions agree follows from
\cite[Theorem 1.9(c)]{BoBrIII}). 
The set $\mathrm{AV}(P)$ is a closed 
$K$-invariant subvariety of $\mathcal{N}_{X,K}(\mathfrak{g}^\ast)$.
Since there are finitely many $K$-orbits on
$\mathcal{N}_{X,K}(\mathfrak{g}^\ast)$, we may write  $\mathrm{AV}(P)$
as a finite union of closures of $K$-orbits of equal dimension
\begin{equation}\label{eq:AV(P)}
\mathrm{AV}(P)\, =\, \overline{\mathcal{O}}_1 \cup \cdots \cup
\overline{\mathcal{O}}_n. 
\end{equation}
By \cite[Theorem 8.4]{Vogan-Avariety}, 
there exists a nilpotent $G$-orbit $\mathcal{O}_{G} \subset
\mathcal{N}(\mathfrak{g}^\ast)$, such that 
for each nilpotent $K$-orbit $\mathcal{O}_i$ appearing in 
(\ref{eq:AV(P)}), we have 
$$G\cdot \mathcal{O}_i=\mathcal{O}_{G}.$$
Let us say a few words about the \( G \)-orbit \( \mathcal{O}_{G} \).
Write $\pi$ for the irreducible $(\mathfrak{g}, K)$-module
corresponding to  $P$
under the Riemann-Hilbert and Beilinson-Bernstein
correspondences (\ref{landscape}).   Denote by $\mathrm{Ann}(\pi)$ the
annihilator of  
$\pi$ in  the universal enveloping algebra $U(\mathfrak{g})$.
By definition, such an  ideal is called \emph{primitive}. 
Associated to $\mathrm{Ann}(\pi)$ is a subvariety
$\mathrm{AV}(\mathrm{Ann}(\pi))$ of 
$\mathcal{N}(\mathfrak{g}^\ast)$, commonly referred to as the
associated variety of $\mathrm{Ann}(\pi)$ \cite[(1.4)]{Vogan-Avariety}.  By
results of Joseph and 
Borho-Brylinski (see \cite[{Theorem 10.2.2}]{Collingwood-McGovern}),
the variety $\mathrm{AV}(\mathrm{Ann}(\pi))$ is the closure
$\overline{\mathcal{O}}_{G}$ of %a single nilpotent $G$-orbit  
$\mathcal{O}_{G}$ in $\mathcal{N}(\mathfrak{g}^\ast)$. Moreover, by the
work of Joseph and Barbasch-Vogan (see, for example, \cite[{Theorem
    10.3.2}]{Collingwood-McGovern}), the orbit $\mathcal{O}_{G}$ is
\emph{special} in the sense of Lusztig--that is, $\mathcal{O}_{G}$ lies in
the image of the Lusztig-Spaltenstein map \cite{Spaltenstein}.  
%We point out that every special nilpotent orbit arises in this way, as
%the associated variety of some primitive ideal
Simply put, the 
associated variety \( \mathrm{AV}(P) \) decomposes as a union of
closures of nilpotent \( K \)-orbits \( \mathcal{O}_{i} \) whose \( G
\)-saturations \( G \cdot \mathcal{O}_{i} \) are special nilpotent orbits.

We come back to the setting of $G = {^\vee}\mathrm{G}_{2}$ and  give
a partial description of the associated varieties of irreducible
perverse sheaves on $X$. The full 
description will be given 
in Proposition \ref{prop:Cells-AV} in the next section.
To this end, we review the work of {\DJ}okovi\'c \cite{Dokovic}, 
which classifies the nilpotent orbits for groups of type $\mathrm{G}_{2}$. 
Let \( {}^{\vee}K \) be as in Section~\ref{groupK}. 
There are six nilpotent \( {}^{\vee}K \)-orbits: one of dimension six,
denoted $\mathcal{O}_6$; two of dimension five, $\mathcal{O}_{5,1}$
and $\mathcal{O}_{5,2}$; one of dimension four, $\mathcal{O}_4$; one
of dimension three, 
$\mathcal{O}_3$; and finally, the zero-dimensional orbit $\mathcal{O}_0$. Among the two five-dimensional nilpotent orbits, the component group of the centralizer of any $N \in \mathcal{O}_{5,1}$ in ${}^{\vee}K$ is $S_3$, while the component group of the centralizer of any $N \in \mathcal{O}_{5,2}$ in ${}^{\vee}K$ is $\mathbb{Z}_2$.
 According to \cite[Theorem 2.1]{Dokovic}, the closure ordering
on these nilpotent ${}^{\vee}K$-orbits is represented by the following
Hasse diagrams 
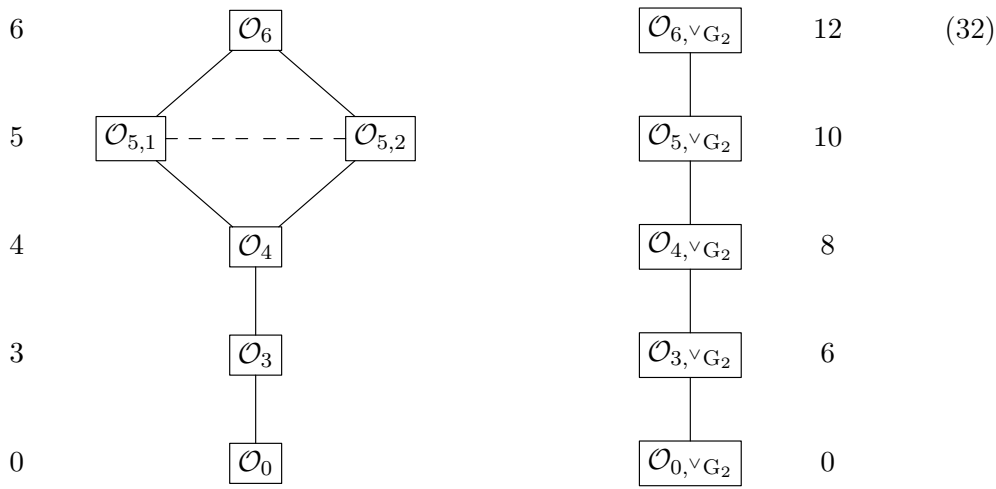
\begin{equation}
\label{eq:nilpotent-orbits}
\xymatrix{
6& & *+[F]{\mathcal{O}_6} \ar@{-}[dr]\ar@{-}[dl]& & &
&*+[F]{\mathcal{O}_{6,{^\vee}\mathrm{G}_{2}}}\ar@{-}[d] &12 \\ 
5& *+[F]{\mathcal{O}_{5,1}}\ar@{--}[r]  \ar@{-}[dr] & \ar@{--}[r]&
*+[F]{\mathcal{O}_{5,2}} 
  \ar@{-}[dl] & & &*+[F]{\mathcal{O}_{5,{^\vee}\mathrm{G}_{2}}}\ar@{-}[d] & 10\\
4  & &  *+[F]{\mathcal{O}_4}\ar@{-}[d] &  & & &
*+[F]{\mathcal{O}_{4,{^\vee}\mathrm{G}_{2}}}\ar@{-}[d] &8  \\ 
3  & &  *+[F]{\mathcal{O}_3}\ar@{-}[d] &  & & &
*+[F]{\mathcal{O}_{3,{^\vee}\mathrm{G}_{2}}}\ar@{-}[d] &6\\ 
0  & &  *+[F]{\mathcal{O}_0} & & & &  *+[F]{\mathcal{O}_{0,{^\vee}\mathrm{G}_{2}}}
&0  }
\end{equation}
In the diagram on the right, we display the closure ordering among the
${}^{\vee}\mathrm{G}_{2}$-saturations of the nilpotent
${}^{\vee}K$-orbits. The ${}^{\vee}K$-orbits $\mathcal{O}_{5,1}$ and
$\mathcal{O}_{5,2}$ lie in the same ${}^{\vee}\mathrm{G}_{2}$-orbit,
as indicated by the horizontal dotted line in the diagram on the
left. The dimension of each orbit is shown alongside it. 

According to the table for type \( \mathrm{G}_{2} \) on
\cite[page 427]{Carter85}, only the 
orbits $\mathcal{O}_{6, {^\vee}\mathrm{G}_{2}}$, $\mathcal{O}_{5,
  {^\vee}\mathrm{G}_{2}}$ and $\mathcal{O}_{0, {^\vee}\mathrm{G}_{2}}$
are special. These are precisely the orbits 
whose closures correspond to the associated varieties of certain
primitive ideals. Consequently, only the orbits $\mathcal{O}_6$,
$\mathcal{O}_{5,1}$, $\mathcal{O}_{5,2}$, and $\mathcal{O}_0$ can
appear in the decomposition (\ref{eq:AV(P)}) of the associated variety
$\mathrm{AV}(P)$ of a irreducible perverse sheaf $P$. 

We now have sufficient information to determine the associated
varieties of the perverse sheaves $P(\xi_i)$, where $\xi_i = (S_i,
\underline{\mathbb{C}}_{S_i})$ for $i = 0, 1, 2$ and $9$ in (\ref{eq:IrrObj1}).  
By Proposition~\ref{closedorb}, the orbits $S_0$, $S_1$, and $S_2$ are
closed and minimal in the Bruhat order. Therefore  Lemma
\ref{lem:lemorbit} implies 
\begin{equation}
\label{CC012}
CC(P(\xi_i))=\overline{T_{S_{i}}^{\ast}X},
\end{equation}
and
$$
\mathrm{AV}(P(\xi_i))=\mu\left(\overline{T_{S_{i}}^{\ast}X}\right)
\quad i = 0,1,2.
$$
Fortunately, these images are computed by Samples. According to
\cite[{Proposition 4.1, tables 1 and 
    4}]{Samples}\footnote{In Samples' notation 
%the involution defining ${}^{\vee}K$ and the numbering of the closed
%orbits differ from the conventions used here. Specifically, in
%\cite{Sample}*{Section 4}, Samples fixes a Cartan involution $\theta$
%satisfying 
%\begin{align*}
%\theta(X_{{}^{\vee}\alpha_1})\, =\, -X_{{}^{\vee}\alpha_1}\quad\text{and}\quad
%\theta(X_{{}^{\vee}\alpha_2})\, =\, X_{{}^{\vee}\alpha_2}.
%\end{align*}
orbit our orbit $S_1$ corresponds a closed orbit labelled with the
number one. However, 
our orbits $S_0$ and $S_2$ correspond to closed orbits number two and
three, respectively. }, we have  
\begin{equation}\label{eq:momentmapfor012}
\mu\left(\overline{T_{S_{0}}^{\ast}X}\right)\, =\, \overline{\mathcal{O}}_{6},\quad
\mu\left(\overline{T_{S_{1}}^{\ast}X}\right)\, =\,
\overline{\mathcal{O}}_{5,1},\quad
\mu\left(\overline{T_{S_{2}}^{\ast}X}\right)\, =\,
\overline{\mathcal{O}}_{5,2}. 
\end{equation}

To treat the perverse sheaf \( P(\xi_9) \), we proceed differently.  
The perverse sheaf \( P(\xi_9) \) corresponds to the trivial representation  
\( \pi({^\vee}\xi_9) \) of the split real form \(
   {}^{\vee}\mathrm{G}_{2}(\mathbb{R}) \) (\ref{landscape1})
following the correspondences of \cite[Proposition 2.7]{ICIII}.    The
trivial representation belongs to a 
   family of representations commonly denoted by $A_{\mathfrak{q}}$
\cite[Section~2]{Vogan-Zuckerman}.  In fact, $\pi({}^{\vee}\xi_9) =
A_{{}^{\vee}\mathfrak{g}_2}.$ (see the paragraph 
following  \cite[Theorem~2.5]{Vogan-Zuckerman}).  The associated
variety of $A_{\mathfrak{q}}$ is known
\cite[Proposition~3.4]{Barbasch-Vogan-Trombi}, and yields 
\begin{equation}\label{eq:AVfor9}
\mathrm{AV}(P(\xi_9)) = \mathcal{O}_0.
\end{equation} 
\begin{comment}
For each \( \theta \)-stable parabolic subalgebra \( \mathfrak{q}
\subset {}^{\vee}\mathfrak{g}_2 \),   recall the construction of the
irreducible   
\( ({}^{\vee}\mathfrak{g}_2, {}^{\vee}K) \)-module \( A_\mathfrak{q} \)  
given in \cite[Section~2]{Vogan-Zuckerman}.  
As noted in the paragraph following \cite[Theorem~2.5]{Vogan-Zuckerman},  
the trivial representation satisfies
${}^{\vee}\pi(\xi_9) = A_{{}^{\vee}\mathfrak{g}_2}.$  
Combining this identification with
\cite[Proposition~3.4]{Barbasch-Vogan-Trombi},  we therefore conclude
that 
\begin{align}\label{eq:AVfor9}
\mathrm{AV}(P(\xi_9)) = \mathcal{O}_0.
\end{align}
\end{comment}
%The description of the associated varieties for the remaining
%irreducible perverse sheaves will be given in Proposition
%\ref{prop:Cells-AV}. 
We can go further with $P(S_{9})$.  Since
\[
\mu\left(\mathrm{Ch}(P(\xi_9))\right) = \mathrm{AV}(P(\xi_9)) = \mathcal{O}_0,
\]
it follows that for every ${}^{\vee}K$-orbit $S$ on $X$ with 
\( \chi_S^{\mathrm{mic}}(P(\xi_9)) \neq 0 \), we have
\begin{align}\label{eq:mu=0}
\mu(\overline{T_S^* X}) = \mathcal{O}_0.
\end{align}
However, if $S \neq S_{9}$ and we take the Borel subalgebra
$\mathfrak{b}'$ in $S$ then the vector space $T_{S,\mathfrak{b}'}^{*}X$ is 
non-zero, and so the projection given by $\mu$ is
non-zero on $T_{S}^{*}X$ \eqref{eq:TSX}.  In consequence
(\ref{eq:mu=0})  holds only when \( S =
S_9 \), and
\[
CC(P(\xi_9)) = \overline{T_{S_9}^{*}X}.
\]
In equation~\eqref{eq:CC9} below, we will present an alternative method to
obtain this equality.

\section{Weyl group actions and the  characteristic cycles map}
\label{cohcont}

Recall from the introduction that the characteristic cycle map
$$CC: \mathscr{K}(X, {^\vee}K) \rightarrow \mathscr{L}(X,
{^\vee}K)$$
is $W$-equivariant.  The $W$-action on the domain stems from the
coherent continuation representation, and the $W$-action on the
codomain is due to Hotta and Rossmann.  We examine each of these two
$W$-actions in turn and then use them to draw conclusions about
characteristic cycles of the perverse sheaves on orbits $S_{3}$,
$S_{4}$ and $S_{9}$ (\ref{eq:BruhatOrbits}).  

We conclude by
returning to the $W$-action on  $\mathscr{L}(X,{^\vee}K)$.  The
definition of this $W$-action 
involves  several unknown integers.   Using the
properties of the characteristic cycles map and the moment map,
we are able to give a table of values for the $W$-action with fewer
unknown integers. 

\subsection{The coherent continuation representation}

We begin by describing the coherent continuation representation in the
context of $\mathscr{K}(X, {^\vee}K)$ (\ref{eq:IrrObj1}).  We then introduce
two concepts related to the coherent continuation representation that are
particularly valuable to us:  the tau-invariant, and Harish-Chandra
cells.  The Harish-Chandra cells are valuable as they provide information
about the associated varieties of irreducible perverse sheaves.

Although the definition of the coherent continuation representation is
fascinating, we need not present it here.   
For our computations on ${}^{\vee}\mathrm{G}_{2}$
the only information about coherent continuation that we 
need, is an explicit 
depiction of the action of the simple reflections in $W$ 
on the irreducible perverse sheaves on $X$. 
Luckily, the Atlas of Lie Groups and Representations software has
command called  
$\mathtt{coherent{\_}irr}$, which, upon entering the parameter for an
irreducible  $({^\vee}\mathfrak{g}_{2},{^\vee}K)$-module
${}^{\vee}\pi$ and a simple root $\alpha$ as input, 
displays the  value of the coherent continuation representation of
$\alpha$ applied to ${^\vee}\pi$.  The value is a virtual
module, displayed  as a $\mathbb{Z}$-linear combination of irreducible 
$({^\vee}\mathfrak{g}_{2}, {^\vee}K)$-modules. We take the image of
this virtual module under the Beilinson-Bernstein and Riemann-Hilbert
correspondences ($P(\xi) \leftrightarrow \pi({^\vee}\xi)$ in
(\ref{landscape1})) to obtain a $\mathbb{Z}$-linear 
combination of irreducible perverse sheaves on $X$, \emph{i.e.}~an
element in $\mathscr{K}(X, {^\vee}K)$.    The same procedure applies
to ${^\vee}\mathrm{G}_{2}(\lambda)$ and $X(\lambda)$ as in Section
\ref{section:FlagNonIntegral}.

The most elaborate case arises for
${}^{\vee}\mathrm{G}_2$ and $X$.  
The action of the simple reflections on the set of irreducible
perverse sheaves  $P(\xi_i)$ as in (\ref{eq:IrrObj1})
is summarized in the  following  table. For simplicity, we write
$P_i$ instead of $P(\xi_i)$ and $s_{i}$ instead of $s_{{^\vee}\alpha_{i}}$

\begin{minipage}{\linewidth}
\bigskip
\centering
\captionof{table}{$\left({}^{\vee}\mathrm{G}_{2}(\lambda),{}^{\vee}K(\lambda)
\right)\,  
  =\,
  {}(\mathrm{G}_{2},\mathrm{SL}_2\times_{\upmu_2}\mathrm{SL}_2)$}
\label{eq:ccaction} 
\begin{tabular}{ | l | c | c | }
\hline 
 $i$ & $s_{1}\cdot P_i$ & $s_{2}\cdot P_i$
\\ \hline \hline
		
   0 & $P_0+P_4$ & $P_0+P_3$\\	\hline
   1 & $-P_1$ & $P_1+P_3$ \\ \hline
   2 & $P_2+P_4$ & $-P_2$ \\ \hline
   3 & $P_1+P_3+P_6$ & $-P_3$  \\ \hline
   4 & $-P_4$ & $P_2+P_4+P_5$\\ \hline
   5 & $P_4+P_5+P_8$ & $-P_5$\\ \hline
   6 & $-P_6$ & $P_3+P_6+P_7$\\ \hline
   7 & $P_6+P_7+P_9+P_{10}$ & $-P_7$\\ \hline
   8 & $-P_8$ & $P_5+P_8+P_9+P_{11}$\\ \hline
   9 & $-P_9$ & $-P_9$\\ \hline
   10 & $-P_{10}$ & $P_7+P_{10}$\\ \hline
   11 & $P_8+P_{11}$ & $-P_{11}$ \\ 
 \hline  
\end{tabular}

\bigskip 
  \end{minipage}
For any of the other pairs
$\left({}^{\vee}\mathrm{G}_{2}(\lambda),{}^{\vee}K(\lambda) \right)$
described in Proposition~\ref{Ilambdaorbits}, the situation is considerably
simpler. A description of the coherent continuation action in those
cases will be given in Section~\ref{section:G2computationsNon-integral}. 

We will make use of the notion of \emph{tau-invariant}, following
\cite[{Definition 7.3.8}]{greenbook}.  For an irreducible
$P\in \mathscr{K}(X, {^\vee}K)$, we say that the simple root
${^\vee}\alpha_{j}$ is in the (Borho-Jantzen-Duflo) tau-invariant
$\tau(P)$ of $P$ if 
\begin{equation}\label{eq:verticalP}
s_{j} \cdot P=-P, \quad j = 1,2.
\end{equation}
By Table \ref{eq:ccaction}, these tau-invariants are given as follows 
\begin{align}
\begin{aligned}\label{eq:tau}
{^\vee}\alpha_1\in \tau(P(\xi_i))&\quad\text{ if and only if }\quad
i\in\{1,4,6,8,9,10\},\\ 
{^\vee}\alpha_2\in \tau(P(\xi_i))&\quad\text{ if and only if }\quad
i\in\{2,3,5,7,9,11\}. 
\end{aligned}
\end{align}

Another notion related to the coherent continuation representation is
that of \emph{Harish-Chandra cells}.  The notion of Harish-Chandra
cells was originally formulated in the setting of
$({^\vee}\mathfrak{g}, {^\vee}K)$-modules.  However, we identify these
modules with perverse sheaves as we did earlier in the discussion on
coherent continuation.
Harish-Chandra cells  partition the
set of irreducible perverse sheaves as follows. Suppose $P, P' \in
\mathscr{K}(X, 
{^\vee}K)$ are irreducible.  Following the description
in \cite[{Definition~2.5}]{Barbasch-Vogan-Trombi}, as enabled by  
\cite[{Proposition~2.6}]{Barbasch-Vogan-Trombi}, we write
$$P\underset{LR}{<}P'$$
if there exists $w \in W$ such that $P$ appears as a summand
in coherent continuation representation of $w$ applied to $P'$ (a
$\mathbb{Z}$-linear combination of irreducible perverse sheaves).
We write $P\sim P'$ if both $P\underset{LR}{<}P'$ and
$P'\underset{LR}{<}P$ hold. It is  
straightforward to verify that $\sim$ is an equivalence
relation. 
The equivalence classes in this partition are called Harish-Chandra cells.
Of particular significance to us is their connection to the associated
varieties (\ref{eq:AVdef}).

Using Table~\ref{eq:ccaction}, it is straightforward to verify the
following description of the Harish-Chandra cells for the irreducible
perverse sheaves on $X = {^\vee}\mathrm{G}_{2}/ \, {^\vee}B$.
\begin{prop}\label{prop:HCcells}
The Harish-Chandra cells of $\mathscr{K}(X, {^\vee}K)$ are 
\begin{align*}
\mathrm{cell}_0\ =& \ [P_{0}]\\
\mathrm{cell}_1\ =& \ [P_1,P_3,P_6,P_7,P_{10}] \\
\mathrm{cell}_2\ =& \ [P_2,P_4,P_5,P_8,P_{11}] \\
\mathrm{cell}_3\ =& \ [P_{9}]
\end{align*}
\end{prop}
The relationship between Harish-Chandra cells and associated varieties
is described in the following theorem due to Vogan. 
\begin{prop}\label{prop:Cell-AV0}
Let $P$ and $P'$ be two irreducible perverse sheaves
that belong to the same Harish-Chandra cell. Then their corresponding
associated varieties \eqref{eq:AVdef} coincide
$$\mathrm{AV}(P) = \mathrm{AV}(P').$$
\end{prop}
The combination of Propositions \ref{prop:HCcells} and \ref{prop:Cell-AV0}
leads us to
\begin{prop}\label{prop:Cells-AV}
The associated varieties of the irreducible perverse sheaves in
$\mathscr{K}(X, {^\vee}K)$ are given as follows:
\begin{align*}
\mathrm{AV}(P(\xi_0))&=\overline{\mathcal{O}}_{6},\\
 \mathrm{AV}(P)&=\overline{\mathcal{O}}_{5,1}, \quad P\in\mathrm{cell}_1,  \\ 
\mathrm{AV}(P)&=\overline{\mathcal{O}}_{5,2},
 \quad P\in\mathrm{cell}_2,\\
\mathrm{AV}(P(\xi_9))&=\overline{\mathcal{O}}_{0}. 
\end{align*}
\end{prop}
\begin{proof}
As established in \eqref{eq:momentmapfor012} and \eqref{eq:AVfor9}, we have
\[
\mathrm{AV}(P(\xi_0)) = \overline{\mathcal{O}}_{6}, \quad
\mathrm{AV}(P(\xi_1)) = \overline{\mathcal{O}}_{5,1}, \quad
\mathrm{AV}(P(\xi_2)) = \overline{\mathcal{O}}_{5,2}, \quad 
\mathrm{AV}(P(\xi_9)) = \overline{\mathcal{O}}_{0}.
\]
The desired result now follows from Proposition~\ref{prop:Cell-AV0},
\end{proof}

\subsection{The $W$-action on $\mathscr{L}(X, {^\vee}K)$}
\label{WactionL}

We have reviewed the $W$-action on $\mathscr{K}(X, {^\vee}K)$, the
domain of the characteristic cycle map (\ref{eq:CC-map}). We now turn
to the study of the $W$-action on its codomain, namely the   $\mathbb{Z}$-module
$\mathscr{L}(X,{^\vee}K)$ (\ref{LXK}).  As we shall see, there is a
compatibility between the $W$-action and the moment map $\mu$ which
allows one to construct subquotients of $\mathscr{L}(X,{^\vee}K)$.
These subquotients can be characterized using the Springer
correspondence and the work of Rossmann.  This characterization will be
exploited in the proof of Lemma \ref{lem=chiS5} below.  

There is a natural isomorphism
between the formal $\mathbb{Z}$-module $\mathscr{L}(X,{^\vee}K)$ and
the top Borel-Moore homology  $H_{\mathrm{top}}^{\mathrm{BM}} (T_{{^\vee}K}^{\ast}X,
\mathbb{Z})$ \cite[Section 2.2]{Maxim}.
By \cite[Lemma 19.1.1]{fulton}, the set
$\left\{[\overline{T_S^{\ast}X}]:\, S\in {^\vee}K\setminus X\right\}$  
of fundamental classes of conormal bundle closures form a basis for
the top Borel-Moore homology  $H_{\mathrm{top}}^{\mathrm{BM}}
( T_{{^\vee}K}^{\ast} X,\mathbb{Z})$, 
$$H_{\mathrm{top}}^{\mathrm{BM}} (T_{{^\vee}K}^{\ast}X, \mathbb{Z})\, =\,
\bigoplus_{S\in {^\vee}K\setminus X} \mathbb{Z}\,
[\overline{T_S^{\ast}X}]. 
$$
We identify $\overline{T_S^{\ast}X}$ with $[\overline{T_S^{\ast}X}]$ and write
$\mathscr{L}(X,{^\vee}K) = H_{\mathrm{top}}^{\mathrm{BM}}
(T_{{^\vee}K}^{\ast}X, \mathbb{Z})$. 

In \cite{Hotta85}, \cite{HottaR} and \cite{Rossmann95}, the
topological construction of the Springer representation 
% due to Kazhdan-Lusztig 
is expanded in order to define an action of the Weyl group
on the Borel-Moore homology space. 
%can be adapted to prove that the top Borel-Moore homology Htop (T ∗K
%B, Z) is a representation of W.  
According to \cite[{Lemma 1 \S3 and Theorem \S4}]{Hotta85}),
this action is graded in the following sense.  If $w \in W$ and 
$\mu\left(\overline{T_S^{\ast}X}\right)=\overline{\mathcal{O}}$,
then $w\cdot \overline{T_S^{\ast}X}$
is a linear combination of conormal bundles 
$\overline{T_{S'}^{\ast}X}$, such that
$$
\mu\left(\overline{T_{S'}^{\ast}X}\right)\subset\overline{\mathcal{O}}.
$$
Let us present the action
of the simple reflections on $H_{\mathrm{top}}^{\mathrm{BM}}
(T_K^{\ast}X, \mathbb{Z})$. As in the previous section it simplifies the 
presentation to work in the context of an arbitrary connected complex
reductive group $G$ and the associated objects $X = G/B$, $K$ \emph{etc}.
Let $\alpha$ be a simple root relative to $B$, $s_{\alpha} \in W$ be its simple
reflection, and consider the natural projection 
\begin{equation}
\label{pimap}
\pi_{\alpha}: G/ B\,
\longrightarrow\, G/(B\sqcup B s_\alpha B).
\end{equation}
Following \cite[Section 2]{Hotta85} and \cite[{\S3, Lemma
    2}]{Tanisaki}, we say that a  
$K$-orbit $S$ is \emph{$s_\alpha$-vertical} if  for every $ x \in S$
the intersection 
$$
\pi_\alpha^{-1}(\pi_\alpha(x)) \cap S,
$$
is open and dense in $\pi_\alpha^{-1}(\pi_\alpha(x))$. If this
condition is not satisfied, we say 
that $S$ is \emph{$s_\alpha$-horizontal}. 
The following theorem gives a partial description of the action. 
\begin{theo}[{\cite[Sections 3 and 4]{Hotta85}}, {\cite[Proposition
        8]{Tanisaki}}] 
\label{theo:Waction} Let $S$ be a $K$-orbit in $X$. Fix a simple root
$\alpha \in R(B,T)$. 
\begin{enumerate}
\item 
If $S$ is $s_\alpha$-{\it{vertical}}, then
\begin{align}\label{eq:vertical}
s_\alpha\cdot \overline{T_S^{\ast}X}\, =\, -{\overline{T_S^{\ast}X}}. 
\end{align}
\item\label{theo:Waction2}  If $S$ is $s_\alpha$-{\it{horizontal}}, then
for any $s_\alpha$-vertical orbit
$S'\subset\overline{\pi_{\alpha}^{-1}(\pi_{\alpha}(S))}$  
with
$$
\mu\left(\overline{T^{\ast}_{S'}X}\right)\, \subset\, 
\mu\left(\overline{T^{\ast}_{S}X}\right).
$$
there is a non-negative integer $n_{S, S'}$, such that
\begin{align}\label{eq:horizontal}
s_{\alpha} \cdot \overline{T^{*}_{S}X}\, =\, \overline{T^{*}_{S}X}\, +\, 
\sum_{S'}n_{S,S'} \, \overline{T^{*}_{S'}X}.
\end{align}
\end{enumerate}
\end{theo}
Although the $W$-action on Borel-Moore homology has been the subject
of extensive study, we do not know of any general method %is currently known 
for explicitly determining the action of a simple reflection on the
conormal orbits. In Table~\ref{tab:Wactiononconormal}, at the end of
the next section, we present an explicit depiction of this action in
the case of $X={}^{\vee}\mathrm{G}_{2}/{}^{\vee}B$.  

Before turning to the discussion of the $W$-equivariance of the
characteristic cycle map \eqref{eq:CC-map}, we briefly elaborate on
the relationship between this $W$-action and the Springer
correspondence, as described in \cite[Section 4]{Rossmann91}. Let $N
\in \mathfrak{g}$ be a nilpotent element. Denote by $\mathcal{O}_K$
the $K$-orbit of $N$, and let $A_K(N)$ and $A_G(N)$ denote the
component groups of the centralizers of $N$ in $K$ and $G$,
respectively. 

According to  \eqref{eq:horizontal}, the $\mathbb{Z}$-module
\begin{equation}\label{eq:Z-module}
\
\bigoplus_{\{S\in K\setminus X:\,
  \mu\left(\overline{T_{S}^{\ast}X}\right)\, \subset\,
  \overline{\mathcal{O}}_{K} \}} \mathbb{Z} \, \overline{T_S^\ast X},
\end{equation}
is $W$-invariant. Hence, the quotient
\begin{equation}\label{eq:Z-module-quotient}
M(\mathcal{O}_K):=\
\bigoplus_{S\in K\setminus X:\,
  \mu\left(\overline{T^{\ast}_{S}X}\right)\, \subset\,
  \overline{\mathcal{O}}_{K}\,} \mathbb{Z}\, \overline{T_S^{\ast}X}/  
\ \bigoplus_{S\in K\setminus X:\,
  \mu\left(\overline{T^{\ast}_{S}X}\right)\, \varsubsetneq\,
  \overline{\mathcal{O}}_{K}\,} \mathbb{Z} \, \overline{T_S^{\ast}X}
\end{equation} 
is naturally a $W$-module, with basis indexed by the $K$-orbits $S$
such that $\mu(T_S^{\ast}X) = \overline{\mathcal{O}}_K$. 
By Springer's correspondence, in the formulation of \cite[Theorem
  4.1]{Rossmann91},  
%following \cite[(Springer's) Theorem 4.1]{Rossmann} version of
%Springer correspondance 
there is a representation of $W \times A_G(N)$ on 
$\mathrm{H}_{\mathrm{top}}^{\mathrm{BM}}(\mu^{-1}(N),\mathbb{Z})$,
which decomposes as a direct sum
\begin{equation}\label{eq:Springer-representation}
\mathrm{H}_{\mathrm{top}}^{\mathrm{BM}}(\mu^{-1}(\eta),\mathbb{Z})=\bigoplus_{\mu\in
  \widehat{A_G(N)}} 
\phi_\mu\otimes \mu,
\end{equation}
where  $\widehat{A_G(N)}$ denotes the set of irreducible
$\mathbb{Z}[A_{G}(N)]$-modules, and $\phi_\mu$ is either zero or an
irreducible $\mathbb{Z}[W]$-module.
% on which $A_G(N)$ acts trivially. 
%and $V_\mu$ is a module on which $A_G(N)$ acts via $\mu$, while $W$
%acts trivially. 
% and $V_\mu$ is a representation of $A_G(N)$ (with type $\mu$) on
% which $W$ acts trivially. 

Since $A_K(N)$ maps to $A_G(N)$, we may consider the space of
$A_K(N)$-invariants in $\mathrm{H}_{\mathrm{top}}^{\mathrm{BM}}(\mu^{-1}(N),
\mathbb{Z})$, denoted 
\[
\mathrm{H}_{\mathrm{top}}^{\mathrm{BM}}(\mu^{-1}(N), \mathbb{Z})^{A_K(N)}.
\]
By \cite[Theorem 2.9.1]{Trapa2005}, there is a $\mathbb{Z}[W]$-module
isomorphism 
\begin{equation}\label{eq:Springer-representation2}
M(\mathcal{O}_K)\cong \mathrm{H}_{\mathrm{top}}^{\mathrm{BM}}(\mu^{-1}(N), \mathbb{Z})^{A_K(N)}.
\end{equation}
It follows that
\begin{align}\label{eq:cardinalite-dimension}
\#\left\{S\, \in\,  K\setminus X\, :\ 
 \mu(\overline{T_S^{\ast}X})=\overline{\mathcal{O}}_K 
 \right\}=\mathrm{rank}\left( \mathrm{H}_{\mathrm{top}}^{\mathrm{BM}}(\mu^{-1}(N),
 \mathbb{Z})^{A_K(N)} \right).   
\end{align}
(\emph{cf.}~\cite[Proposition 2.15 and Remark 2.16]{CNT}).

For groups of type $\mathrm{G}_{2}$, Carter attributes the computation of
the Springer correspondence to Springer himself.
Carter displays the description of the $W$-action on
$\mathrm{H}_{\mathrm{top}}^{\mathrm{BM}} 
(T_{{}^{\vee}K}^{\ast}({}^{\vee}\mathrm{G}_{2}/{}^{\vee}B),
\mathbb{Z})$ on \cite[page 427]{Carter85}. In Section
\ref{section:G2computations} this description will allow us to 
explicitly describe the $W$-action on a basis of 
$\mathrm{H}_{\mathrm{top}}^{\mathrm{BM}}(\mu^{-1}(N),
\mathbb{Z})^{A_{{}^{\vee}K}(N)}$ for $N$ in any nilpotent ${}^{\vee}K$-orbit
$\mathcal{O}$. 
%For more details, see the comment after Table \ref{tab:Wactiononconormal}. 

\subsection{The $W$-equivariance of the characteristic cycle map}

It is time to put the $W$-actions to good use. We maintain the general
setting of $G$ being a connected complex reductive algebraic group as
in the previous section.
\cite[Theorem 
  1]{Tanisaki}, asserts that taking characteristic cycle map is
$W$-equivariant, \emph{i.e.}
\begin{equation}\label{eq:W-equivariant}
CC(w\cdot P)=w\cdot CC(P), \quad P\in \mathscr{K}(X,K), \ w\in W.
\end{equation}
The $W$-equivariance of $CC(\cdot)$ 
is the central tool upon which our method for computing the
characteristic cycles of irreducible perverse sheaves on
${}^{\vee}\mathrm{G}_{2}/ \,{^\vee}B$ relies.

To simplify the notation, from now on, for each $K$-orbit $S$ in $X$
and each irreducible perverse sheaf $P$ on $X$, we write 
$$
[T_S]=\overline{T_S^\ast X}\quad\text{and}\quad \chi_{S}^{}(P) =
\chi_{S}^{\mathrm{mic}}(P).
$$
The following result  plays an important role in 
simplifying the computations in the next section.
\begin{prop}\label{prop:CC2}
Let $P \in \mathscr{K}(X,K)$ be irreducible.
Suppose $\alpha \in \tau (P)$ (see \eqref{eq:verticalP}). If 
$\chi_{S} (P) \neq 0$ then $S$ is $s_\alpha$-vertical. 
That is, all conormal bundles appearing in ${CC}(P)$ have
$s_\alpha$-vertical orbits. 
\end{prop}
\begin{proof} 
We have
$$
CC(P)\ =\ \sum_{S,\, s_\alpha\mbox{-vert.}} \chi_{S}^{}(P)
  \, [T_{S}]\, +  \sum_{S,\, s_\alpha\mbox{-hor.}}
  \chi_{S}^{}(P)   \,   [T_{S}].
$$
Now, since $\alpha \in \tau (P)$, equalities (\ref{eq:horizontal}) and
(\ref{eq:W-equivariant}) imply   
  \begin{align*}
  &s_{\alpha}\cdot {CC}(P) = 
{CC}(s_\alpha\cdot P)  
  =\mathrm{Ch}(-P).
\end{align*}
Therefore
\begin{align*}
  \sum_{S,\, s_\alpha\mbox{-vert.}} \chi_{S}^{}(P)
  \,s_{\alpha} \cdot [T_{S}] +  \sum_{S,\, s_\alpha\mbox{-hor.}}
  \chi_{S}^{}(P)   \,s_{\alpha} \cdot   [T_{S}] =&
   -\sum_{S,\, s_\alpha\mbox{-vert.}} \chi_{S}^{}(P)
  \, [T_{S}]\ -\sum_{S,\, s_\alpha\mbox{-hor.}}
  \chi_{S}^{}(P)   \,   [T_{S}],
\end{align*}
and by (\ref{eq:vertical}) and (\ref{eq:horizontal}), we  write
\begin{align*}  
  &\sum_{S,\, s_\alpha\mbox{-vert.}} -\chi_{S}^{}(P)
  \,[T_{S}] +  \sum_{S,\, s_\alpha\mbox{-hor.}}
  \chi_{S}(P)   \,s_{\alpha} \cdot   [T_{S}] =
   -\sum_{S,\, s_\alpha\mbox{-vert.}} \chi_{S}^{}(P)
  \, [T_{S}] \, -
  \sum_{S,\, s_\alpha\mbox{-hor.}}
  \chi_{S}^{}(P)   \,   [T_{S}]\\
  &\Longrightarrow   \sum_{S,\, s_\alpha\mbox{-hor.}}
  \chi_{S}^{}(P)   \,s_{\alpha} \cdot   [T_{S}] =
   -  \sum_{S,\, s_\alpha\mbox{-hor.}}
   \chi_{S}^{}(P)   \,   [T_{S}]\\
   &\Longrightarrow  \sum_{S,\, s_\alpha\mbox{-hor.}}
  \chi_{S}^{}(P)   \left([T_{S}] + \sum_{S'}
  n_{S,S'} \,  [T_{S'}]\right)=
   -  \sum_{S,\, s_\alpha\mbox{-hor.}}
   \chi_{S}^{}(P)   \,   [T_{S}]\\
   &\Longrightarrow  \sum_{S,\, s_\alpha\mbox{-hor.}}
  \chi_{S}^{}(P)   \left(2[T_{S}] + \sum_{S'}
  n_{S,S'} \,  [T_{S'}]\right) = 0.
\end{align*}  
Therefore, for all $s_\alpha\mbox{-horizontal orbit }S$, we have
$\chi_{S}^{}(P)=0.$ 
The result follows.  
\end{proof}

Now we return to the setting of $G = {^\vee}\mathrm{G}_{2}$.
Using Proposition \ref{prop:CC2} we can deduce which of the
${}^{\vee}K$-orbits $S$ are horizontal and which ones are vertical.
To see this, recall that  by Lemma \ref{lem:lemorbit}
the conormal bundle of orbit $S_i$ appears with multiplicity one in
the characteristic cycle of $P(\xi_i)$ for $1\leq i\leq 9$
(\ref{eq:IrrObj1}).    Combining this observation with
the tau-invariants  in (\ref{eq:tau}) and Proposition \ref{prop:CC2},
we deduce that
\begin{equation}
\begin{aligned}\label{eq:verticaldescription}
S_i \text{ is $s_{1}$-vertical }\quad&\text{if and only
  if}\quad i\in\{1,4,6,8,9\},\\
S_i \text{ is $s_{2}$-vertical }\quad&\text{if and only
  if}\quad i\in\{2,3,5,7,9\}.
\end{aligned}
\end{equation}

As a warm-up for the computations in the next section,
% we demonstrate the utility of Proposition \ref{prop:CC2} by 
we compute the characteristic cycles of 
$P(\xi_3)$, $P(\xi_4)$, and $P(\xi_9)$. By the Bruhat
ordering relation described in  Diagram (\ref{eq:BruhatOrbits}), we have
$$\overline{S_3}\, =\, S_0\, \cup\,  S_1\, \cup\,  S_3
\quad \text{and} \quad \overline{S_4}\, =\, S_0\, \cup\,  S_2\, \cup\,  S_4.$$
Consequently, Lemma \ref{lem:lemorbit} tells us that 
\begin{align*}
CC(P(\xi_3)) &= \chi_{S_0}^{}(P(\xi_3))\, [T_{S_0}] +
\chi_{S_1}^{}(P(\xi_3))\, [T_{S_1}] + [T_{S_3}], \\ 
CC(P(\xi_4)) &= \chi_{S_0}^{}(P(\xi_4))\, [T_{S_0}] +
\chi_{S_2}^{}(P(\xi_4))\, [T_{S_2}] + [T_{S_4}]. 
\end{align*}
Since \( s_{2} \in \tau(P(\xi_3)) \), and  \( S_0 \) and \( S_1 \) are
\( s_{2} \)-horizontal, the contrapositive of
Proposition~\ref{prop:CC2} implies that 
\begin{equation}\label{eq:CC3}
CC(P(\xi_3)) = [T_{S_3}].
\end{equation}
Similarly, since \( s_1 \in \tau(P(\xi_4)) \), and \( S_0 \) and \(
S_2 \) are \( s_{1} \)-horizontal, we deduce that 
\begin{equation}\label{eq:CC4}
CC(P(\xi_4)) = [T_{S_4}].
\end{equation}
The case of $P(\xi_9)$ is simpler, since $\tau(P(\xi_9))=\{ {^\vee}\alpha_1,
{^\vee}\alpha_2 \} $, and \( S_9 \) is the only orbit that is \( s_\alpha
\)-vertical for both simple roots. By the contrapositive of
Proposition~\ref{prop:CC2}, we deduce that $\chi^{}_{S_i}(P(\xi_9))=0$
for $0\leq i\leq 8$, that is
\begin{equation}\label{eq:CC9}
CC(P(\xi_9))\ =\ [T_{S_9}].
\end{equation}

\subsection{Interim progress on the $W$-action on $\mathscr{L}(X, {^\vee}K)$}

Using the $W$-equivariance of the characteristic cycles map, we are
able to provide a better description of the $W$-action 
on $\mathscr{L}(X, {^\vee}K)$. 
The complete description, provided in
Table~(\ref{tab:Wactiononconormal}), will be obtained alongside the
computation of the characteristic cycles of the remaining perverse
sheaves $P(\xi_i)$, for $5\leq i\leq 11$ (\ref{eq:IrrObj1}).

Let $\alpha_{\ell}$, $\ell = 1,2$, be a simple root and recall the
projection $\pi_{\ell} = \pi_{\alpha_{\ell}}$  (\ref{pimap}).  For each
$s_{\ell}$-horizontal orbit $S \in {^\vee}K \backslash X$, let 
$S'  \in  {^\vee}K \backslash X$ denote the unique
open orbit in the closure
$\overline{\pi_{\ell}^{-1}(\pi_{\ell}(S))}$. Then for any other ${^\vee}K$-orbit
$S'' \subset \overline{\pi_{\ell}^{-1}(\pi_{\ell}(S))}$, we have 
\begin{equation}\label{eq:SS'S''}
\dim(S'') \leq \dim(S') - 1.
\end{equation} 
Notice from Table~\ref{eq:ccaction} and (\ref{eq:tau}), that  
if $\alpha_{\ell} \notin \tau(P(\xi_{i}))$ for $0 \leq i \leq 6$, 
then there exists a unique perverse sheaf $P(\xi_j)$ 
such that $\dim(S_j) = \dim(S_{i}) + 1$ and $P(\xi_j)$ appears with
multiplicity one in $s_{\ell} \cdot P(\xi_{i})$. 
It follows from the linearity of the characteristic cycle map and
Lemma~\ref{lem:lemorbit} that $[T_{S_j}]$ appears with multiplicity
one in  
\[
CC(s_{\ell} \cdot P(\xi_{i})).
\]
On the other hand, using equation~\eqref{eq:horizontal} and 
\eqref{eq:SS'S''} we deduce that $[T_{S_j}]$ does not appear in
$s_{\ell} \cdot [T_S]$ 
for any orbit $S \neq S_{i}$ with $\chi_S(P(\xi_{i})) \neq 0$. 
By equation~\eqref{eq:W-equivariant} we have
\begin{equation}\label{eq:alpha-equivariance}
CC(s_{\ell} \cdot P(\xi_{i})) = s_{\ell} \cdot CC(P(\xi_{i})),
\end{equation}
we conclude that
\begin{equation}\label{eq:nSiSj=1}
n_{S_{i}, S_j} = 1, \quad 0 \leq i \leq 6.
\end{equation}
In the case of $P(\xi_i)$ for $i = 7, 8$, we see from 
Table~\ref{eq:ccaction} that $[T_{S_9}]$ appears with multiplicity two in 
$CC(s_\alpha \cdot P(\xi_i))$. As before, equation~\eqref{eq:SS'S''} 
implies that $[T_{S_9}]$ does not appear in $s_\alpha \cdot [T_S]$ 
for any orbit $S$ with $\dim(S) < 5$. Once again, using equality 
\eqref{eq:alpha-equivariance}, we deduce that
\begin{equation}\label{eq:nSiS9=2}
n_{S_i, S_9} = 2,\quad i=7,8.
\end{equation}
 
In order to gain more information about further coefficients $n_{S_{i}, S_{j}}$
we make use of the moment map $\mu$ as it appears in Theorem \ref{theo:Waction}.
We first determine the image under $\mu$ of the conormal bundles to
the ${^\vee}K$-orbits on $X$.   The reader should review
(\ref{eq:nilpotent-orbits}) before reading the following proposition.
\begin{prop}\label{prop:momentmapimage}
In the setting of Sections \ref{Korbits} and \ref{sec:momap}, the moment map 
%$$\mu: T_{{}^{\vee}K}^\ast X\, =\, 
%\bigcup_{i=0}^{9} [T_{S_i}]\, \longrightarrow\,  \mathcal{N}_{X,
%{}^{\vee}K}({}^{\vee}\mathfrak{g}_2^{\ast}),$$ 
%of the conormal bundle to the ${}^{\vee}K$-orbits on $X$,  
takes the following values
\begin{align*}%\label{eq:momentmapimage}
\begin{aligned}
\mu([T_{S_9}])\, &=\,\mathcal{O}_{0}\\
\mu([T_{S_8}])\, &=\,\overline{\mathcal{O}}_{3}\\
\mu([T_{S_6}])\, &=\,\mu([T_{S_7}])=\overline{\mathcal{O}}_{4}\\
\mu([T_{S_1}])\, &=\,\mu([T_{S_3}])=\overline{\mathcal{O}}_{5,1}\\
\mu([T_{S_2}])\, &=\,\mu([T_{S_4}])=\mu([T_{S_5}])=\overline{\mathcal{O}}_{5,2}\\
\mu([T_{S_0}])\, &=\, \overline{\mathcal{O}}_6.\\
\end{aligned}
\end{align*}
\end{prop}
\begin{proof}
Since $S_9$ is open in $X$, it follows from 
\eqref{eq:TSX} and the definition of moment map,
that $\mu([T_{S_9}])=\mathcal{O}_{0}$. The statement concerning the
image of the conormal bundle corresponding to any of the three closed
orbits $S_0$, $S_1$ and $S_2$ 
has been verified in (\ref{eq:momentmapfor012}). Moving to the image
of $[T_{S_4}]$, we saw in equality (\ref{eq:CC4}) that    
$CC(P(\xi_4))\, =\,  [T_{S_4}]$.
Consequently, $\mu([T_{S_4}])=\mu(\mathrm{Ch}(P(\xi_4)))$ and 
by Propositions \ref{prop:HCcells} and \ref{prop:Cells-AV}, we obtain 
$$\mu([T_{S_4}])= \mu(\mathrm{Ch}(P(\xi_4)))= \mathrm{AV}(P(\xi_4))=
\overline{\mathcal{O}}_{5,2}.$$  
By (\ref{eq:CC3}), the same argument proves that 
$\mu([T_{S_3}])=\overline{\mathcal{O}}_{5,1}$.

We see that $\mu$ sends both $[T_{S_{1}}]$ and $[T_{S_3}]$ to
$\overline{\mathcal{O}}_{5,1}$.  We next prove that these are the only
two conormal bundles sent to $\overline{\mathcal{O}}_{5,1}$.  
Let $N_{5,1}\in \mathcal{O}_{5,1}$. 
%be the element denoted by $f_1$ in \cite[Table~1]{Samples}.  
According to column three of \cite[Table~1]{Samples}, we have
\[
A_K(N_{5,1}) = S_3.
\]
By the table for type \( \mathrm{G}_{2} \) on  \cite[page 427]{Carter85}, we obtain $A_K(N_{5,1}) =A_{G}(N_{5,1})$.
Consequently, by \eqref{eq:cardinalite-dimension} and \eqref{eq:Springer-representation}, we can write
\[
\#\left\{ S \in {}^\vee K \setminus X : \mu([T_S]) = \mathcal{O}_{5,1} \right\}
= \mathrm{rank}\left( \mathrm{H}_{\mathrm{top}}^{\mathrm{BM}}
(\mu^{-1}(N_{5,1}), \mathbb{Z})^{S_3} \right)
= \mathrm{rank} \ \phi_{\mathrm{triv}},
\]
where $\mathrm{triv}$ denotes the trivial $\mathbb{Z}[S_3]$-module. As noted in 
the table for type \( \mathrm{G}_{2} \) on
\cite[page 427]{Carter85}, the rank of $\phi_{\mathrm{triv}}$ is two. Hence,  
$[T_{S_1}]$ and $[T_{S_3}]$ are the only two conormal bundles whose
moment map image is equal  to $\overline{\mathcal{O}}_{5,1}$.

To determine $\mu([T_{S_8}])$ we appeal directly to the definition of
the moment map and identify $X$ with the set of Borel subalgebras of
${^\vee}\mathfrak{g}$.   The element  $S_{8}$ in Table
\ref{orbitdata} corresponds to the ${^\vee}K$-orbit  ${^\vee}K y_{8}^{-1}
\,{^\vee}B$, where $y_{8}^{-1} = \sigma_{1} \sigma_{2} g_{1}$.  This
${^\vee}K$-orbit  is identified with the ${^\vee}K$-orbit of the
Borel subalgebra  
${^\vee}\mathfrak{b}_{8} = \mathrm{Ad}(y_{8}^{-1})  \, {^\vee}\mathfrak{b}$.
Set ${^\vee}\mathfrak{t}_{8} = \mathrm{Ad}(y_{8}^{-1})  \,
{^\vee}\mathfrak{t}$ and ${^\vee}\mathfrak{n}_{8} = \mathrm{Ad}(y_{8}^{-1})  \,
{^\vee}\mathfrak{n}$ so that 
${^\vee} \mathfrak{b}_{8} =
{}^{\vee}\mathfrak{t}_8+{}^{\vee}\mathfrak{n}_8$. 
Through equation \eqref{eq:TSX} (see the remark before
\cite[Definition 2.2]{Barchini-Zierau}), we deduce  
\begin{align*}
\mu([T_{S_8}])&=  \overline{{}^\vee K \cdot
  \left({}^{\vee}\mathfrak{g}_2/
  ({}^\vee\mathfrak{k}+{}^\vee\mathfrak{b}_8)\right)^{\ast}} \\ 
&\cong \overline{{}^\vee K\cdot \left({}^\vee\mathfrak{n}_8^{-}\cap
  ({}^\vee\mathfrak{g}_2/{}^\vee\mathfrak{k})\right)}. 
\end{align*}
 Let $\theta_{8}$  be the involution of ${}^\vee\mathfrak{g}_2$ 
defined by conjugation by the element $S_8$ given in Table
\ref{orbitdata}.  Then 
$$\mathrm{Ad}(y_{8}) ({}^{\vee}\mathfrak{n}_8^-\cap
({}^\vee\mathfrak{g}_2/{}^\vee\mathfrak{k})) \cong \mathrm{Ad}(y_{8})
({}^{\vee}\mathfrak{n}_8^-\cap 
({}^\vee\mathfrak{g}_2)^{-\mathrm{Ad}(e(\rho/2))} ) = 
{}^{\vee}\mathfrak{n}^-\cap  ({}^\vee\mathfrak{g}_2)^{-\theta_8}.$$
By Table~\ref{orbitdata3}, the simple roots ${}^\vee\alpha_1$ 
and ${}^\vee\alpha_2$ are, respectively, complex and imaginary
non-compact for $S_8$. The action of $\theta_{8}$ on the roots is
given by the Weyl group element  \(p(S_8) = s_1 s_2 s_1 s_2 s_1\) from
Table~\ref{orbitdata}.  It is then a straightforward computation to verify that
\(\theta_8 \cdot {}^\vee\alpha\) is negative for every positive root
${^\vee}\alpha \neq {^\vee}\alpha_{2}$, and that every 
positive non-simple root is also complex.  
It is immediate that
for all negative roots ${}^{\vee}\alpha\neq - {^\vee}\alpha_2$,  the root
$\theta_{8} \cdot {}^\vee\alpha $ is positive, and that for
these negative roots ${^\vee}\alpha$
\[X_{{}^\vee\alpha} \notin  ({}^\vee\mathfrak{g}_2)^{-\theta_8}\]
for any non-zero root vector $X_{{^\vee}\alpha}$. 
Consequently,
$
{}^\vee\mathfrak{n}^-\cap ({}^\vee\mathfrak{g}_2)^{-\theta_8}\, =\,
\mathbb{C}X_{-{}^\vee\alpha_2}$, and  
$$
\mathfrak{n}_8^-\cap  ({}^\vee\mathfrak{g}_2/{}^\vee\mathfrak{k}) =
\mathbb{C}\, (\mathrm{Ad}(y_{8}^{-1})X_{-{}^\vee\alpha_2}). 
$$
Moreover, since ${}^\vee\alpha_2$ is a long root, it follows from
\cite[Table 1]{Dokovic} that the orbit ${{}^\vee K\cdot
  X_{-{}^\vee\alpha_2}}$ is three-dimensional. This implies that
$X_{-{}^\vee\alpha_2}$ belongs to the six-dimensional nilpotent
${}^{\vee}\mathrm{G}_{2}$-orbit $\mathcal{O}_{3,
  {^\vee}\mathrm{G}_{2}}$.   As
$y_{8} \in {^\vee}\mathrm{G}_{2}$ the 
nilpotent element  $\mathrm{Ad}(y_{8}^{-1})X_{-{}^\vee\alpha_2}$ also
belongs to $\mathcal{O}_{3, {^\vee}\mathrm{G}_{2}}$. 
Since $\mathcal{O}_3$ is the unique nilpotent ${}^{\vee}K$-orbit such
that ${}^{\vee}\mathrm{G}_{2}\cdot \mathcal{O}_3 = \mathcal{O}_{3,
  {^\vee}\mathrm{G}_{2}}$, we deduce that ${}^{\vee}K \cdot
(gX_{-{}^\vee\alpha_2}g^{-1}) = \mathcal{O}_3$. 
This implies  $\mu([T_{S_8}]) = \overline{\mathcal{O}}_3$.

On the other hand, according to the table for type \( \mathrm{G}_{2}
\) on \cite[page 427]{Carter85},  
for any $N_3\in\mathcal{O}_3$, equality 
\eqref{eq:cardinalite-dimension} reduces to
\begin{align*}
\#\left\{ S \in {}^\vee K \setminus {}^\vee G / {}^\vee B : \mu([T_S])
= \mathcal{O}_3 \right\} 
\ &=\  \mathrm{rank}
\mathrm{H}_{\mathrm{top}}^{\mathrm{BM}}(\mu^{-1}(N_3), \mathbb{Z})^{\{1\}} 
\ =\ \mathrm{rank} \phi_{\mathrm{triv}}= 1.
\end{align*}
Hence, $[T_{S_8}]$ is the unique conormal bundle mapping to
$\mathcal{O}_3$ under the moment map. 

To conclude, since $S_3 \leq S_6 \leq S_7$ in the weak order
closure—that is, the partial order on ${}^{\vee}K$-orbits obtained by
considering only the solid arrows in Diagram~\ref{eq:BruhatOrbits}—it
follows from \cite[Proposition 2.6]{CNT}, together with the fact that 
$\overline{\mathcal{O}}_3 \subset \mu([T_{S}])$ for all $S \neq S_9$, that
\[
\overline{\mathcal{O}}_3 \subset \mu([T_{S_7}]) \subset \mu([T_{S_6}])
\subset \mu([T_{S_3}]) = \overline{\mathcal{O}}_{5,1}. 
\] 
The only possibilities for $\mu([T_{S_7}])$ and $\mu([T_{S_6}])$ are
$\overline{\mathcal{O}}_3, \overline{\mathcal{O}}_4$ and
$\overline{\mathcal{O}}_{5,1}$. 
However, $[T_{S_1}]$ and $[T_{S_3}]$ are the only conormal bundles
whose moment map image equals $\overline{\mathcal{O}}_{5,1}$,  and we
have just verified that $[T_{S_8}]$ is the unique conormal bundle
whose image under the moment map is $\overline{\mathcal{O}}_3$.   
Therefore
\[
\mu([T_{S_6}]) = \mu([T_{S_7}]) = \overline{\mathcal{O}}_4.
\]
We may prove that these are the only two conormal bundles whose moment
map image is equal to \( \overline{\mathcal{O}}_4 \) following the
argument we gave for $\overline{\mathcal{O}}_{3}$. 
Indeed, according to the table for type \( \mathrm{G}_2 \) on \cite[page
  427]{Carter85}, for any \( N_4 \in \mathcal{O}_4 \),  
equality~\eqref{eq:cardinalite-dimension} gives
\[
\#\left\{ S \in {}^\vee K \setminus {}^\vee G / {}^\vee B : \mu([T_S])
= \mathcal{O}_4 \right\} 
= \dim \mathrm{H}_{\mathrm{top}}^{\mathrm{BM}}(\mu^{-1}(N_4), \mathbb{Z})^{\{1\}} 
= \dim \phi_{\mathrm{triv}} = 2.
\]
By the process of elimination the conormal bundle of $S_5$ is taken to
$\overline{\mathcal{O}}_{5,2}$ by $\mu$.    
\end{proof}
Suppose $S$ is an $s_{\ell}$-horizontal ${^\vee}K$-orbit so that
(\ref{eq:horizontal}) describes the $s_{\ell}$-action on $[T_S]$.
Using Table \ref{eq:ccaction},
Theorem \ref{theo:Waction}.\ref{theo:Waction2}, equations \eqref{eq:nSiSj=1} and
\eqref{eq:nSiS9=2}, and Proposition \ref{prop:momentmapimage},
we can rule out some of the conormal bundles
from appearing on the right-hand side of (\ref{eq:horizontal}).
For example
$$s_{2}\cdot [T_{S_{8}}] = [T_{S_{8}}] + \sum_{S'} n_{S_{8},S'} \,
[T_{S'}]$$ 
where $S'$ is $s_{2}$-vertical and $\mu([T_{S'}]) \subset \mu([T_{S_{8}}]) =
\overline{\mathcal{O}}_{3}$.  It follows from
(\ref{eq:nilpotent-orbits}) that $\mu([T_{S'}])$ equals
$\overline{\mathcal{O}}_{0}$ or $\overline{\mathcal{O}}_{3}$.  By
Proposition \ref{prop:momentmapimage} the only $s_{2}$-vertical
${^\vee}K$-orbit satisfying this property is $S' = S_{9}$.  By equation
\eqref{eq:nSiS9=2} $n_{S_{8}, S_{9}} = 2$, so that
$$s_{2}\cdot [T_{S_{8}}] = [T_{S_8}]+2[T_{S_9}].$$
Arguing in this manner we obtain the following table. 

\begin{minipage}{\linewidth}
\medskip
\centering
\captionof{table}{$\left({}^{\vee}\mathrm{G}_{2}(\lambda),{}^{\vee}K(\lambda)\right)\, =\, {}(\mathrm{G}_{2},\mathrm{SL}_2\times_{\upmu_2}\mathrm{SL}_2)$}\label{eq:ccactionOnConormal}
\begin{tabular}{ | l | c | c | }
\hline 
 $i$ & $s_{1}\cdot [T_{S_i}]$ & $s_{2}\cdot [T_{S_i}]$ \\ \hline \hline		
   0 & $[T_{S_0}]+[T_{S_4}]$ & $[T_{S_0}]+[T_{S_3}]$\\	\hline
   1 & $-[T_{S_1}]$ & $[T_{S_1}]+[T_{S_3}]$ \\ \hline
   2 & $[T_{S_2}]+[T_{S_4}]$ & $-[T_{S_2}]$ \\ \hline
   3 & $n_{S_3,S_1}^{}[T_{S_1}]+[T_{S_3}]+[T_{S_6}]$ & $-[T_{S_3}]$  \\ \hline
   4 & $-[T_{S_4}]$ & $n_{S_4,S_2}^{}[T_{S_2}]+[T_{S_4}]+[T_{S_5}]$\\ \hline
   5 & $n_{S_5,S_4}^{}[T_{S_4}]+T_{S_5}+n_{S_5,S_6}^{}[T_{S_6}]+[T_{S_8}]$ & $-[T_{S_5}]$\\ \hline 
   6 & $-[T_{S_6}]$ & $[T_{S_6}]+[T_{S_7}]$\\ \hline
   7 & $n_{S_7,S_6}^{}[T_{S_6}]+[T_{S_7}]+n_{S_7,S_8}^{}[T_{S_8}]+2[T_{S_9}]$ & $-[T_{S_7}]$\\ \hline 
   8 & $-[T_{S_8}]$ & $[T_{S_8}]+2[T_{S_9}]$\\ \hline
   9 & $-[T_{S_9}]$ & $-[T_{S_9}]$\\ 
  \hline  
\end{tabular}
\bigskip 
  \end{minipage}

\section{Computing characteristic cycles for
  ${}^{\vee}\mathrm{G}_{2}$}\label{section:G2computations} 
In this section, we work within the framework of Section \ref{Korbits}
and complete the description of the characteristic cycles of the
irreducible perverse sheaves (\ref{eq:IrrObj1}).  In view of equations
(\ref{CC012}), (\ref{eq:CC3})-(\ref{eq:CC9}), it remains to treat the
cases of the 
irreducible perverse sheaves $P(\xi_i)$, associated to the parameters 
$$\xi_i=\left(S_i,\underline{\mathbb{C}}_{S_i}\right),\, i=5,6,7,8
\ \text{ and }\ \xi_i=\left(S_9,\mathcal{L}_{i}\right),\, i=10,11.$$
To simplify notation throughout this section, we denote
$$
P_i:=P(\xi_i),
$$
and continue to write
$$
[T_S]:=\overline{T_S^\ast X}\quad\text{and}\quad \chi_{S}^{}(P):=
\chi_{S}^{\mathrm{mic}}(P).
$$
The remaining characteristic cycles are given by the next theorem.
\begin{theo}\label{theo:CCs}
\begin{align*}
CC(P_{5})&= [T_{S_5}]\\
CC(P_{6})&= 2[T_{S_1}]\, +\, 
[T_{S_6}]\\
CC(P_{7})&= [T_{S_3}]\, +\, 
[T_{S_7}]\\
CC(P_8)&= [T_{S_4}]\, +\, 
[T_{S_6}]\, +\, [T_{S_8}]\\
CC(P_{10})&=[T_{S_1}]\, +\, 
[T_{S_6}]\, +\, 
[T_{S_8}]\, +\, [T_{S_9}]\\
CC(P_{11})&=[T_{S_2}]\, +\, [T_{S_7}]\, +\, [T_{S_9}].
\end{align*}
\end{theo}
The proof of this theorem will occupy the entire section and relies
heavily on the $W$-equivariance of the characteristic cycle map
\eqref{eq:W-equivariant}. Before we roll up our sleeves and delve into
the intricacies of the argument, let us appeal to some results from
the previous section to provide a finer 
description of the cycles in Theorem \ref{theo:CCs} beyond that of
Lemma \ref{lem:lemorbit}. 
These results impose additional restrictions on the conormal bundles
that may appear in the characteristic cycle of an irreducible perverse
sheaf. The first restriction arises from the relationship between the
$\tau$-invariant and the horizontal/vertical nature of the orbits, as
seen in Proposition \ref{prop:CC2}. The second restriction
depends on the structure of the Harish-Chandra cells and our
understanding of the associated varieties.

By Lemma~\ref{lem:lemorbit}, the microlocal multiplicity $\chi_S(P_i)$
can be nonzero only if $S$ is contained in the support of
$P_i$. Using an argument similar to the one leading to
equations~(\ref{eq:CC3})-(\ref{eq:CC9}), another
family of conormal bundles can be removed from $CC(P_i)$.
Indeed, Proposition~\ref{prop:CC2} implies that if $\alpha_{\ell} \in
\tau(P_{i})$  (\ref{eq:verticalP})
% $s_{\ell} P_{i}= -P_{i}$
and $S$ is $s_{\ell}$-horizontal, then 
$$
\chi_{S}(P_{i})=0.
$$
Therefore, from the description of the $\tau$-invariants
in~(\ref{eq:tau}) and of the vertical orbits
in~(\ref{eq:verticaldescription}),  
%the contrapositive of Proposition~\ref{prop:CC2} implies that
we deduce
\begin{align}\label{eq:CCP7CCP11}
%\begin{aligned}
CC(P_5)&=\chi^{}_{S_2}(P_5)[T_{S_2}]\, +\, \chi^{}_{S_3}(P_5) [T_{S_3}]\, +\, [T_{S_5}]\nonumber\\
CC(P_6)&=\chi^{}_{S_1}(P_6)[T_{S_1}]\, +\, \chi^{}_{S_4}(P_6)[T_{S_4}]\, +\, [T_{S_6}]\nonumber\\
CC(P_7)&=\chi^{}_{S_2}(P_7)[T_{S_2}]\, +\, \chi^{}_{S_3}(P_7) [T_{S_3}]\, +\, 
\chi^{}_{S_5}(P_7)[T_{S_5}]
\, +\, [T_{S_7}]\\
CC(P_8)&=\chi^{}_{S_1}(P_8)[T_{S_1}]\, +\, \chi^{}_{S_4}(P_8)[T_{S_4}]\, +\, 
\chi^{}_{S_6}(P_8)[T_{S_6}]\, +\, [T^{}_{S_8}]\nonumber\\
CC(P_{10})&=\chi^{}_{S_1}(P_{10})[T_{S_1}]\, +\, \chi^{}_{S_4}(P_{10}) 
[T_{S_4}]\,+\, 
\chi^{}_{S_6}(P_{10})[T_{S_6}]
\, +\, \chi^{}_{S_8}(P_{10}) [T_{S_8}]\, +\, [T_{S_9}]\nonumber\\
CC(P_{11})&=\chi^{}_{S_2}(P_{11})[T_{S_2}]\, +\,
\chi^{}_{S_3}(P_{11})[T_{S_3}]\, +\, \chi^{}_{S_5}(P_{11})[T_{S_5}]\,
+\, \chi^{}_{S_7}(P_{11})[T_{S_7}]+[T_{S_9}].\nonumber 
%\end{aligned}
\end{align}
These formulas can be further simplified by appealing to the structure
of the Harish-Chandra cells described in Proposition
\ref{prop:HCcells}, together with the image of the moment map given in
Proposition \ref{prop:momentmapimage}.  
More precisely,  since $\mu(\mathrm{Ch}(P_i)) = \mathrm{AV}(P_i)$ 
the conormal bundles contributing to $CC(P_i)$ must map under $\mu$ to
subsets of $\mathrm{AV}(P_i)$.   There are two simplifying conclusions
to be made from this observation.

The first conclusion is a result of the discussion surrounding
(\ref{eq:nilpotent-orbits}) and reads as follows.
% Any conormal bundle whose moment
%map image is the closure of a nilpotent orbit whose
%${}^{\vee}\mathrm{G}_{2}$-saturation is a special nilpotent orbit
%distinct from $\mathrm{AV}(P_i)$ cannot appear in $CC(P_i)$. 
If $\mu([T_{S}]) = \overline{\mathcal{O}}$ and
${}^{\vee}\mathrm{G}_{2} \cdot \mathcal{O}$ is special with
${}^{\vee}\mathrm{G}_{2} \cdot \mathcal{O} \neq   
{}^{\vee}\mathrm{G}_{2} \cdot \mathrm{AV}(P_{i})$ then $\chi_{S}(P_{i}) =
0$.
For example, by Proposition~\ref{prop:momentmapimage}, $\mu([T_{S_5}])
= \overline{\mathcal{O}}_{5,2}$ and ${^\vee}\mathrm{G}_{2}\cdot
\mathcal{O}_{5,2} = \mathcal{O}_{5,{^\vee}\mathrm{G}_{2}}$ is special. Therefore, by
Proposition~\ref{prop:Cells-AV}, the conormal bundle
$\chi_{S_{5}}(P_{i}) \neq 0$  only if $P_i \in \mathrm{cell}_2$. 
The same holds for $[T_{S_2}]$ and $[T_{S_4}]$, the other two conormal
bundles whose image under the moment map is also
$\overline{\mathcal{O}}_{5,2}$. 
In other words, for any $P_i \notin \mathrm{cell}_2$, we have 
$$
0=\chi_{S_2}(P_i)=\chi_{S_4}(P_i)=\chi_{S_5}(P_i).
$$
The same reasoning applies to any conormal bundle whose image under
the moment map is $\overline{\mathcal{O}}_{5,1}$. Therefore, the
formulas in (\ref{eq:CCP7CCP11}) simplify to 
\begin{align}\label{eq:CCP5CCP11V2}
\begin{aligned}
CC(P_5)&=\chi^{}_{S_2}(P_5)[T_{S_2}]\, +\, 
[T_{S_5}]\\
CC(P_6)&=\chi^{}_{S_1}(P_6)[T_{S_1}]\, +\, 
[T_{S_6}]\\
CC(P_7)&= \chi^{}_{S_3}(P_7) [T_{S_3}]
\, +\, [T_{S_7}]\\
CC(P_8)&= \chi^{}_{S_4}(P_8)[T_{S_4}]\, +\, 
\chi^{}_{S_6}(P_8)[T_{S_6}]\, +\, [T^{}_{S_8}]\\
CC(P_{10})&=\chi^{}_{S_1}(P_{10})[T_{S_1}]\,+\, 
\chi^{}_{S_6}(P_{10})[T_{S_6}]
\, +\, \chi^{}_{S_8}(P_{10}) [T_{S_8}]\, +\, [T_{S_9}]\\
CC(P_{11})&=\chi^{}_{S_2}(P_{11})[T_{S_2}]\, +\,
\chi^{}_{S_5}(P_{11})[T_{S_5}]\, +\,
\chi^{}_{S_7}(P_{11})[T_{S_7}]+[T_{S_9}]. 
\end{aligned}
\end{align}

The second conclusion to be drawn from the identity $\mu(\mathrm{Ch}(P_i)) =
\mathrm{AV}(P_i)$  is that if $\mathrm{AV}(P_{i}) = \overline{\mathcal{O}}$ then
at least one $[T_{S}]$ appearing in $CC(P_{i})$ must satisfy
$\mu([T_{S}]) = \mathcal{O}$.  Applying this property to
(\ref{eq:CCP5CCP11V2}) with the aid of
Propositions~\ref{prop:Cells-AV} and~\ref{prop:momentmapimage}, we
conclude that the microlocal multiplicities 
\begin{align}\label{eq:nonzeromicro} 
\chi_{S_1}(P_6),\quad \chi_{S_3}(P_7),\quad \chi_{S_4}(P_8),\quad 
\chi_{S_1}(P_{10}),
\end{align}
are all nonzero, and that at least one of two values
\[
\chi_{S_2}(P_{11}) \quad \text{and} \quad \chi_{S_2}(P_{5})
\]
is nonzero.

To compute the remaining coefficients, we use the $W$-equivariance of
the characteristic cycle map $CC(\cdot)$, as described in
(\ref{eq:W-equivariant}), to establish relationships among the
microlocal multiplicities. We first consider $CC(P_i)$ for $i = 7, 8,
10,$ and $11$, and postpone the computation of $CC(P_5)$ and $CC(P_6)$
to the end of this section. 

Let us start by looking at row 11 of Table \ref{eq:ccaction}. 
%Thus and from 
Equality (\ref{eq:W-equivariant}) implies
\begin{align}\label{eq:s2onP11}
s_{1}\cdot CC(P_{11})=
CC(s_{1}\cdot P_{11})
=CC(P_{11})\, +\, CC(P_8).
\end{align}
By Equality~(\ref{eq:CCP5CCP11V2}) and Table~\ref{eq:ccactionOnConormal}, the left-hand side of the previous equation is 
equal to the sum of the right-hand sides of the following three equations
\begin{align*}
  \chi_{S_2}^{\mathrm{}}(P_{11})(s_{1}\cdot[T_{S_2}])&=\mathbf{\chi_{S_2}^{}(P_{11})[T_{S_2}]}\, +\,  \chi_{S_2}^{}(P_{11})[T_{S_4}]\nonumber\\
\chi_{S_5}^{}(P_{11})(s_{1}\cdot[T_{S_5}])\, &=\, \mathbf{\chi_{S_5}^{}(P_{11})[T_{S_5}]}\, +\,  \chi_{S_5}^{}(P_{11}) \left(n_{S_5,S_4}[T_{S_4}]\, +\, \nonumber
n_{S_5,S_6}[T_{S_6}]\, +\, [T_{S_8}]\right)\\%\label{eq:nS6S5}\\
 \chi_{S_7}^{}(P_{11})(s_{1}\cdot [T_{S_7}])\, &=\,  \mathbf{\chi_{S_7}^{}(P_{11})
[T_{S_7}]}\, +\,  \chi_{S_7}^{}(P_{11})\left(n_{S_7,S_6}[T_{S_6}]\, +\, n_{S_7,S_8} [T_{S_8}]\right)\\ 
&\qquad\qquad\qquad\qquad\qquad\qquad\qquad\qquad\qquad +\, \mathbf{2\cdot\chi_{S_7}^{}(P_{11}) [T_{S_9}]}\\
s_{1}\cdot [T_{S_9}]\, &=\, \mathbf{-[T_{S_9}]}.\nonumber
\end{align*}
In bold, we have written the constituents of $CC(P_{11})$, which
cancel with the right-hand side of \eqref{eq:s2onP11}. 
Comparing with the right-hand side of (\ref{eq:s2onP11}), we conclude 
\begin{align}
\chi^{}_{S_4}(P_{8})&=\chi^{}_{S_2}(P_{11})
\, +\, \chi^{}_{S_5}(P_{11})n_{S_5,S_4}\label{eq:S11S8a} \\
\chi^{}_{S_6}(P_{8})&=
\chi^{}_{S_5}(P_{11}) n_{S_5,S_6}
\, +\, 
\chi^{}_{S_7}(P_{11})n_{S_7,S_6} \label{eq:S11S8b}\\
1&=\chi^{}_{S_5}(P_{11})%n_{S_5,S_8}
\,  +\,   \chi^{}_{S_7}(P_{11})n_{S_7,S_8}\label{eq:S11S8c} \\
1&=2\cdot \chi^{}_{S_7}(P_{11})-1.\label{eq:S11S8d} 
\end{align}
Since by Table \ref{eq:ccaction} row 11,
the action of $s_{2}$ 
on $P_{10}$ is given by $P_{10}+P_{7}$, 
the exact same argument to the one 
leading to equalities (\ref{eq:S11S8a})-(\ref{eq:S11S8d}),
yields the following list of equations
\begin{align}
\chi^{}_{S_3}(P_7)\, &=\, 
\chi^{}_{S_1}(P_{10})\label{eq:S10S7a}\\
1&=\chi^{}_{S_6}(P_{10})\label{eq:S10S7b}\\
1&=2\cdot \chi^{}_{S_8}(P_{10}) -1.\label{eq:S10S7c}
\end{align}
Note that four equations appear in
  (\ref{eq:S11S8a})- (\ref{eq:S11S8d}), and  three equations appear in
(\ref{eq:S10S7a})-(\ref{eq:S10S7c}).  The fewer and simpler
  equations in the latter case are due to the image of the moment map
  and its relationship with the $W$-action on 
the conormal bundles, as described in
Table~\ref{eq:ccactionOnConormal}. For example, since 
\[
\mu([T_{S_8}]) = \overline{\mathcal{O}}_3 \subset
\overline{\mathcal{O}}_4 = \mu([T_{S_7}]), 
\]
it follows from~\eqref{eq:horizontal} that the conormal bundle
$[T_{S_7}]$ does not appear in $s_{1} \cdot [T_{S_8}]$, while 
$[T_{S_8}]$ may appear in $s_{2} \cdot [T_{S_7}]$. This results in the
asymmetry between Equations~\eqref{eq:S11S8c} and~\eqref{eq:S10S7b}. 
%Therefore, $n_{S_8, S_7} = 0$, and the product $\chi_{S_8}^{}(P_{10})
%n_{S_8, S_7}$ disappears from Equation~\eqref{eq:S10S7b}. 

Looking at equalities (\ref{eq:S11S8d}) and (\ref{eq:S10S7c}), we deduce that
\begin{equation}\label{eq:S7P11S8P10}
\chi^{}_{S_7}(P_{11})=\chi^{}_{S_8}(P_{10})=1.
\end{equation}
Above, we applied $s_{2}$ to $P_{10}$ and $s_{1}$ to $P_{11}$. 
We now apply 
$s_{1}$ to $P_7$, and $s_{2}$ to $P_8$.
Once again, using the $W$-equivariance (\ref{eq:W-equivariant}) and Table
(\ref{eq:ccaction}), we write 
\begin{align}
s_{1}\cdot CC(P_7)\ &=\ CC(P_6)\, +\, CC(P_7)\, +\, CC(P_9)\, +\, CC(P_{10}),
\label{alpha0toB7}
\\
s_{2}\cdot CC(P_8)\ &=\ CC(P_5)\, +\, CC(P_8)\, +\, CC(P_9)\, +\, CC(P_{11}).
\label{alpha1toB8}
\end{align}
According to \eqref{eq:CCP5CCP11V2}, the left-hand side of
(\ref{alpha0toB7})  contains a  multiple of 
$s_{1} \cdot [T_{S_3}]$.  We can say more about this term.  By
\eqref{eq:CC3}, we have $CC(P_3)=[T_{S_3}]$.  The $W$-equivariance
(\ref{eq:W-equivariant}) combined with row three of Table
\ref{eq:ccaction} therefore implies 
\begin{align*}
s_{1}\cdot[T_{S_3}]= s_{1}\cdot CC(P_3)&=CC(s_1\cdot P_3)\\
&=CC(P_3)\, +\, CC(P_1)\, +\, CC(P_6)\\
&=[T_{S_3}]\, +\, [T_{S_1}]\, +\,  \left(\chi_{S_1}(P_6)[T_{S_1}]+[T_{S_6}]\right).
\end{align*} 
The same argument with $P_4$ and $s_{2}$ in place of 
$P_3$ and $s_{1}$, can be used to obtain
\begin{align}\label{eq:TS4}
s_{2}\cdot[T_{S_4}]=[T_{S_4}]\, +\,  [T_{S_2}]\, +\, 
\left(\chi_{S_2}^{}(P_{5})[T_{S_2}]\, +\, [T_{S_5}]\right).
\end{align}
Consequently, by \eqref{eq:CCP5CCP11V2} and
Table~\ref{eq:ccactionOnConormal}, the left-hand side of
(\ref{alpha0toB7}) is the sum of the  
right-hand sides of the following two equations 
\begin{align}
\chi_{S_3}^{}(P_7)(s_{1}\cdot [T_{S_3}])\, &=\,
\mathbf{\chi_{S_3}^{}(P_7)[T_{S_3}]}+  
 \chi_{S_3}^{}(P_7)[T_{S_1}]\, +\nonumber\\
 &\qquad\qquad\qquad\qquad
 \chi_{S_3}^{}(P_{7})\left(\chi_{S_1}^{}(P_{6})[T_{S_1}]\, +\,
     [T_{S_6}]\right)\label{eq:alpha2S3}\\ 
s_{1}\cdot [T_{S_7}]\, &=\, \mathbf{[T_{S_7}]}\, +\, 
n_{S_7,S_6}  [T_{S_6}]\, +\, n_{S_7,S_8} [T_{S_8}]\, +\, 2 [T_{S_9}]. \label{eq:nS7S8}
\end{align}
The summands of $CC(P_7)$ are indicated in bold.

Similarly, the left-hand side of (\ref{alpha1toB8}) is 
equal to the sum of the 
right-hand sides of the following list of equations
\begin{align}
 \chi_{S_4}^{}(P_8)(s_{2}\cdot [T_{S_4}])\, &=\,
 \mathbf{\chi_{S_4}^{}(P_8)[T_{S_4}]}+  
 \chi_{S_4}^{}(P_8)[T_{S_2}]\, +\, \nonumber\\
 &\qquad\qquad\qquad\qquad
 \chi_{S_4}^{}(P_{8})\left(\chi_{S_2}^{}(P_{5})[T_{S_2}]\, +\,
     [T_{S_5}]\right)\label{eq:alpha1S4}\\ 
 \chi_{S_6}^{}(P_{8})(s_{2}\cdot[T_{S_6}])\, &=\,
 \mathbf{\chi_{S_6}^{}(P_{8})[T_{S_6}]}\, +\,   
\chi_{S_6}^{}(P_{8}) [T_{S_7}]\label{eq:nS6S5}\\
\nonumber s_{2}\cdot [T_{S_8}]\, &=\,  
 \mathbf{[T_{S_8}]}%\, +\,  n_{S_8,S_7} [T_{S_7}]
 \, +\, 2 [T_{S_9}]. 
%\label{eq:alpha1S8}
\end{align}
The summands of $CC(P_8)$ are indicated in bold.

Now, by \eqref{eq:CC9}, \eqref{eq:CCP5CCP11V2} and
\eqref{eq:S7P11S8P10}, $[T_{S_8}]$ appears 
on the right-hand-side of (\ref{alpha0toB7}) with
multiplicity given by $\chi^{}_{S_8}(P_{10})=1$. Meanwhile, from
\eqref{eq:alpha2S3} and \eqref{eq:nS7S8}, the conormal bundle $[T_{S_8}]$
appears on the left-hand side of \eqref{alpha0toB7} with multiplicity  
\(n_{S_7,S_8}\). Hence,
$$
n_{S_7,S_8}=1,
$$
and then from  (\ref{eq:S11S8c}) and \eqref{eq:S7P11S8P10}, that
\begin{equation} \label{eq:S5P11}
\chi_{S_5}^{}(P_{11})=0.
\end{equation}
To finish with equation  \eqref{alpha0toB7}, we notice that:
\begin{enumerate}
\item $[T_{S_1}]$ appears on the right-hand-side of
  (\ref{alpha0toB7}), multiplied by $\chi_{S_1}^{}(P_{10})\, +\,
  {\chi_{S_1}^{}(P_{6})}$. Hence, by (\ref{eq:alpha2S3}) and 
  (\ref{eq:S10S7a}), we have 
$$
\chi^{}_{S_3}(P_7)(1+\chi^{}_{S_1}(P_6))
=\chi_{S_1}^{}(P_{10})\, +\, {\chi_{S_1}^{}(P_{6})}\quad\Longrightarrow\quad
\chi^{}_{S_3}(P_7)\chi^{}_{S_1}(P_6)
=\chi_{S_1}^{}(P_{6}).$$  
By (\ref{eq:nonzeromicro}) the microlocal multiplicity
$\chi^{}_{S_1}(P_6)$ is nonzero and so
$1 =\chi_{S_3}^{}(P_{7})=\chi_{S_1}^{}(P_{10})$. 
\item $[T_{S_6}]$ appears on the right-hand-side of
  (\ref{alpha0toB7}), multiplied by $1\, +\,
  {\chi_{S_6}^{}(P_{10})}$. Hence, by (\ref{eq:S10S7b}),  
(\ref{eq:alpha2S3}), (\ref{eq:nS7S8}) and the previous point, we have
\begin{align*} 
\chi^{}_{S_3}(P_7)\, +\, n_{S_7,S_6}&=1+\chi_{S_6}^{}(P_{10})=2
\quad\Longrightarrow\quad n_{S_7,S_6}=1,
\end{align*}
\end{enumerate}
Applying the same reasoning to equation (\ref{alpha1toB8}), we deduce
the following: 
\begin{enumerate}
\item By  (\ref{eq:S5P11}), the conormal bundle $[T_{S_5}]$ appears
  only once on the right-hand side of
  (\ref{alpha1toB8}). Consequently, equation (\ref{eq:alpha1S4}) tells
  us that
$$\chi^{}_{S_4}(P_8)=1,$$
and the from  (\ref{eq:S11S8a}) and \eqref{eq:S5P11}, we obtain
$\chi^{}_{S_2}(P_{11})=1.$ 
\item By (\ref{eq:S7P11S8P10}), the conormal bundle $[T_{S_7}]$
  appears only once on  the right-hand-side of (\ref{alpha1toB8}).
By  (\ref{eq:nS6S5}), we therefore obtain
$$\chi_{S_6}(P_8)=1.$$
\end{enumerate}
Putting all the above equations together, we conclude that:
\begin{align*}
%\begin{aligned}\label{eq:CC7-11}
CC(P_{7})&= [T_{S_3}]\, +\, 
[T_{S_7}]\\
CC(P_8)&= [T_{S_4}]\, +\, 
[T_{S_6}]\, +\, [T_{S_8}]\\
CC(P_{10})&=[T_{S_1}]\, +\, 
[T_{S_6}]\, +\, 
[T_{S_8}]\, +\, [T_{S_9}]\\
CC(P_{11})&=[T_{S_2}]\, +\, [T_{S_7}]\, +\, [T_{S_9}].
%\end{aligned}
\end{align*}

To complete our computations for Theorem \ref{theo:CCs}, it remains to
determine $CC(P_5)$ and $CC(P_6)$. Recall from (\ref{eq:CCP5CCP11V2}) that 
\begin{align*}
CC(P_5)&=\chi^{}_{S_2}(P_5)[T_{S_2}]\, +\, 
[T_{S_5}],\\
CC(P_6)&=\chi^{}_{S_1}(P_6)[T_{S_1}]\, +\, 
[T_{S_6}].
\end{align*}
Using the $W$-equivariance (\ref{eq:W-equivariant})
and tables (\ref{eq:ccaction}) and (\ref{eq:ccactionOnConormal}),  
we can write
\begin{align*}
CC(P_6)\, +\, CC(P_3)\, +\, CC(P_7)\ 
&=\ s_{2}\cdot CC(P_6)\\ 
&=\chi_{S_1}(P_6)([T_{S_1}]\, +\, [T_{S_3}])\, + \,[T_{S_6}]\, +\,
  [T_{S_7}]\\
&=\ CC(P_6) \, +\, CC(P_3) \, +\,
 \left(\chi_{S_1}(P_6)-1\right) [T_{S_3}] \, +\,  [T_{S_7}]\\
CC(P_5)\, +\, CC(P_4)\, +\, CC(P_8)\ 
&=\ s_{1}\cdot CC(P_5)\\ 
&=\chi_{S_2}(P_5)([T_{S_2}]\, +\, [T_{S_4}])\, + \,
n_{S_5,S_4}[T_{S_4}]\\
&\qquad
 \, +\, [T_{S_5}]\, +\, n_{S_5,S_6}[T_{S_6}]\ +\ 
  [T_{S_8}]\\
&=\ CC(P_5) \, +\, CC(P_4) \, +\,
  \left(\chi_{S_2}(P_5)+n_{S_5,S_4}-1\right) [T_{S_4}]\\ 
 & \qquad \, +\, 
 n_{S_5,S_6}[T_{S_6}] 
 \, +\,  [T_{S_8}].
\end{align*}
Since  $\chi_{S_3}(P_7)=\chi_{S_4}(P_8)=\chi_{S_6}(P_8)=1$, we conclude
\begin{align*}
\chi_{S_1}(P_6)-1&=\chi_{S_3}(P_7)=1\quad\Longrightarrow\quad \chi_{S_1}(P_6)=2\\
\chi_{S_2}(P_5)+n_{S_5,S_4}-1&=\chi_{S_4}(P_8)=1\quad\Longrightarrow\quad \chi_{S_2}(P_5)+n_{S_5,S_4}=2\\
n_{S_5,S_6}&=\chi_{S_6}(P_8)=1.
\end{align*}
Therefore
\begin{align}\label{eq:TS4TS5}
\begin{aligned}
s_{2}\cdot [T_{S_4}]&=
(1+\chi_{S_2}(P_5))[T_{S_2}]+[T_{S_4}]
+[T_{S_5}]\\
s_{1}\cdot [T_{S_5}]&=
(2-\chi_{S_2}(P_5))[T_{S_4}]
+[T_{S_5}]+[T_{S_6}] 
 +  [T_{S_8}],
 \end{aligned}
\end{align}
where the first equation is simply a rearrangement of~\eqref{eq:TS4}, and
\begin{align}\label{eq:CCP5CCP6}
CC(P_5)&=\chi_{S_{2}}^{}(P_5)
[T_{S_2}]\ +\, [T_{S_5}],\ \text{ where }\ \chi_{S_{2}}^{}(P_5)\leq 2,\\
CC(P_6)&=
2[T_{S_1}]\, +\, [T_{S_6}].\nonumber
\end{align}
In the following Lemma, we prove that $\chi_{S_{2}}^{}(P_5)=0$, which
in turn implies that $n_{S_5,S_4}=2$. 
This lemma concludes the proof of Theorem \ref{theo:CCs}. 
\begin{lem}\label{lem=chiS5}
We have
$$
\chi_{S_{2}}^{}(P_5)=0.
$$
\end{lem}
\begin{proof}
We use the theory presented in Section \ref{WactionL}.  Let $N_{5,2}
\in \mathcal{O}_{5,2}$. From 
% be the element denoted by $f_3$ in \cite[Table~1]{Samples}.  
column three of \cite[Table~1]{Samples}, we have
\[
A_K(N_{5,2}) \cong \mathbb{Z}_2
\]
and equation~\eqref{eq:Springer-representation2} tells us that
\begin{equation}\label{eq:M52}
M(\mathcal{O}_{5,2}) \, \cong\,
\mathrm{H}_{\mathrm{top}}^{\mathrm{BM}}(\mu^{-1}(N_{5,2}), \mathbb{Z})^{\mathbb{Z}_2}. 
\end{equation}
By the table for type \( \mathrm{G}_{2} \) on
\cite[page 427]{Carter85}, we obtain $$A_K(N_{5,2}) \subset
A_G(N_{5,2}) \cong  \mathcal{S}_3,$$
the symmetric group of order six.   
The irreducible \(\mathbb{Z}[
\mathcal{S}_3] \)-modules are indexed by the partitions of 3:  
\[
\psi_{(3)} = \mathrm{triv}, \quad 
\psi_{(1,1,1)} = \mathrm{sgn}, \quad 
\psi_{(1,2)},
\]
where \( \mathrm{triv} \) and \( \mathrm{sgn} \) denote the trivial
and sign modules, respectively, and \( \psi_{(1,2)} \) is the
unique two-dimensional irreducible module. 
According to the table for type \( \mathrm{G}_{2} \) on \cite[page
  427]{Carter85}, the Springer correspondence
(\ref{eq:Springer-representation}) assigns to the pair \( 
\left( \mu^{-1}(N_{5,2}), \psi_{(3)} \right) \) a two-dimensional
irreducible character \( \phi_{\psi_{(3)}} \) of \( W \), and to \(
\left( \mu^{-1}(N_{5,2}), \psi_{(1,2)} \right) \) a one-dimensional
irreducible character \( \phi_{\psi_{(1,2)}} \) of \( W \). The pair
\( \left( \mu^{-1}(N_{5,2}), \psi_{(1,1,1)} \right) \), however, is
not associated to any irreducible character of \( W \) via the
Springer correspondence. In summary,  the top Borel-Moore homology  
decomposes as the direct sum
$$
\mathrm{H}_{\mathrm{top}}^{\mathrm{BM}}(\mu^{-1}(N_{5,2}), \mathbb{Z})=
\left( \phi_{\psi_{(3)}}\otimes {\psi_{(3)}} \right) \oplus \left(
\phi_{\psi_{(1,2)}}\otimes {\psi_{(1,2)}} \right).
$$
Moreover, since the restriction
$\left.\psi_{(1,2)}\right|_{\mathbb{Z}_2}^{}=\mathrm{triv}\oplus\mathrm{sgn}$, 
by \eqref{eq:M52}, we have 
\begin{align}\label{eq:isoV1}
M(\mathcal{O}_{5,2})\, &\cong\,  \left( \left( 
\phi_{\psi_{(3)}}\otimes {\psi_{(3)}} \right) \oplus \left( \phi_{\psi_{(1,2)}}\otimes
    {\psi_{(1,2)}} \right) \right)^{\mathbb{Z}_2}\nonumber\\ 
\ & \cong \ \phi_{\psi_{(3)}}\oplus \phi_{\psi_{(1,2)}}
\end{align}
as $\mathbb{Z}(W)$-modules.  
Consider the projection \( q \) from the \(\mathbb{Z}\)-module
\eqref{eq:Z-module}  to the quotient \eqref{eq:Z-module-quotient}, and
let \( \sigma \) denote the \(\mathbb{Z}[ W] \)-module 
on~\eqref{eq:Z-module-quotient} induced by the \( W \)-action
on~\eqref{eq:Z-module}.   
By Proposition~\ref{prop:momentmapimage}, we have
\[
\mu([T_{S_2}]) = \mu([T_{S_4}]) = \mu([T_{S_5}]) = \overline{\mathcal{O}}_{5,2},
\]
and so  from~\eqref{eq:isoV1} we deduce that
\begin{align}\label{eq:isoV2}
\left( \langle q([T_{S_2}]), q([T_{S_4}]), q([T_{S_5}]) \rangle, \sigma \right)
\cong \phi_{\psi_{(3)}} \oplus \phi_{\psi_{(1,2)}}.
\end{align}
Additionally,  Table \ref{eq:ccactionOnConormal} and \eqref{eq:TS4TS5} imply
\begin{align}
\begin{aligned}\label{eq:sigma-action}
\sigma(s_{1})\cdot q([T_{S_2}])&=q([T_{S_{2}}])+q([T_{S_{4}}])\\
\sigma(s_{2})\cdot q([T_{S_2}])&=-q([T_{S_{2}}])\\
\sigma(s_{1})\cdot q([T_{S_4}])&=-q([T_{S_4}]) \\
\sigma(s_{2})\cdot q([T_{S_4}])&=
(1+\chi_{S_2}(P_5))q([T_{S_2}])+q([T_{S_4}])
+q([T_{S_5}])\\
\sigma(s_{1})\cdot q([T_{S_5}])&=
(2-\chi_{S_2}(P_5))q([T_{S_4}])+q([T_{S_5}])\\
\sigma(s_{2})\cdot q([T_{S_5}])&=-q([T_{S_5}]). 
\end{aligned}
\end{align}
Now, by~\eqref{eq:isoV2}, there exists a $\mathbb{Z}$-linear combination with
of  
\( q([T_{S_2}]) \), \( q([T_{S_4}]) \), and \( q([T_{S_5}]) \) that
spans a one-dimensional  
\( W \)-invariant subspace isomorphic to \( \phi_{\psi_{(1,2)}} \). 
Examining the action defined in~\eqref{eq:sigma-action}, it is
straightforward to verify that 
\begin{align*}
\sigma(s_{1}) \cdot \left( (\chi_{S_2}(P_5)-2)\, q([T_{S_2}]) +
q([T_{S_5}]) \right)  
&= \left( (2 - \chi_{S_2}(P_5))\, q([T_{S_2}]) + q([T_{S_5}]) \right),\\
\sigma(s_{2}) \cdot \left( (\chi_{S_2}(P_5)-2)\, q([T_{S_2}]) +
q([T_{S_5}]) \right)  
&= -\left( (2 - \chi_{S_2}(P_5))\, q([T_{S_2}]) + q([T_{S_5}]) \right). 
\end{align*}
Therefore, the vector 
\[
v = (\chi_{S_2}(P_5)-2)\, q([T_{S_2}]) + q([T_{S_5}])
\]
spans a rank one \( \mathbb{Z}[W] \)-module on which \( s_{1} \)
acts trivially and \( s_{2} \) acts by multiplication by \(-1\). This
module corresponds to \( \phi_{\psi_{(1,2)}} \). Returning to
\eqref{eq:isoV2}, we obtain the decomposition 
\[
\left( \langle q([T_{S_2}]), q([T_{S_4}]), q([T_{S_5}]) \rangle, \sigma \right)
\cong \left( \langle v \rangle, \sigma \right) \oplus \phi_{\psi_{(3)}}.
\]
Following~\eqref{eq:CCP5CCP6}, the only possible values for \(
\chi_{S_2}^{}(P_5) \) are \( 0 \), \( 1 \), or \( 2 \). We will show
that any value other than \( 0 \) leads to a contradiction. 
Suppose \( \chi_{S_2}^{}(P_5) = 2 \) so that
\[
\phi_{\psi_{(1,2)}} \cong (\langle v \rangle, \sigma ) =  \left(
\langle q([T_{S_5}]) \rangle, \sigma \right). 
\]
Any complement of \( \langle q([T_{S_5}]) \rangle \) is of the form
\begin{equation}\label{eq:complement0}
\langle q([T_{S_2}]) + n\, q([T_{S_5}]),\ q([T_{S_4}]) + m\,
q([T_{S_5}]) \rangle \quad \text{with } n, m \in \mathbb{Z}. 
\end{equation}
Let us show that the complement cannot be \( W \)-invariant. 
Indeed, 
 by~\eqref{eq:sigma-action}, the element
\[
\sigma(s_{2}) \cdot \left( q([T_{S_4}]) + m\, q([T_{S_5}]) \right)
\]
only lies in the submodule~\eqref{eq:complement0}  when \(
3n + 2m = 1 \), and if this condition is satisfied, then 
\[
\sigma(s_{1}) \cdot \left( q([T_{S_4}]) + m\, q([T_{S_5}]) \right)
\]
does not lie in~\eqref{eq:complement0}.

Similarly, if \( \chi_{S_2}(P_5) = 1 \), then
\[
\phi_{\psi_{(1,2)}} \cong 
\left( \langle -q([T_{S_2}]) + q([T_{S_5}]) \rangle, \sigma \right).
\]
Any  complement of this submodule is of the form
\begin{align}\label{eq:complement}
\left\langle 
q([T_{S_2}]) + q([T_{S_5}]) + n\, z,\ 
q([T_{S_4}]) + m\, z 
\right\rangle, \quad \text{with } n, m \in \mathbb{Z},
\end{align}
where \( z = -q([T_{S_2}]) + q([T_{S_5}]) \).
By~\eqref{eq:sigma-action}, we have   
\begin{align*}
\sigma(s_{2})\cdot (q([T_{S_4}])+mv)&=
2q([T_{S_2}])+q([T_{S_4}]) +q([T_{S_5}])-m z\\
&=\frac{3}{2}\left(q([T_{S_2}])+q([T_{S_5}])\right)+q([T_{S_4}])
-\frac{1+2m}{2} z.
%\left(-q([T_{S_2}])+q([T_{S_5}])\right)
\end{align*} 
and this element does not belong to the $\mathbb{Z}$-module
(\ref{eq:complement}).    
We are thus left with the remaining possibility that
$\chi_{S_2}^{}(P_5)=0$.  This proves Lemma \ref{lem=chiS5}.
Additionally, we observe that  equations \eqref{eq:TS4TS5} reduce to
\begin{align*}
s_{2}\cdot [T_{S_4}]&=
[T_{S_2}]+[T_{S_4}]
+[T_{S_5}]\\
s_{1}\cdot [T_{S_5}]&=
2[T_{S_4}]
+[T_{S_5}]+[T_{S_6}] 
 +  [T_{S_8}]
\end{align*}
yielding
\[
\phi_{\psi_{(1,2)}}\oplus \phi_{\psi_{(3)}} \cong 
\left( \langle -2q([T_{S_2}]) + q([T_{S_5}]) \rangle, \sigma \right)\oplus
\left( \langle q([T_{S_2}]) + q([T_{S_5}]), q([T_{S_4}]) \rangle, \sigma \right)
\]

\end{proof}
In computing these characteristic cycles, we have determined all the
coefficients appearing in equation  \eqref{eq:horizontal} of Theorem
\ref{theo:Waction}. Table  
\ref{eq:ccactionOnConormal}, which
describes the $W$-action of the simple reflections on the conormal
bundles to the ${}^{\vee}K$-orbits in the flag variety $X =
{}^{\vee}\mathrm{G}_{2} / {}^{\vee}B$, can now be completed as follows.

\begin{minipage}{\linewidth}
\medskip
\centering
\captionof{table}{$\left({}^{\vee}\mathrm{G}_{2}(\lambda),
  {}^{\vee}K(\lambda)\right)\,   =\, 
  {}(\mathrm{G}_{2},\mathrm{SL}_2\times_{\upmu_2}\mathrm{SL}_2)$} 
\label{tab:Wactiononconormal}   
\begin{tabular}{ | l | c | c | }
\hline 
 $i$ & $s_{1}\cdot [T_{S_i}]$ & $s_{2}\cdot [T_{S_i}]$ \\ \hline	\hline	
   0 & $[T_{S_0}]+[T_{S_4}]$ & $[T_{S_0}]+[T_{S_3}]$\\	\hline
   1 & $-[T_{S_1}]$ & $[T_{S_1}]+[T_{S_3}]$ \\ \hline
   2 & $[T_{S_2}]+[T_{S_4}]$ & $-[T_{S_2}]$ \\ \hline
   3 & $3[T_{S_1}]+[T_{S_3}]+[T_{S_6}]$ & $-[T_{S_3}]$ \\ \hline
   4 & $-[T_{S_4}]$ & $[T_{S_2}]+[T_{S_4}]+[T_{S_5}]$\\ \hline
   5 & $2[T_{S_4}]+[T_{S_5}]+[T_{S_6}]+[T_{S_8}]$ & $-[T_{S_5}]$\\ \hline
   6 & $-[T_{S_6}]$ & $[T_{S_6}]+[T_{S_7}]$\\ \hline
   7 & $[T_{S_6}]+[T_{S_7}]+[T_{S_8}]+2[T_{S_9}]$ & $-[T_{S_7}]$\\ \hline
   8 & $-[T_{S_8}]$ & $[T_{S_8}]+2[T_{S_9}]$\\ \hline
   9 & $-[T_{S_9}]$ & $-[T_{S_9}]$\\ 
 \hline  
\end{tabular}
\bigskip
\end{minipage}

In~\eqref{eq:sigma-action}, using the identification
in~\eqref{eq:Springer-representation2}, we described the action of $W$
on a basis of $\mathrm{H}_{\mathrm{top}}^{\mathrm{BM}}(\mu^{-1}(N),
\mathbb{Z})^{A_K(N)}$ for 
$N \in \mathcal{O}_{5,2}$.  
Similarly, for any nilpotent ${}^{\vee}K$-orbit $\mathcal{O}$ and $N
\in \mathcal{O}$, the action of $W$ on a basis of
$\mathrm{H}_{\mathrm{top}}^{\mathrm{BM}}(\mu^{-1}(N),
\mathbb{Z})^{A_K(N)}$ can be readily 
deduced from Table~\ref{tab:Wactiononconormal}. We leave the details
to the interested reader. 

\section{Characteristic cycles in the case of a regular non-integral
  infinitesimal character}
\label{section:G2computationsNon-integral}

Let $\lambda\in {}^\vee\mathfrak{t}$ be a regular integrally dominant element
(\ref{intdom}).  In the previous section we assumed that $\lambda$ was
integral (\ref{intdom}).  According to Proposition \ref{glambda}, the
integrality of $\lambda$ is equivalent to
${}^{\vee}\mathrm{G}_{2}(\lambda) = {^\vee}\mathrm{G}_{2}$.
In this section we assume that $\lambda$ is \emph{not} integral.
By Proposition \ref{glambda}, there are four  possibilities for
${}^{\vee}\mathrm{G}_{2}(\lambda)$: 
$${^\vee}T, \ \mathrm{GL}_{2}, \
\mathrm{SL}_{2} \times_{\upmu_{2}} \mathrm{SL}_{2}, \ \mathrm{SL}_{3}.$$
Write $X(\lambda)$ for the flag variety 
${}^{\vee}\mathrm{G}_{2}(\lambda) /\,  ^{\vee}B(\lambda)$
of ${}^{\vee}\mathrm{G}_{2}(\lambda)$.
As in Proposition \ref{Ilambdaorbits}, we continue to denote
${}^{\vee}K(\lambda)$  for the analogues of the subgroup 
group ${}^{\vee}K \subset {}^{\vee}\mathrm{G}_{2}$.  

The set $\{P(\xi):\, \xi\in
\Xi(X(\lambda),{}^{\vee}K(\lambda)) \}$ of
${}^{\vee}K(\lambda)$-equivariant irreducible perverse sheaves on
$X(\lambda)$, was described in (\ref{eq:IrrObj2}). 
The goal of this section is to compute the characteristic cycles of all  
these irreducible perverse sheaves. This turns out to be an easy exercise,
which is a direct consequence of Lemma \ref{lem:lemorbit}  and
Proposition \ref{prop:CC2}. 
\begin{theo}\label{theo:CCsNonintegral}
Let $P$ be an irreducible perverse sheaf in
$\mathcal{P}(X(\lambda),{}^{\vee}K(\lambda))$   
with support in the ${}^{\vee}K(\lambda)$-orbit closure
$\overline{S}$. Then 
\begin{equation}\label{eq:CCnonintegral}
CC(P)=[T_S].
\end{equation}
\end{theo}
\begin{proof}
We provide a proof based on the cases listed in Proposition
\ref{Ilambdaorbits}.  In the case that
${}^{\vee}K(\lambda)={}^{\vee}\mathrm{G}_{2}(\lambda)$ there is
only one ${}^{\vee}K(\lambda)$-orbit on $X(\lambda)$, and the result follows
from Lemma \ref{lem:lemorbit}.   

For all other pairs
$({}^{\vee}\mathrm{G}_{2}(\lambda),{}^{\vee}K(\lambda))$, equality
(\ref{eq:CCnonintegral}) ends up being a direct consequence of Proposition
\ref{prop:CC2}.   Among those pairs, the most interesting case is
the pair $(\mathrm{SL}_3,\mathrm{GL}_2)$.  We  provide the details for
this case.  As orbits $S_0$, $S_1$, and $S_2$ are closed
(\ref{sl3hasse}), equality 
(\ref{eq:CCnonintegral} follows for these orbits from 
Lemma \ref{lem:lemorbit}. 

As we did in Table \ref{eq:ccaction}, in the case of the pair
$(\mathrm{G}_{2},\mathrm{SL}_2\times_{\upmu_2}\mathrm{SL}_2)$, we
summarize in the following table, the coherent continuation
representation of the simple reflections on the set irreducible
perverse sheaves $\{P(\xi_i):0\leq i\leq 5\}$ (\ref{eq:IrrObj2}). For
simplicity, we write $P_i$ instead of $P(\xi_i)$.  

\begin{table}[!htb]
\centering   
   % \caption{Global caption}
    \begin{minipage}{.5\linewidth}
      \caption{$\left({}^{\vee}\mathrm{G}_{2}(\lambda),{}^{\vee}K(\lambda)\right)\, =\, (\mathrm{SL}_3,\mathrm{GL}_2)$      
      }\label{eq:ccactionSL3}
      \centering
      \begin{tabular}{ | l | c | c | }
\hline 
 $i$ & $s_{1}\cdot P_i$ & $s_{2}\cdot P_i$ \\ \hline \hline		
   0 & $P_0+P_4$ & $P_0+P_3$\\	\hline
   1 & $-P_1$ & $P_1+P_4$ \\ \hline
   2 & $P_2+P_3$ & $-P_2$ \\ \hline
   3 & $-P_3$ & $P_2+P_3+P_5$  \\ \hline
   4 & $P_1+P_4+P_5$ & $-P_4$\\ \hline
   5 & $-P_5$ & $-P_5$\\ 
 \hline  
\end{tabular}
    \end{minipage}%
\end{table}

Looking at Table \ref{eq:ccactionSL3} and
(\ref{eq:verticalP}), we see that 
\begin{align*}
\begin{aligned}
%\label{eq:tauSL3}
\alpha_1\in \tau(P(\xi_i))&\quad\text{ if and only if }\quad i\in\{1,3,5\},\\
\alpha_2\in \tau(P(\xi_i))&\quad\text{ if and only if }\quad i\in\{2,4,5\}.
\end{aligned}
\end{align*}
As a consequence, equality~(\ref{eq:CCnonintegral}) follows  for
$P(\xi_3)$ and $P(\xi_4)$ by the exact same argument used to prove
(\ref{eq:CC3}) and (\ref{eq:CC4}). 
The case of $P(\xi_5)$ is similar and follows by the same argument
used to obtain  (\ref{eq:CC9}).

For the remaining pairs described in Proposition \ref{Ilambdaorbits},
the situation is less exciting. In these cases, the open orbit $S$
carries only the trivial local system $\underline{\mathbb{C}}_S$, and
the action of any simple reflection on  $P(\xi)$, where
$\xi=(S,\underline{\mathbb{C}}_S)$, satisfies  
%\label{eq:ccothercases}
$$s_\alpha\cdot P(\xi)=-P(\xi).$$
Equality $CC(P(\xi))=[T_S]$ can then be proved as equality (\ref{eq:CC9}).
This completes the proof of Theorem~\ref{theo:CCsNonintegral},
since all the other orbits are closed, and Lemma~\ref{lem:lemorbit}
takes care of Equality~(\ref{eq:CCnonintegral}) for closed orbits.
\end{proof}

As an application of Theorem \ref{theo:CCsNonintegral}, we determine
the $W$-action of the simple reflections on the conormal bundles for
the pair $(\mathrm{SL}_3,\mathrm{GL}_2)$.
Since the characteristic cycle of every irreducible perverse sheaf 
reduces to a single conormal bundle, the $W$-equivariance
\eqref{eq:W-equivariant}  of the characteristic cycle map implies that
the $W$-action  on the conormal bundles follows the same pattern as
Table \ref{eq:ccactionSL3}. More precisely,
\begin{table}[!htb]
\centering   
   % \caption{Global caption}
    \begin{minipage}{.5\linewidth}
       \caption{$\left({}^{\vee}\mathrm{G}_{2}(\lambda),{}^{\vee}K(\lambda)\right)\, =\, (\mathrm{SL}_3,\mathrm{GL}_2)$      
      }
      \centering
\begin{tabular}{ | l | c | c | }
\hline 
 $i$ & $s_{1}\cdot [T_{S_i}]$ & $s_{2}\cdot [T_{S_i}]$ \\ \hline \hline		
   0 & $[T_{S_0}] + [T_{S_4}]$ & $[T_{S_0}] + [T_{S_3}]$\\	\hline
   1 & $-[T_{S_1}]$ & $[T_{S_1}] + [T_{S_4}]$ \\ \hline
   2 & $[T_{S_1}] + [T_{S_3}]$ & $-[T_{S_2}]$ \\ \hline
   3 & $-[T_{S_3}]$ & $[T_{S_2}] + [T_{S_3}] + [T_{S_5}]$ \\ \hline
   4 & $[T_{S_1}] + [T_{S_4}] + [T_{S_5}]$ & $-[T_{S_4}]$\\ \hline
   5 & $-[T_{S_5}]$ & $-[T_{S_5}]$\\ 
 \hline  
\end{tabular}
          \end{minipage}
    \end{table}

\section{Computing Micro-Packets}

Let $\lambda\in {}^\vee\mathfrak{t}$ be a regular integrally dominant
element (\ref{intdom}).   In this section we compute the micro-packets
of  each of the
${}^{\vee}K(\lambda)$-orbits $S$ on the flag variety $X(\lambda)$ 
of ${}^{\vee}\mathrm{G}_{2}(\lambda)$. 
We begin by briefly recalling the definition of the micro-packet corresponding
to $S$ as given in the introduction. The definition of characteristic
cycles (\ref{cc}) involves a map 
$$\chi_{S}^{\mathrm{mic}}:\mathscr{K}(X,H)\rightarrow \mathbb{Z}$$
called the microlocal multiplicity along $S$
(\ref{eq:characteristicvariety}). 
% Adams, Barbasch and 
%Vogan use the microlocal multiplicity map to associate to $S$ a the
%micro-packet $\Pi_S^{\mathrm{mic}}$, which is a finite set of
%irreducible representations of pure real forms.  
%In the case of $\mathrm{G}_{2}$, there are only two pure real
%forms, the split real form and the compact real form. 
%. We denote by $\mathrm{G}_{2}(\mathbb{R}, \delta_s)$ and
%$\mathrm{G}_{2}(\mathbb{R}, \delta_c)$ the corresponding sets of real
%points. 
The micro-packet $\Pi_S^{\mathrm{mic}}$ is defined in the generalized setting of
(\ref{landscape}) and (\ref{landscape1}) as follows. By the
Local Langlands Correspondence, to each
complete geometric parameter $\xi \in \Xi(X(\lambda),
{}^{\vee}K(\lambda))$ there corresponds an irreducible representation
$\pi(\xi)$ of either the split or compact form of $\mathrm{G}_{2}$. The 
 micro-packet is defined as 
\begin{equation}\label{eq:DefMicroPacket}
\Pi_{S}^{\mathrm{mic}}\, =\, \{\pi(\xi):\chi_{S}^{\mathrm{mic}}(P(\xi))\neq 0\}.
\end{equation}
(\emph{cf.}~(\ref{mpacket})).
For each $S$ there is also a strongly stable distribution given by
$$
\eta_S^{\mathrm{mic}}=\sum_{\xi\in \Xi(X(\lambda),{}^{\vee}K(\lambda))}
e(\xi) \, (-1)^{\dim(S_\xi)-\dim(S)} \, \chi_{S}^{\mathrm{mic}}(P(\xi)) \ \pi(\xi)
$$
(\emph{cf.}~(\ref{etamic})). 
%Here $e(\xi)$ is Kottwitz' sign (see \cite{Kot} and \cite[Definition
%15.8]{ABV}) whose value is one if 
%$\pi(\xi)$ is a representation of $\mathrm{G}_{2}(\R,\delta_s)$ and
%minus one otherwise. 
\begin{comment}
Now, recall that to each A-parameter $\psi$ of $\mathrm{G}_{2}$, there
is an associated L-parameter  
$\varphi_\psi$, and by \eqref{lparambij}, this corresponds to a
${}^{\vee}K(\lambda)$-orbit on $X(\lambda)$, which we denote by
$S_\psi$. In this case, we  
refer to the micro-packet of $S_\psi$ as the  
$\mathrm{ABV}$-packet of $\psi$, and write
$$
\Pi_\psi^{\mathrm{ABV}}:=\Pi_{S_\psi}^{\mathrm{mic}}.
$$
\end{comment}
We shall give an explicit description of these micro-packets and stable
distributions. 

The most interesting case arises when $\lambda$ is
regular and integral. In this case 
$$\left({}^{\vee}\mathrm{G}_{2}(\lambda),{}^{\vee}K(\lambda)\right) =
({^\vee}\mathrm{G}_{2} , {^\vee}K) \cong
 {}(\mathrm{G}_{2},\mathrm{SL}_2\times_{\upmu_2}\mathrm{SL}_2),$$  
and $X(\lambda)$ is the flag variety ${}^{\vee}\mathrm{G}_{2}/
{^\vee}B$ as in Sections \ref{Korbits} and \ref{section:G2computations} .
The microlocal multiplicities along ${}^{\vee}K$-orbits of
$X(\lambda)$ were computed in Theorem~\ref{theo:CCs}. The following
theorem is therefore an immediate consequence of Theorem~\ref{theo:CCs} and
the definitions of the micro-packets and stable distributions.  
\begin{theo}\label{theo:Micropackets}
 The micro-packet and stable distribution, associated to each
 ${}^{\vee}K$-orbit of the flag variety of ${}^{\vee}\mathrm{G}_{2}$
 are described in the following table. 
\begin{center}
\begin{tabular}{ | l | c | c | }
\hline
 $i$ & $\Pi_{S_i}^{\mathrm{mic}}$ & $\eta_{S_i}^{\mathrm{mic}}$
\\ \hline \hline		
   0 & $\{\pi(\xi_0)\}$ & $\pi(\xi_0)$\\	\hline
   1 & $\{\pi(\xi_{10}),\pi(\xi_6),\pi(\xi_1)\}$ & $\pi(\xi_{10})+2\pi(\xi_6)+\pi(\xi_1)$ \\ \hline
   2 & $\{\pi(\xi_{11}),\pi(\xi_2)\}$ & $\pi(\xi_{11})+\pi(\xi_2)$ \\ \hline
   3 & $ \{\pi(\xi_7),\pi(\xi_3)\}$  & $ \pi(\xi_7)+\pi(\xi_3)$ \\ \hline
   4 & $\{\pi(\xi_8),\pi(\xi_4)\}$ & $\pi(\xi_8)+\pi(\xi_4)$\\ \hline
   5 & $\{\pi(\xi_5)\}$ & $\pi(\xi_5)$\\ \hline
   6 & $\{\pi(\xi_{10}),\pi({\xi_8}),\pi(\xi_6)\}$ & $\pi(\xi_{10})-\pi({\xi_8})+\pi(\xi_6)$\\ \hline
   7 & $ \{\pi(\xi_{11}),\pi(\xi_7)\}$ & $ -\pi(\xi_{11})+\pi(\xi_7)$\\ \hline
   8 & $\{\pi(\xi_{10}),\pi(\xi_8)\}$ & $-\pi(\xi_{10})+\pi(\xi_8)$\\ \hline
   9 & $ \{\pi(\xi_9),\pi(\xi_{10}),\pi(\xi_{11}),
\boldsymbol{\pi(\xi_{12})}\}$ & $ \pi(\xi_9)+\pi(\xi_{10})+\pi(\xi_{11})-
\boldsymbol{
\pi(\xi_{12})}$\\ 
  \hline  
 \end{tabular}
\end{center}
We have written $\pi(\xi_{12})$ in bold to indicate that it is a
representation of the compact pure real form of $\mathrm{G}_2$. All
other representations are representations of the split real form. 
\begin{comment}
\begin{eqnarray*}
\Pi_{S_0}^{\mathrm{mic}} = \{\pi(x_9,y_0)\}&;& 
\eta_{S_0}^{\mathrm{mic}}=\pi(x_9,y_0)\\
\Pi_{S_1}^{\mathrm{mic}} = \{\pi(x_2,y_{9}),\pi(x_6,y_6),\pi(x_{9},y_1)\}
&;& 
\eta_{S_1}^{\mathrm{mic}}=\pi(x_2,y_{9})+2\pi(x_6,y_6)+\pi(x_{9},y_1)
\\
\Pi_{S_2}^{\mathrm{mic}} = \{\pi(x_1,y_{9}),\textcolor{red}{\pi(x_5,y_5)},\pi(x_{9},y_2)\}
&;& 
\eta_{S_2}^{\mathrm{mic}}=\pi(x_1,y_{9})+\textcolor{red}{\chi_{S_2}^{}(P_5)\pi(x_5,y_5)}+\pi(x_{9},y_2)\\
\Pi_{S_3}^{\mathrm{mic}} = \{\pi(x_3,y_7),\pi(x_7,y_3)\}
&;& 
\eta_{S_3}^{\mathrm{mic}}=
\pi(x_3,y_7)+\pi(x_7,y_3)
\\
\Pi_{S_4}^{\mathrm{mic}} = \{\pi(x_4,y_8),\pi(x_{8},y_4)\}
&;& 
\eta_{S_4}^{\mathrm{mic}}=\pi(x_4,y_8)+\pi(x_{8},y_4)\\
\Pi_{S_5}^{\mathrm{mic}} = \{\pi(x_5,y_5)\}
&;& 
\eta_{S_5}^{\mathrm{mic}}=\pi(x_5,y_5)\\
\Pi_{S_6}^{\mathrm{mic}} = \{\pi({x_2,y_9}),\pi({x_4,y_8}),\pi(x_6,y_6)\}
&;& 
\eta_{S_6}^{\mathrm{mic}}=
\pi({x_2,y_9})-\pi({x_4,y_8})+\pi(x_6,y_6)
\\
\Pi_{S_7}^{\mathrm{mic}} = \{\pi(x_1,y_{9}),\pi(x_{4},y_8)\}
&;& 
\eta_{S_7}^{\mathrm{mic}}=-\pi({x_1,y_9})+\pi({x_3,y_7})
\\
\Pi_{S_8}^{\mathrm{mic}} = \{\pi(x_2,y_{9}),\pi(x_{4},y_8)\}
&;& 
\eta_{S_8}^{\mathrm{mic}}=-\pi({x_2,y_9})+\pi({x_4,y_8})\\
\Pi_{S_9}^{\mathrm{mic}} = \{\pi(x_0,y_9),\pi(x_{1},y_9),\pi(x_2,y_9),\pi(x_0',y_9)\}
&;& 
\eta_{S_1}^{\mathrm{mic}}=
\pi(x_0,y_9)+\pi(x_{1},y_9)+\pi(x_2,y_9)-\pi(x_0',y_9).
\end{eqnarray*}
\end{comment}
\end{theo} 

It is of interest to know which of these micro-packets can be 
Arthur packets in the sense of \cite[Section 22]{ABV}.  We refer the
reader to \cite[Section 22]{ABV} for more details in the following
investigation.  In the setting
of $\mathrm{G}_{2}$, an A-parameter reduces to a homomorphism
$$\psi: W_{\mathbb{R}} \times \mathrm{SL}_{2}  \rightarrow
{^\vee}\mathrm{G}_{2}$$
whose restriction $\psi_{|W_{\mathbb{R}}}$ is a tempered L-parameter,
and whose restriction $\psi_{|\mathrm{SL}_{2}}$ is algebraic.  The
Arthur packet $\Pi_{\psi}$ only depends on $\psi$ up to
${^\vee}\mathrm{G}_{2}$-conjugacy.  The 
restriction $\psi_{|\mathrm{SL}_{2}}$ is determined by the value of
its differential at the nilpotent element
$\begin{bmatrix} 0 & 1 \\ 0 &  0 \end{bmatrix} \in \mathfrak{sl}_{2}$, 
and the conjugacy 
class of this value falls into one of five possible nilpotent orbits
\cite[Section 6.20, Section 7.18]{HumphCC}.  Since
$\psi(W_{\mathbb{R}})$ commutes with $\psi(\mathrm{SL}_{2})$, the
image of $\psi_{|W_{\mathbb{R}}}$ lies in the (Levi subgroup of the)
centralizer of $\psi(\mathrm{SL}_{2})$.  These centralizers are well-understood
\cite[Section 13.1]{Carter85} and place restrictions on the nature of
$\psi$. 

By definition, each Arthur packet $\Pi_{\psi}$  is equal to
some micro-packet 
$\Pi^{\mathrm{mic}}_{S_{\ell}}$.  In order to determine which
$S_{\ell}$ this is in the present setting, we
need to first ensure that the infinitesimal character $\lambda$ of $\psi$
\cite[(22.8)(c)]{ABV}, is
dominant, regular and integral (\emph{cf.}~Section \ref{compgroupsec}).  Let
$S$ be the ${^\vee}K$-orbit which corresponds to the  tempered L-parameter
$\psi_{|W_{\mathbb{R}}}$ as in (\ref{lparambij}). Then
it turns out that $\Pi_{\psi} = \Pi_{S_{\ell}}^{\mathrm{mic}}$ only if $p(S_{\ell})
= p(S)$ (see Table \ref{orbitdata}).  
% Actually $p(S_{\ell})$ is equal to the image in the Weyl group of the
% element (22.8)(c)[abv], where $S$ here  equals $y_{0}$ there.  However,
% $y_{1}$ in (22.8)(b)[abv] belongs to ${^\vee}T$ in our setup, so it
% does not contrubute to the image in the Weyl group and only $p(S)$ matters.
We provide a summary of how to determine $S_{\ell}$ for each of the five
nilpotent orbits mentioned above.  The 
terminology of the headings follows \cite[Table 7.18]{HumphCC}.
\begin{description}
\item[Trivial ($1$)] 
In this case the image of $\psi_{|\mathrm{SL}_{2}}$ is trivial and we may
identify $\psi = \psi_{|W_{\mathbb{R}}}$, a tempered
L-parameter of $\mathrm{G}_{2}$.  In this case the Arthur packet is
equal to the L-packet of $\psi$.  As a tempered $L$-parameter with
dominant, regular and integral infinitesimal character, 
the L-packet $\Pi_{\psi}$ consists of discrete series \cite[Theorem 14.91,
  Proposition 8.22]{Knapp}.    The ${^\vee}K$-orbit $S_{\ell}$ corresponding  to
the L-parameter $\psi$, as outlined (\ref{lparam})-(\ref{lparambij}), 
must therefore have $p(S_{\ell})$ equal to the long Weyl group element
(\cite[proof of Lemma 3.4]{Langlands}).   Consulting Table
\ref{orbitdata} we see that  $\ell  = 9$ and so
$\Pi_{\psi} = \Pi_{S_{9}}^{\mathrm{mic}}$.
%  If $\pi \in \Pi_{\psi}$ is properly induced then it is induced from
%  either an intermediate parabolic or a Borel.  If $\pi$ is induced from
%  a Borel then its continuous parameter must be zero since it is
%  imaginary and the infinitesimal character is integral.  This forces
%  the infinitesimal character to be zero, which is not regular.  If
%  $\pi$ is induced from an intermediate parabolic corresponding to a
%  simple root $\alpha_{j}$ then there is an orthogonal root $\beta$
%  to $\alpha_{j}$.  As in the previous case, the continuous parameter
%  is zero and the infinitesimal character is then a
%  multiple of $\alpha_{j}$.  This infinitesimal character is
%  orthogonal to $\beta$ so is singular.  

\item[Long/short root ($A_{1}$/$\tilde{A}_{1}$)]  As an example,
  suppose $\psi(\mathrm{SL}_{2})$ is given by the
  $\mathrm{SL}_{2}$-triple of the  long  root
  $3{^\vee}\alpha_{1}+ 2 {^\vee}\alpha_{2}$.  As the group
  $\psi(W_{\mathbb{R}})$ centralizes $\psi(\mathrm{SL}_{2})$,
  $\psi(W_{\mathbb{R}})$  lies in the  
  group generated by the $\mathrm{SL}_{2}$-triple of the orthogonal short root
  ${^\vee}\alpha_{1}$.  
 The infinitesimal
  character of $\psi$ is integral  only if the infinitesimal character
  of $\psi_{|W_{\mathbb{R}}}$ is integral.  Since
  $\psi_{|W_{\mathbb{R}}}$ is also tempered, it must be
  a discrete parameter.  As a discrete parameter with values in a
  group isomorphic to
  $\mathrm{SL}_{2}$, it has infinitesimal character equal to 
$m \alpha_{1} $ where $m$ is a positive odd integer.  If we write
  $W_{\mathbb{R}} = \mathbb{C}^{\times} \cup j \mathbb{C}^{\times}$ as
  in Section \ref{compgroupsec},  the image of $\psi_{|W_{\mathbb{R}}}(j)$ in
  the Weyl group is the simple reflection $s_{1}$ generated by
  ${^\vee}\alpha_{1}$.  According to
\cite[(22.8)(c)]{ABV}, the infinitesimal character of $\psi$ is
$$\lambda = m \alpha_{1}  + {^\vee}(3{^\vee}\alpha_{1} + 2 {^\vee}\alpha_{2}) =
m \alpha_{1}  + (2\alpha_{2} + \alpha_{1})  \in {^\vee}\mathfrak{t}$$ 
which is not dominant.  Conjugation of $\psi$ by the Weyl group
element $s_{1}s_{2}$ produces an equivalent A-parameter with dominant
infinitesimal character $s_{1}s_{2} \cdot \lambda$.  The infinitesimal
character is regular for $m>1$.  Following the outline of
(\ref{lparam})-(\ref{lparambij}), the Weyl group element $p(S_{\ell})$
must be equal to 
the conjugate of $s_{1}$ by $s_{1}s_{2}$.  This corresponds to orbit
$S_{8}$, as seen in Table \ref{orbitdata}.  Consequently, $\Pi_{\psi}
= \Pi^{\mathrm{mic}}_{S_{8}}$.

Similar arguments show that when $\psi(\mathrm{SL}_{2})$ is given by
the $\mathrm{SL}_{2}$-triple of a short root, then $\Pi_{\psi} =
\Pi^{\mathrm{mic}}_{S_{7}}$.

\item[Subregular (${^\vee}\mathrm{G}_{2}(a_{1})$)] In this case the
  identity component of the centralizer of $\psi(\mathrm{SL}_{2})$
  consists of unipotent elements \cite[Table 7.12]{HumphCC}.  This
  forces $\psi_{|W_{\mathbb{R}}}(\mathbb{C}^{\times}) = 1$ and so
  $\psi_{|W_{\mathbb{R}}}$ has infinitesimal character zero.  The
  infinitesimal character of $\psi$ is then
\begin{equation}
\label{lampsi}
  \lambda = d\psi \left(0, \begin{bmatrix}1/2 & 0 \\ 0 &
    -1/2  \end{bmatrix}  \right) 
  \in {^\vee}\mathfrak{t}
\end{equation}
\cite[(22.8)(c)]{ABV}.
The first column of \cite[Table 7.12]{HumphCC} indicates that
${^\vee}\alpha_{1}(\lambda) = 0$.  Therefore $\lambda$ is singular and
$\Pi_{\psi}$ does not equal any of the micro-packets above.  

We will
need a  description of subregular A-parameters without restriction on
the infinitesimal character in Section
\ref{section:singularcase}, 
so we may as well do that now. The first column of \cite[Table
  7.12]{HumphCC} also indicates that ${^\vee}\alpha_{2}(2\lambda) = 2$.
This implies 
\begin{equation}
\label{lampsi1}
\lambda =   \alpha_{1} + 2 \alpha_{2} = (2\alpha_{1}+ 3\alpha_{2})/2 +
\alpha_{2}/2,
\end{equation}
a dominant element in ${^\vee}\mathfrak{t}$.
The roots $\alpha_{2}$ and $2\alpha_{1} + 3\alpha_{2}$ are orthogonal,
and their respective coroots are ${^\vee}\alpha_{2}$ and $2 \,
{^\vee}\alpha_{1}+ {^\vee}\alpha_{2}$.  There is a natural choice for
$\psi(1, \begin{bmatrix} 1 & 1 \\ 0
  & 1 \end{bmatrix})$, namely a product of nontrivial elements in the
root subgroups $U_{{^\vee}\alpha_{2}}$ and $U_{2\, {^\vee}\alpha_{1} +
{^\vee}\alpha_{2}}$. With this choice
\begin{equation}
\label{sl2im}
\psi(1, \mathrm{SL}_{2}) = \langle U_{\pm{^\vee\alpha_{2}}} \rangle
\langle U_{\pm( 2 \, {^\vee}\alpha_{1}+ {^\vee}\alpha_{2})} \rangle \cong \mathrm{SL}_{2}
\times_{\mu_{2}} \mathrm{SL}_{2} 
\end{equation}
(\emph{cf.}~Section \ref{groupK}).
 This being fixed, the only remaining freedom in choosing $\psi$ is in
 the value of the 
semisimple element
$\psi_{|W_{\mathbb{R}}}(j)$. As indicated above, the identity
component of the image
$\psi(W_{\mathbb{R}})$ consists of unipotent elements. Therefore the square of 
$\psi_{|W_{\mathbb{R}}}(j)$, being  both unipotent and 
semisimple,  is trivial.   This shows that
$\psi_{|W_{\mathbb{R}}}(j)$ can have order 
at most two.  Let $\psi_{a}$ be a subregular A-parameter for which
$(\psi_{a})_{|W_{\mathbb{R}}}(j) = 1$.   The only remaining
possibility  is a subregular A-parameter for which
$\psi_{|W_{\mathbb{R}}}(j)$ has order two.   This semisimple element
must then lie outside of the identity component of the centralizer of
$\psi(\mathrm{SL}_{2})$.  \cite[Table 7.18]{HumphCC} indicates that the
component group of the centralizer $\psi(1, \begin{bmatrix} 1 & 1 \\ 0
  & 1 \end{bmatrix})$ is  $\mathcal{S}_{3}$--the symmetric group 
of order six.  The elements of order two in $\mathcal{S}_{3}$ form a
single conjugacy 
class, so are all conjugate in the centralizer.  One of the elements
of order two is the unique nontrivial central element in
(\ref{sl2im}), which is $\exp( \uppi i \alpha_{2})$.   We define
$\psi_{b}$ to be the unique A-parameter with
$(\psi_{b})_{|W_{\mathbb{R}}}(j) = \exp( \uppi i 
\alpha_{2})$. This exhausts all subregular  A-parameters up to conjugacy.

\item[Regular (${^\vee}\mathrm{G}_{2}$)]  In this case the entire centralizer
  of $\psi(\mathrm{SL}_{2})$  consists of unipotent elements
  \cite[Tables 7.12 and 7.18]{HumphCC}.   This forces
  $\psi_{|W_{\mathbb{R}}}$ to be trivial. 
% as it consists of elements that are both semisimple and unipotent
 As in the subregular case, the
  infinitesimal character $\lambda$ of $\psi$ is given by (\ref{lampsi}).  The
  first column of \cite[Table 7.12]{HumphCC} then indicates that 
${^\vee}\alpha_{1}(2\lambda) = {^\vee}\alpha_{2}(2\lambda) = 2$.  
  \cite[Lemma A 13.3]{Humphreys} then tells us that
$\lambda = \rho$, which is regular, dominant and integral.
  Following
  (\ref{lparam})-(\ref{lparambij}) and \cite[(22.8)]{ABV}, the Weyl
  group element $p(S_{\ell})$ 
  is equal to the image of $\psi_{|W_{\mathbb{R}}}(j) = 1$ in the Weyl
  group.   As the image is trivial, Table \ref{orbitdata}
  indicates that $\Pi_{\psi} = \Pi_{S_{\ell}}^{\mathrm{mic}}$ for some
  $0 \leq \ell \leq 2$.  A careful examination of the terms in
  \cite[(22.8)(c)]{ABV}, together with \cite[(5.7)(c) and Proposition
    6.17]{ABV}, reveal 
  that $\Pi_{\psi} = \Pi_{S_{0}}^{\mathrm{mic}}$, where $S_{0}$
  corresponds to the Langlands parameter of the trivial
  representation of the split form of $\mathrm{G}_{2}$.  The details
  are left to the reader. 
% $y = y_{0} y_{1}$ where $y_{0} = 1$ and $y_{1} = exp(\pi i \rho) =
% e(\rho/2)$. (5.7)(c) implies $\phi_{\psi}}(j) = 1$.  Therefore
% $\phi_{\psi}$ is the Langlands parameter for the Langlands quotient
% of the standard principal series with infinitesimal character $rho$
% which is trivial on the component group of the inducing torus.  The
% Langlands quotient is the trivial representation.  In addition,
% $S_{0}=e(\rho/2)$  corresponds to the normalized torus part of the
% dual strong involution as in (14i) \cite{AVParameters}. This gives
% trivial unnormalized torus part and that's equal to $\phi_{\psi}(j)$.
  
  \end{description}
We note that, according to \cite[Corollary
  4.18]{arancibia_characteristic} the Arthur packets corresponding to
the orbits 
$S_0,\, S_7,\, S_8$ and $S_9$ above all fall within the class of packets
described by Adams and Johnson in \cite[Definition
  2.11]{Adams-Johnson}. Each such packet is associated with a
parabolic subgroup $Q \subset \mathrm{G}_{2}$, and to
each strong real form 
of the Levi subgroup of $Q$, there corresponds a unitary
character. The packets are then constructed by cohomologically
inducing these unitary characters to $\mathrm{G}_{2}$. That these
packets are of Adams-Johnson type may be computed using the
$\mathtt{Aq\_packet}$ command in the Atlas of Lie Groups and
Representations software. 

Using the Atlas software command
\texttt{is\_unitary}, we may verify 
that  each of the remaining micro-packets contains at
least one representation that is not unitary. In the next section, we
will see that some of these micro-packets are nonetheless useful in
describing Arthur packets with singular infinitesimal character. This
occurs for the singular Arthur packet associated to
the subregular nilpotent orbit discussed above. 
   
We conclude with the case that $\lambda$ is regular, but not integral. 
By (\ref{eq:IrrObj2}) and Theorem \ref{theo:CCsNonintegral},  the microlocal
multiplicity   
$\chi_S^{\mathrm{mic}}(P(\xi))$, for any
${}^{\vee}K(\lambda)$-orbit $S$ in $X(\lambda)$ and  $\xi\in
\Xi(X(\lambda), {}^{\vee}K(\lambda))$, 
is nonzero only when $\xi=(S,\underline{\mathbb{C}}_S)$.
Therefore, the micro-packet corresponding to $S$ 
reduces to its associated L-packet $\Pi_{S}$. This proves the following theorem.
\begin{theo}\label{theo:MicropacketsNonIntegral}
Suppose $\lambda \in {^\vee}\mathfrak{t}$ is integrally dominant,
regular, but non-integral. Then
$$
\Pi^{\mathrm{mic}}_S=\Pi_{S}.
$$
\end{theo}
In more familiar notation, Theorem
\ref{theo:MicropacketsNonIntegral} asserts $\Pi^{\mathrm{mic}}_S=\Pi_{\varphi}$,
where $\varphi$ is the L-parameter corresponding to
the ${}^{\vee}K(\lambda)$-orbit $S$.

\section{Micro-packets with singular infinitesimal character} 
\label{section:singularcase}

To conclude our description of the micro-packets of $\mathrm{G}_{2}$,
we remove the regularity hypothesis on the infinitesimal character
$\lambda\in{}^{\vee}\mathfrak{t}$. In 
other words, the integrally dominant   (\ref{intdom}) element
$\lambda\in{}^{\vee}\mathfrak{t}$ may now satisfy
$${^\vee}\alpha(\lambda) = 0$$
for some roots ${^\vee}\alpha \in R({^\vee}\mathrm{G}_{2}, {^\vee}T)$.
To compute the micro-packets with singular infinitesimal character, we
extend the description of the microlocal multiplicities  
$\chi_S^{\mathrm{mic}}$ in Sections \ref{section:G2computations} and
\ref{section:G2computationsNon-integral} to incorporate  
 ${}^{\vee}K$-orbits $S$ of a \emph{generalized} flag variety of
${}^{\vee}\mathrm{G}_{2}$.  The main tool for this extension is the
\emph{translation principle}, which allows one to transfer results
from regular infinitesimal characters to those for singular
infinitesimal characters (\cite{Jantzen}, \cite[Chapter
  7]{greenbook}).
The reader is assumed to be somewhat familiar with the translation
principle. We will begin by presenting
the geometric perspective of the translation principle, as described
in \cite[Chapter 8]{ABV}.  Applying the translation principle to
Theorems \ref{theo:Micropackets} and 
\ref{theo:MicropacketsNonIntegral} will allow us to describe all
micro-packets of $\mathrm{G}_{2}$ with singular infinitesimal character. 

Let $\mathcal{O} \subset  {^\vee}\mathfrak{g}_2$ be the
${^\vee}\mathrm{G}_{2}$-orbit of  $\lambda 
\in {^\vee}\mathfrak{t}$ under the adjoint action.
The translation principle begins with the existence of a regular
element $\lambda' \in  {}^{\vee}\mathfrak{t}$, with 
\({^\vee}\mathrm{G}_2\)-orbit \(\mathcal{O}' \subset
  {^\vee}\mathfrak{g}_2\),  and a
  \emph{translation datum} \(\mathcal{T}\) from \(\mathcal{O}\) to
  \(\mathcal{O}'\) (\cite[Definition 8.6, Lemma 8.7]{ABV}). Two
  requirements of a translation datum are that
\begin{equation}\label{eq:lambdaprime}
\lambda - \lambda' \in X_*( {^\vee}T),
\end{equation}
and that for every root $\alpha \in
R({}^{\vee}\mathrm{G}_{2}, {}^{\vee}T)$
\begin{align}\label{eq:translationdatum1}
\alpha(\lambda) \in \{1,2,3, \ldots \}
\quad \Longrightarrow \quad 
\alpha(\lambda') \in \{1,2,3, \ldots\}.
\end{align}
Let \({}^{\vee}P(\lambda)\) be the parabolic subgroup of
\({}^{\vee}\mathrm{G}_{2}(\lambda)\) given by the root spaces of the roots
$\alpha \in R( {^\vee}\mathrm{G}_{2} {^\vee}T)$, satisfying 
%\label{eq:positiveroots}
$$\alpha(\lambda) \in \{0, 1,2,3, \ldots \}.$$
We define \({}^{\vee}P(\lambda')\) similarly \cite[(6.1),
  (6.2)]{ABV}. Due to equations \eqref{eq:lambdaprime} and
\eqref{eq:translationdatum1}, when \(\lambda'\) is integral,
\({}^{\vee}P(\lambda')\) is the Borel subgroup
\({^\vee}B \subset {}^{\vee}\mathrm{G}_{2}\) fixed at the beginning of Section
\ref{section:Korbitsflag}. Otherwise,  \({}^{\vee}P(\lambda')\) is one
of the Borel subgroups ${^\vee}B(\lambda')$
introduced in Section \ref{section:FlagNonIntegral}. In any case, we
have 
\[
{}^{\vee}P(\lambda') \subset {}^{\vee}P(\lambda).
\]
Let $\left({}^{\vee}\mathrm{G}_{2}(\lambda'),
{}^{\vee}K(\lambda')\right)$ be one of the pairs described in
Proposition \ref{Ilambdaorbits}. Recall from Section
\ref{section:FlagNonIntegral} that ${^\vee}\mathrm{G}_{2}(\lambda')$
is the centralizer of $e(\lambda')$ in $\mathrm{G}_{2}$, and that
${^\vee}K(\lambda')$ is the centralizer of an element $y \in
{^\vee}\mathrm{G}_{2}(\lambda')$ with $y^2 = e(\lambda')$. It follows
from  \eqref{eq:lambdaprime} and \eqref{eq:translationdatum1}
that $e(\lambda) = e(\lambda')$, which in turn implies that the
corresponding centralizers coincide, that is 
\[
{^\vee}\mathrm{G}_{2}(\lambda) = {^\vee}\mathrm{G}_{2}(\lambda') \quad
\text{and} \quad 
{^\vee}K(\lambda) = {^\vee}\mathrm{G}_{2}(\lambda)^y= {^\vee}K(\lambda').
\]
Despite these equalities, we will persist in  writing 
${^\vee}\mathrm{G}_{2}(\lambda)$ and
${^\vee}\mathrm{G}_{2}(\lambda')$, and likewise  
${^\vee}K(\lambda)$ and ${^\vee}K(\lambda')$, so that the reader can
easily distinguish when the theory pertains to the singular or to the regular
setting.  

In keeping with the notation of the previous sections, we denote by
\(X(\lambda)\) the generalized flag variety
\({}^{\vee}\mathrm{G}_{2}(\lambda) / {}^{\vee}P(\lambda)\) of
\({}^{\vee}\mathrm{G}_{2}(\lambda)\), and by \(X(\lambda')\) the flag
variety \({}^{\vee}\mathrm{G}_{2}(\lambda') /
{}^{\vee}P(\lambda')\). By \cite[Proposition 8.8]{ABV}, the surjection
\begin{align}\label{eq:ftee}
f_\mathcal{T}: \, X(\lambda') \, \longrightarrow \, X(\lambda)
\end{align}
defines a  smooth and proper morphism. In addition, the morphism
$f_\mathcal{T}$ has connected fibres of fixed dimension \(d\). 
%These facts will be
%used to define the translation functor and when comparing orbit dimensions.  
According to \cite[Proposition 7.15]{ABV}, the morphism
$f_\mathcal{T}$
% in \eqref{eq:ftee} 
induces an inclusion
%\label{eq:fstar}
$$f^{*}_{\mathcal{T}}: \Xi(X(\lambda), {}^{\vee}K(\lambda))
\hookrightarrow \Xi(X(\lambda'), {}^{\vee}K(\lambda))$$
of complete geometric parameters (\ref{geoparam}). We refer the reader
to \cite[page 95]{ABV} for a detailed description of the image 
\[
f_\mathcal{T}^{\ast}(\xi) = (f_\mathcal{T}^{\ast}S, f_\mathcal{T}^{\ast}\tau)
\]
of a complete geometric parameter \(\xi = (S, \tau) \in
\Xi(X(\lambda), {}^{\vee}K(\lambda))\). We only mention here that,
since there are only a finite number of
\({}^{\vee}K(\lambda')\)-orbits on \(X(\lambda')\), and
\(f^{\ast}_\mathcal{T}\) has connected fibres, there is a unique open
\({}^{\vee}K(\lambda')\)-orbit in \(f_{\mathcal{T}}^{-1}(S)\). The
orbit \(f_\mathcal{T}^{\ast}S\) is defined to be this unique open
orbit. 
%Following (\ref{orbitparams}), for each
%\({}^{\vee}K(\lambda')\)-orbit \(S'\) on \(X(\lambda')\), there exists
%an element \(y_{S'}^{} \in {}^{\vee}\mathrm{G}_{2}(\lambda')\)
%\textcolor{red}{ with \(y_{S'}^2 = 
%e(\lambda')\)}, such that 
%\[
%S' = {}^{\vee}K(\lambda') \, y_{S'}^{-1} \, {}^{\vee}P(\lambda').
%\]
%Similarly, for some \(y_{S} \in {}^{\vee}\mathrm{G}_{2}(\lambda)\)
%\textcolor{red}{with \(y_{S}^2 = e(\lambda)\)}, we have 
%\[
%S = {}^{\vee}K(\lambda') \,y_{S}^{-1}\, {}^{\vee}P(\lambda).
%\]
%Thus, the \({}^{\vee}K(\lambda')\)-orbits on \(X(\lambda')\) inside
%the preimage \(f_{\mathcal{T}}^{-1}(S)\) are those whose corresponding
%element \(y_{S'}^{}\) satisfies 
%\[
%y_{S}^{-1}y_{S'}^{} \in {}^{\vee}P(\lambda).
%\]
%The open orbit \(f_\mathcal{T}^{\ast}S\) is the one whose image under
%the map i%n \eqref{eq:pmap} has the highest length. 

We can now introduce the geometric version of the
translation functor.   
As defined in \cite[Proposition~8.8(b)]{ABV}, this functor maps the
category of \({}^{\vee}K(\lambda)\)–equivariant perverse sheaves on
\(X(\lambda)\) to the category of \({}^{\vee}K(\lambda')\)–equivariant
perverse sheaves on \(X(\lambda')\).   
It is given by the inverse image of the morphism
$f_\mathcal{T}^{\ast}$ in \eqref{eq:ftee}, shifted by the relative
dimension \(d\) 
\begin{equation}\label{eq:geometricTF}
f^{*}_{\mathcal{T}}[d]: \mathcal{P}\left(
X(\lambda),{}^{\vee}K(\lambda)\right) \longrightarrow \mathcal{P}\left(
X(\lambda'),{}^{\vee}K(\lambda')\right).
\end{equation}
This is a fully faithful exact functor, sending irreducible perverse
sheaves to irreducible perverse sheaves
\[
f^{\ast}_{\mathcal{T}}[d]\left(P(\xi)\right) =
P(f_\mathcal{T}^{\ast}(\xi)), \quad \xi \in  \Xi(X(\lambda), {}^{\vee}K(\lambda))
\]
 \cite[Proposition
  7.15(b)]{ABV}.
%These microlocal multiplicities behive well under the translation functor.
Another important property of the translation functor is that it
preserves the microlocal multiplicities. 
%along the \({}^{\vee}K(\lambda)\)-orbits of \(X(\lambda)\). 
Indeed, it is shown
in \cite[Proposition 20.1 (e)]{ABV} that for all \(\xi \in
\Xi(X(\lambda), {}^{\vee}K(\lambda))\) and for all
\({}^{\vee}K(\lambda)\)-orbits \(S\) in \(X(\lambda)\), we have 
\begin{align}
\begin{aligned}\label{eq:translationofchiV1}
\chi_{S}^{\mathrm{mic}}(P(\xi)) &=
\chi_{f^{\ast}_{\mathcal{T}}S}^{\mathrm{mic}}(f^{\ast}_{\mathcal{T}}P(\xi))
\\ 
&= \chi_{f^{\ast}_{\mathcal{T}}S}^{\mathrm{mic}}(P(f^{\ast}_{\mathcal{T}}\xi)).
\end{aligned}
\end{align}
As a consequence of  \eqref{eq:translationofchiV1}, we can write
$$CC(P(\xi)) = \sum_{S}
\chi_{f^{\ast}_{\mathcal{T}}S}^{\mathrm{mic}}(P(f^{\ast}_{\mathcal{T}}\xi))
\cdot \overline{T^{\ast}_{S}X(\lambda)}$$
(\emph{cf.}~(\ref{cc})).
In addition, by the definition of a micro-packet
\eqref{eq:DefMicroPacket}, we have 
\begin{eqnarray}\label{eq:singularmicpacketrelation}
\pi(\xi)\in\Pi_{S}^{\mathrm{mic}}
\quad \Longleftrightarrow\quad 
\pi(f^{*}_{\mathcal{T}}\xi)\in\Pi_{f^{*}_{\mathcal{T}}S}^{\mathrm{mic}}.
\end{eqnarray}
To complete the description of the micro-packets of
\(\mathrm{G}_{2}\), we need to relate the geometric translation
functor \eqref{eq:geometricTF} with its representation-theoretic
counterpart. The (Jantzen–Zuckerman) translation functor in
representation theory  may be regarded as a homomorphism 
%  \label{transfunct}
$$  \Psi_{\mathcal{T}}:\, K \Pi\left(\lambda',
  \mathrm{G}_{2}/\mathbb{R} \right)
  \longrightarrow K \Pi\left(\lambda,
  \mathrm{G}_{2}/\mathbb{R}\right),$$
from the Grothendieck group of  representations of pure real forms of
$\mathrm{G}_{2}$ with infinitesimal character~$\lambda'$, to that of
pure real forms of $\mathrm{G}_{2}$ with infinitesimal
character $\lambda$.  
%(\cite[(17.8j)]{AvLTV}). 
The functor $\Psi_{\mathcal{T}}$ is surjective (see, for example,
\cite[Corollary 17.9.8]{AvLTV}).  \cite[Proposition 16.4(b)]{ABV}
tells us that  for any complete geometric parameter \(\gamma \in
\Xi(X(\lambda'), {}^{\vee}K(\lambda'))\) the image
\(\Psi_{\mathcal{T}}(\pi(\gamma))\) is either irreducible or zero—the
former occurring if and only if \(\gamma = f^{\ast}_\mathcal{T}(\xi)\)
for some \(\xi \in \Xi(X(\lambda), {}^{\vee}K(\lambda))\).
%Alternatively, the image
%\(\Psi_{\mathcal{T}}(\pi(\gamma))\) is zero if and only if a root
%in the tau-invariant of \(P(\gamma)\), as in (\ref{eq:verticalP}), is
%orthogonal to \(\lambda\) \cite[Proposition 11.16]{ABV}.  
By \cite[Proposition 16.6]{ABV}, we have the identity
\begin{equation}\label{eq:SingularPimic}
\Psi_{\mathcal{T}} \left(\pi(f^{\ast}_\mathcal{T}(\xi)) \right) = \pi(\xi).
\end{equation}
An immediate consequence of this identity and the definition of the
L-packet $\Pi_{S}$ is
\begin{equation}
\label{lpackettrans}
\Pi_{S} = \left\{ \Psi_{\mathcal{T}}(\pi) : \pi \in
\Pi_{f_{\mathcal{T}}^{*} S}, \ \Psi_{\mathcal{T}}(\pi) \neq 0 \right\}.
\end{equation}

Identity (\ref{eq:SingularPimic}) applies to micro-packets as well.
Indeed, the equivalence of \eqref{eq:singularmicpacketrelation}
implies that the image of the micro-packet of
\(f^{\ast}_\mathcal{T}S\)  under the translation functor is equal to
the micro-packet of \(S\)  
\begin{equation}\label{eq:Singularetamic}
\Pi_{S}^{\mathrm{mic}}\ =\ \left\{\Psi_\mathcal{T}(\pi):\ \pi\in 
\Pi_{f^{\ast}_\mathcal{T}S}^{\mathrm{mic}},\ \Psi_\mathcal{T}(\pi)\neq 0 \right\}.
\end{equation}
The parallel statement for the strongly stable distributions is
\begin{align}
\label{singetamic}
\eta_{S}^{\mathrm{mic}}&=\sum_{\xi\in \Xi(X(\lambda),{}^{\vee}K(\lambda))}
e(\xi)(-1)^{\dim(S_\xi)-\dim(S)}\chi_{S}^{\mathrm{mic}}(P(\xi)) \,\pi(\xi)\\
\nonumber &=\sum_{\xi\in \Xi(X(\lambda),{}^{\vee}K(\lambda))}
e(\xi)(-1)^{\dim(f^{\ast}_{\mathcal{T}}S_\xi)-\dim(f^{\ast}_{\mathcal{T}}S)}
\chi_{f^{\ast}_{\mathcal{T}}S}^{\mathrm{mic}}(P(f^{\ast}_{\mathcal{T}}\xi)) \ 
\Psi_{\mathcal{T}}(\pi(f^{\ast}_\mathcal{T}(\xi))) \\
\nonumber &=\Psi_{\mathcal{T}}(\eta_{f^{\ast}_{\mathcal{T}}S}^{\mathrm{mic}}).
\end{align}

By Theorems \ref{theo:Micropackets} and
\ref{theo:MicropacketsNonIntegral}, we know how to compute
\(\Pi_{S'}^{\mathrm{mic}}\) for all \({}^{\vee}K(\lambda')\)-orbits
\(S'\) in \(X(\lambda')\).  Applying the translation functors as
above, the only difficulty in computing 
\(\Pi_{S}^{\mathrm{mic}}\) and \(\eta_{S}^{\mathrm{mic}}\) lies in
recognizing \(f^{\ast}_{\mathcal{T}}S\) among the 
\({}^{\vee}K(\lambda')\)-orbits in \(X(\lambda')\).
Identity (\ref{eq:SingularPimic}) converts the problem of recognizing
\(f^{\ast}_{\mathcal{T}}S\) among the  
\({}^{\vee}K(\lambda')\)-orbits into the equivalent
problem of computing the representation-theoretic
translation functor $\Psi_{\mathcal{T}}$.
The Atlas of Lie Groups and Representations software allows
one to compute  the image of an
irreducible or standard representation under $\Psi_{\mathcal{T}}$
using the commands \texttt{T\_irr} and \texttt{T\_std}.  One may 
use these commands to identify the orbits $f_\mathcal{T}^{\ast}S$.

One does not need software for these computations.  As an
example, we sketch the  computation of  $\Pi_\psi =
\Pi_{S_{\psi}}^{\mathrm{mic}}$ for $\psi$ an A-parameter of 
$\mathrm{G}_{2}$ whose restriction to $\mathrm{SL}_2$ is
determined by the subregular nilpotent orbit as presented on page
\pageref{lampsi}.  Following this presentation, there are two unipotent
A-parameters, $\psi_a$ and $\psi_b$, of $\mathrm{G}_{2}$ that satisfy 
$$(\psi_a)_{|\mathrm{SL}_2}= (\psi_b)_{|\mathrm{SL}_2}.$$
Here, $\psi_a$ denotes the A-parameter whose restriction to $W_\R$
is trivial,  and $\psi_b$ denotes the unique A-parameter such that 
$(\psi_b)_{|W_{\mathbb{R}}}(j) = \exp(\uppi i \alpha_2)$. 
%Moreover, by setting $\lambda$ as in \eqref{lampsi}, 
%we have $S_{\psi_a}=K(\lambda)P(\lambda)$, and from Table \ref{orbitdata} 
%we obtain $S_{\psi_b}=K(\lambda)\sigma_2^{-1}P(\lambda)$, where
%$\sigma_2$ is as in \eqref{titsrep}. 
Let $\lambda$ be as in \eqref{lampsi1}. We may then choose  
$\lambda' = \rho$, the half-sum of the positive roots in
$R(\mathrm{G}_{2},T)$, and use the translation datum
$\mathcal{T}$ from $\mathcal{O}$ to $\mathcal{O}'$. The parabolic
subgroup ${^\vee}P(\lambda')$ is the Borel subgroup ${^\vee}B$.   Since
${^\vee}\alpha_{1}(\lambda) = 0$, the
parabolic subgroup ${^\vee}P(\lambda)$ is given by
\[
{}^{\vee}P(\lambda)={}^{\vee}B \sqcup {}^{\vee}B
\sigma_{1} {}^{\vee}B = \langle U_{\pm {^\vee}\alpha_{1}} \rangle \,  {^\vee}B,
\]  
where $\sigma_{1}$ is the Tits representative (\ref{titsrep}) of the reflection
of ${^\vee}\alpha_{1}$, and 
$U_{{^\vee}\alpha_{1}}$ is the root subgroup of ${^\vee}\alpha_{1}$.
The prescriptions of \cite[(22.8)(c), (5.7)(c) and Proposition 6.17]{ABV}
imply that 
$$S_{\psi_a}={}^{\vee}K(\lambda)\,  {}^{\vee}P(\lambda) =
{}^{\vee}K(\lambda)\,  \langle U_{\pm {^\vee}\alpha_{1}} \rangle \,
{^\vee}B$$ 
The column labelled with $y_{j}$ in Table \ref{orbitdata} indicates
that  $f_\mathcal{T}(S_0)=f_\mathcal{T}(S_2)=S_{\psi_a}$. Moreover,
since $g_{1} \in \langle U_{\pm {^\vee}\alpha_{1}} \rangle$ 
%is obtained via Cayley transform through ${}^{\vee}\alpha_1$ from
%$S_0$ (and $S_1$),    
the table shows that $f_\mathcal{T}(S_4)=S_{\psi_a}$.  In summary,
Table \ref{orbitdata} shows that
\[
f^{-1}_{\mathcal{T}}(S_{\psi_a}) = S_0 \cup S_2 \cup S_4 \quad \text{and} \quad
f^{\ast}_{\mathcal{T}}(S_{\psi_a}) = S_4.
\]
We may now substitute the information from Theorem
\ref{theo:Micropackets}   into equations 
(\ref{eq:Singularetamic}) and (\ref{singetamic}) to
conclude that 
\[
\Pi_{S_{\psi_a}}^{\mathrm{mic}} =
\{\pi(f^{\ast}_{\mathcal{T}}(\xi_8)),
\pi(f^{\ast}_{\mathcal{T}}(\xi_4))\}  
\quad \text{and} \quad
\eta_{S_{\psi_a}}^{\mathrm{mic}} = \pi(f^{\ast}_{\mathcal{T}}(\xi_8))
+ \pi(f^{\ast}_{\mathcal{T}}(\xi_4)). 
\]
Similar considerations with $\psi_{b}$ bring
to light that
\begin{align*}
S_{\psi_b}  = {}^{\vee}K(\lambda)\sigma_{2}^{-1}  {}^{\vee}P(\lambda)
& =  {}^{\vee}K(\lambda) \sigma_2^{-1} \,  \langle U_{\pm
  {^\vee}\alpha_{1}} \rangle \, {}^{\vee}B,\\
f^{-1}_{\mathcal{T}}(S_{\psi_b}) = S_1, & \quad
f^{\ast}_{\mathcal{T}}(S_{\psi_b}) = S_1.
\end{align*}
Substituting Theorem \ref{theo:Micropackets} into
(\ref{eq:Singularetamic}) and (\ref{singetamic}), we obtain
\[
\Pi_{S_{\psi_b}}^{\mathrm{mic}} =
\{\pi(f^{\ast}_{\mathcal{T}}(\xi_{10})),
\pi(f^{\ast}_{\mathcal{T}}(\xi_6)),
\pi(f^{\ast}_{\mathcal{T}}(\xi_1))\}  
\quad \text{and} \quad
\eta_{S_{\psi_b}}^{\mathrm{mic}} =
\pi(f^{\ast}_{\mathcal{T}}(\xi_{10})) +
2\pi(f^{\ast}_{\mathcal{T}}(\xi_6)) +
\pi(f^{\ast}_{\mathcal{T}}(\xi_1)). 
\]
%By the description of the tau-invariant in (\ref{eq:tau}), the image
%of any of these irreducible representations under
%\(\Psi_{\mathcal{T}}\) is nonzero. 
The stable sums \(\eta_{S_{\psi_a}}^{\mathrm{mic}}\) and
\(\eta_{S_{\psi_b}}^{\mathrm{mic}}\) are the ones appearing in
\cite[Theorem 18.10]{VoganG2}.  However,  Vogan uses entirely different
methods.

We finish this section by supposing that  $\lambda$ is
non-integral. In this case it follows from
\eqref{eq:lambdaprime} that $\lambda'$ is also non-integral.
By Theorem \ref{theo:MicropacketsNonIntegral}, 
the micro-packets with infinitesimal character 
$\lambda'$ are equal to their corresponding
L-packets.  According to equations (\ref{lpackettrans}) and
\eqref{eq:Singularetamic}, 
the same can be said for any micro-packet $\Pi_S^{\mathrm{mic}}$
for $S$ a \({}^{\vee}K(\lambda)\)-orbit in $X(\lambda)$.

\section{Kashiwara's local index theorem}
\label{KashiwaraSection}

We continue working in the setting of
Section~\ref{section:singularcase}. As an application of
Theorem~\ref{theo:CCs}, we provide an explicit 
description of Kashiwara's local index formula
(\ref{Kashiwaraformula}) stated in the introduction.
%for any irreducible
%perverse sheaf on a generalized flag variety of
%${}^{\vee}\mathrm{G}_2(\lambda)$.

%That is, we fix \(\lambda \in {}^{\vee}\mathfrak{t}\), let
%\({}^{\vee}P(\lambda)\) denote the parabolic subgroup of
%\({}^{\vee}G_2\) corresponding to the family of roots of
%\({}^{\vee}T\) in \({}^{\vee}G_2\) satisfying
%Equation~\eqref{eq:positiveroots}, and consider the generalized flag
%variety   
%\[
%X := {}^{\vee}G_2(\lambda)/{}^{\vee}P(\lambda).
%\]
%Finally, we define \( {}^{\vee}K \) according to
%Proposition~\ref{Ilambdaorbits}. 

We give more details to the 
notion of \emph{local multiplicity} along a
${}^{\vee}K(\lambda)$-orbit in \( X(\lambda) \)  as in
\cite[Definition 23.6]{ABV}.    
To this end, fix a ${}^{\vee}K(\lambda)$-orbit \( S \subset X(\lambda)
\) and a point \( x \in S \).  For any constructible sheaf \( \mathcal{C} \) of
finite-dimensional complex vector spaces on \( X(\lambda) \),
we denote by \( \mathcal{C}_x \) the stalk of \( \mathcal{C} \) at \(
x \). 
Let $\mathcal{C}(X(\lambda),{}^{\vee}K(\lambda))$ be the category of
${}^{\vee}K(\lambda)$-equivariant constructible sheaves on
$X(\lambda)$. The map  
\[
\chi_{S}^{\mathrm{loc}} : \mathcal{C}(X(\lambda),{}^{\vee}K(\lambda))
\longrightarrow 
\mathbb{N},
\] 
defined by $\chi_{S}^{\mathrm{loc}}( \mathcal{C}) = \dim(\mathcal{C}_x)$,
is independent of the choice of $x \in S$, and is additive with respect
to short exact sequences.  It extends to a $\mathbb{Z}$-linear map
$$\chi_{S}^{\mathrm{loc}} : \mathscr{K}(X(\lambda),{}^{\vee}K(\lambda))
\longrightarrow \mathbb{Z},$$
which we refer to as the \emph{local multiplicity} along $S$
\cite[Definition 1.28, Lemma 7.8]{ABV}. 
%Any irreducible ${^\vee}K(\lambda)$-equivariant perverse sheaf \( P\)
%on $X(\lambda)$ may be regarded as an element in 
%\(\mathscr{K}(X(\lambda),
%{}^{\vee}K(\lambda)) \) and so $\chi_{S}^{\mathrm{loc}}(P)$ is defined
%in this sense. 
% the local multiplicity along \( S \) is also
%known as the \emph{local Euler characteristic} of \( P \) at \( S \),   
%and is given by the formula 
%\label{eq:local-multiplicity}
%$$\chi_{S}^{\mathrm{loc}}(P) = \sum_i (-1)^i \dim\left((\mathrm{H}^i
%P)_x\right).$$ 
%The local multiplicities can be effectively computed using the Atlas
%of Lie Groups and Representations software. 

There is a close relationship between the values of
$\chi_{S}^{\mathrm{loc}}$ on irreducible perverse sheaves and the
values on
irreducible constructible sheaves.  To see this, let $\xi = (S_\xi,
\mathcal{V}_\xi) \in \Xi(X(\lambda),{}^{\vee}K(\lambda))$ be a
complete geometric parameter.  
%as in~(\ref{eq:complete-geoparam}). 
Consider the irreducible constructible sheaf
\[
\mu(\xi) = i_\ast j_! \mathcal{V}_\xi,
\]
where $j\colon S \hookrightarrow \overline{S}$ is the inclusion
of $S$ into its closure, and $i\colon \overline{S} \hookrightarrow
X(\lambda)$ is the embedding of $\overline{S}$ into $X$. 
%$i$ and $j$ are the inclusion maps defined in~\eqref{eq:i-map}
%and~\eqref{eq:j-map}, respectively.  
By \cite[Corollary~23.3(a)-(b)]{ABV}, we have
\[
\chi_S^{\mathrm{loc}}(\mu(\xi)) = \dim( (\mathcal{V}_\xi)_x)=1,
\]
where $x\in S_\xi$.
% Since the set $\{\mu(\xi) : \xi \in
%\Xi(X(\lambda),{}^{\vee}K(\lambda))\}$ forms a basis of the Grothendieck group
%$\mathscr{K}(X,{}^{\vee}K)$, 
By~\cite[(7.11)(a)]{ABV},  any
irreducible perverse sheaf $P(\gamma)$,  $\gamma \in
\Xi(X(\lambda),{}^{\vee}K(\lambda))$ decomposes as a
$\mathbb{Z}$-linear combination
%\label{eq:characterformula}
$$P(\gamma) = \sum_{\xi \in \Xi(X(\lambda),{}^{\vee}K(\lambda))} (-1)^{\dim(S_\xi)}
\ c_g(\xi,\gamma) \, \mu(\xi) $$
in $\mathscr{K}(X(\lambda), {}^{\vee}K(\lambda))$.  
Applying $\chi_S^{\mathrm{loc}}$ to this equation, we see that
$$\chi_S^{\mathrm{loc}}(P(\gamma)) = \sum_{\xi \in \Xi(X(\lambda), {}^{\vee}K(\lambda))}
(-1)^{\dim(S_\xi)} \ c_g(\xi,\gamma)\, \chi_{S}^{\mathrm{loc}}(\mu(\xi)).$$ 
\cite[Theorem~16.19]{ABV} tells us that the coefficients
$c_g(\xi,\gamma)$ are given by Kazhdan–Lusztig–Vogan polynomials,
and the Atlas of Lie Groups and Representations software can compute
them explicitly. This 
provides an effective method for evaluating
$\chi_S^{\mathrm{loc}}(P(\gamma))$.
% In fact the Atlas
%software can do even more for us.  It is capable of computing the full
%expression in Equation~\eqref{eq:characterformula}. 

The local multiplicity map $\chi_{S}^{\mathrm{loc}}$ is related to the
microlocal multiplicity map $\chi_{S}^{\mathrm{mic}}$ through 
Kashiwara's local index formula (\ref{Kashiwaraformula}) which we
state again here.
%More precisely, by the theorem at the end 
%of~\cite[Section 2]{Kashiwara73}, the local and microlocal multiplicities are
%related as follows.   
For every \( {}^{\vee}K(\lambda) \)-orbit \( S' \subset \overline{S}
\), there exists an integer \( a(S, S') \) such that
% for every \( {}^{\vee}K(\lambda) \)-equivariant perverse sheaf on \(
% X(\lambda) \), we have 
\begin{equation}\label{eq:c(SS)'}
\chi_{S}^{\mathrm{loc}}(P(\gamma)) = \sum_{S' \subset \overline{S}}
(-1)^{\dim(S')} \, a(S, S') \ \chi_{S'}^{\mathrm{mic}}(P(\gamma)). 
\end{equation}  
In \cite[Section~3]{MacPherson74}  MacPherson points out
that the local Euler obstructions $a(S, S')$ may be used to determine   the
singularity  of $\overline{S}$ at a point $x \in S'$.
%MacPherson in his construction of 
%Chern classes for singular varieties
We also mention that the inverse of the
relationship in \eqref{eq:c(SS)'} appears in   
\cite[Section 8]{Ginsburg86} as the index formula of Dubson-Kashiwara.
This index formula 
expresses $\chi_{S}^{\mathrm{mic}}$ in terms of the various
$\chi_{S'}^{\mathrm{loc}}$.  There, the matrix entries are
denoted by $c_{S,S'}$, and these entries form the inverse of the matrix
with entries $(-1)^{\dim(S)} a(S',S)$.

In what follows, we explain how to compute them in the specific
setting of our generalized flag variety \( X(\lambda) \). 
Suppose first that \( \lambda \in {^\vee}\mathfrak{t} \) is regular
and integral \eqref{regint}.   
Then \( X(\lambda) = X \) is the full flag variety of \(
{}^{\vee}\mathrm{G}_2 \).   
We fix \( {}^{\vee}K \) as in Section~\ref{Korbits}. Since the
left-hand side of~\eqref{eq:c(SS)'} can be computed using
the Atlas software, and  Theorem~\ref{theo:CCs} gives
the values  $\chi_{S'}^{\mathrm{mic}}(P(\gamma))$ on the right, we
can compute all the local Euler 
obstructions $a(S, S')$ by a simple recursive argument on the
dimension of the ${}^{\vee}K$-orbit $S$. 
We record these values in the form of a matrix, where the entry in
position $(i,j)$ is the local Euler obstruction $a(S_i, S_j)$. 
\[
\begin{bmatrix}
1 & 0 & 0 & 1 & 1 & 1 & 1 & 0 & 0 & 1 \\
0 & 1 & 0 & 1 & 0 & 1 & -1 & 0 & 2 & 1 \\
0 & 0 & 1 & 0 & 1 & 1 & 1 & 2 & 0 & 1 \\
0 & 0 & 0 & 1 & 0 & 1 & 1 & 0 & 1 & 1 \\
0 & 0 & 0 & 0 & 1 & 1 & 1 & 1 & 0 & 1 \\
0 & 0 & 0 & 0 & 0 & 1 & 0 & 1 & 1 & 1 \\
0 & 0 & 0 & 0 & 0 & 0 & 1 & 1 & 1 & 1 \\
0 & 0 & 0 & 0 & 0 & 0 & 0 & 1 & 0 & 1 \\
0 & 0 & 0 & 0 & 0 & 0 & 0 & 0 & 1 & 1 \\
0 & 0 & 0 & 0 & 0 & 0 & 0 & 0 & 0 & 1 \\
\end{bmatrix} 
\]
Examining columns~6, 7, and~8, we deduce
from~\cite[Example~1]{Kashiwara73} that the ${}^{\vee}K$-orbit
closures $\overline{S_6}$ and $\overline{S_8}$ have a singularity at
$S_1$, and that the closure $\overline{S_7}$ has a singularity at
$S_2$. 
All other ${}^{\vee}K$-orbits have smooth closures. Indeed, orbits $S_{0}$,
$S_{1}$  and $S_{2}$ are closed homogeneous spaces, and thus
smooth. The open orbit $S_{9}$ has the smooth variety $X$ as its closure. 
We cannot determine whether orbit closures
$\overline{S}_{3}$, $\overline{S}_{4}$ and $\overline{S}_{5}$ are
smooth or not using \cite[Example~1]{Kashiwara73}.  
However, we have verified that these orbit closures are smooth
by using Magma software; we omit the details. 

Suppose now that \( \lambda \in {^\vee}\mathfrak{t} \) is regular, but not
integral.  The possibilities for the pairs
\(\left({}^{\vee}\mathrm{G}_2(\lambda), {}^{\vee}K(\lambda)\right)\)
are listed in Proposition~\ref{Ilambdaorbits}.  When
\(\left({}^{\vee}\mathrm{G}_2(\lambda), {}^{\vee}K(\lambda)\right)\)
is not isomorphic to  \((\mathrm{SL}_3, \mathrm{GL}_2)\), the orbits
are either open or closed in $X(\lambda)$.  Hence, they are smooth.
For the pair \( (\mathrm{SL}_3, \mathrm{GL}_2) \), the smoothness of
the \( {}^{\vee}K(\lambda) \)-orbits follows from~\cite[Theorem
  7.3.3]{McGovern2023}.  
It follows from \cite[Example~1]{Kashiwara73} that
all local Euler obstructions are equal to one
in these cases. 

Suppose finally that $\lambda \in {^\vee}\mathfrak{t}$ is singular. 
Once again the
Atlas software can be used to compute the local
multiplicities  on the left-hand side of equation~\eqref{eq:c(SS)'}.
We know the microlocal multiplicities on the right-hand side of
equation (\ref{eq:c(SS)'}) by  Theorem \ref{theo:CCs}, Theorem
\ref{theo:CCsNonintegral} and equation~\eqref{eq:translationofchiV1}.
In this manner we may compute the local Euler obstructions $a(S,
S')$.

As an example, we consider $\lambda$ as in equation~\eqref{lampsi1}. 
This is the infinitesimal character corresponding to the subregular
nilpotent orbit  
which we studied at the end of Section \ref{section:singularcase}. 
As we saw there,
% the parabolic ${}^{\vee}P(\lambda)$ is given by 
%${}^{\vee}B\sqcup {}^{\vee}B \sigma_{1} {}^{\vee}B$, and 
taking
\(\lambda' = \rho\) furnishes  a translation datum
\(\mathcal{T}\) from the  ${}^{\vee}\mathrm{G}_2$-orbit of \(\lambda\)
to the ${}^{\vee}\mathrm{G}_2$-orbit of \(\lambda'\).   
Recall the set \(\{\xi_i : 0 \leq i \leq 11\}\) of complete geometric
parameters for the flag variety $X$ of ${}^{\vee}\mathrm{G}_2$ listed in
\eqref{eq:IrrObj1}.  
The
${^\vee}K(\lambda)$-orbit structure on $X(\lambda)$ is given by the
image of the ${^\vee}K(\lambda)$-orbit structure on $X$ under (\ref{eq:ftee}).
There are five ${}^{\vee}K(\lambda)$-orbits on $X(\lambda)$ of
${}^{\vee}\mathrm{G}_2$, which we denote by $Q_i$ for $0 \leq i \leq 4$.  
Among these orbits, $Q_0$ and $Q_1$ are closed and are equal to
$S_{\psi_a}$ and $S_{\psi_b}$ respectively.
The orbit $Q_4$ is open,  
and the closure relations are given as follows
\[
Q_0, Q_1 \subset \overline{Q_2} \subset \overline{Q_3} \subset
\overline{Q_4} = X(\lambda).
\] 
%are encoded in the following Hasse diagram of the Bruhat order.
For orbits $Q_i$ with $0 \leq i \leq 3$, the constant sheaf
\(\underline{\mathbb{C}}_{Q_i}\) is the only irreducible
${}^{\vee}K(\lambda)$-equivariant local system supported on \(Q_i\).   
In orbit $Q_4$, in addition to \(\underline{\mathbb{C}}_{Q_4}\), there
exists a non-trivial local system which we denote by
\(\mathcal{L}_5\).   
In short, there are exactly six complete geometric parameters for
$X(\lambda)$
\[
\gamma_i = (Q_i, \underline{\mathbb{C}}_{Q_i}) \quad \text{for } 0
\leq i \leq 3, \quad\gamma_4 = (Q_4, \underline{\mathbb{C}}_{Q_4}),
\quad  \text{and} \quad \gamma_5 = (Q_4, \mathcal{L}_5). 
\]
 
Using the Atlas software command \texttt{T\_irr} in combination with
(\ref{eq:SingularPimic}), 
we obtain  
\begin{align*}
f_\mathcal{T}^{\ast}(\gamma_0)=\xi_1,\ 
f_\mathcal{T}^{\ast}(\gamma_1)=\xi_4,\ 
f_\mathcal{T}^{\ast}(\gamma_2)=\xi_6,\ 
f_\mathcal{T}^{\ast}(\gamma_3)=\xi_8,\ 
f_\mathcal{T}^{\ast}(\gamma_4)=\xi_9,\ 
f_\mathcal{T}^{\ast}(\gamma_5)=\xi_{10}.
\end{align*}
By Theorem~\ref{theo:CCs} and equation~\eqref{eq:translationofchiV1},
we deduce that 
$
CC(P(\gamma_i)) = \overline{T_{Q_i}^{\ast}X} \quad \text{for } i = 0, 1, 4,
$
and that 
\begin{align*}
CC(P(\gamma_{2})) &= 2\cdot \overline{T_{Q_0}^{\ast}X} +
\overline{T_{Q_2}^{\ast}X}, \\ 
CC(P(\gamma_3))   &= \overline{T_{Q_1}^{\ast}X} +
\overline{T_{Q_2}^{\ast}X} + \overline{T_{Q_3}^{\ast}X}, \\ 
CC(P(\gamma_5))   &= \overline{T_{Q_0}^{\ast}X} +
\overline{T_{Q_2}^{\ast}X} + \overline{T_{Q_3}^{\ast}X} +
\overline{T_{Q_4}^{\ast}X}. 
\end{align*} 
From these equations we can read off the values of $\chi_{Q_{j}}^{\mathrm{mic}}$.
Finally, using the Atlas software to compute the local
multiplicities in  equation~\eqref{eq:c(SS)'} we arrive to the
matrix
\[
\begin{bmatrix}
1 & 0 & 1 & 1 & 1 \\
0 & 1 & -1 & 1 & 1 \\
0 & 0 &  1& 2 & 1 \\
0 & 0 &  0& 1 & 1 \\
0 & 0 & 0 & 0 & 1 \\
\end{bmatrix}
\]
where the entry 
in position \((i, j)\) is the local Euler obstruction \(a(Q_i, Q_j)\).

%\bibliographystyle{alpha}
%\bibliography{reference}

\begin{thebibliography}{AvLTV20}

\bibitem[ABV92]{ABV}
Jeffrey Adams, Dan Barbasch, and David~A. Vogan, Jr.
\newblock {\em The {L}anglands classification and irreducible characters for
  real reductive groups}, volume 104 of {\em Progress in Mathematics}.
\newblock Birkh\"auser Boston, Inc., Boston, MA, 1992.

\bibitem[Ach21]{Achar}
Pramod~N. Achar.
\newblock {\em Perverse sheaves and applications to representation theory},
  volume 258 of {\em Mathematical Surveys and Monographs}.
\newblock American Mathematical Society, Providence, RI, [2021] \copyright
  2021.

\bibitem[AdC09]{Adams-Fokko}
Jeffrey Adams and Fokko du~Cloux.
\newblock Algorithms for representation theory of real reductive groups.
\newblock {\em J. Inst. Math. Jussieu}, 8(2):209--259, 2009.

\bibitem[AJ87]{Adams-Johnson}
Jeffrey Adams and Joseph~F. Johnson.
\newblock Endoscopic groups and packets of nontempered representations.
\newblock {\em Compositio Math.}, 64(3):271--309, 1987.

\bibitem[APM{\etalchar{+}}]{atlas}
Jeffrey Adams, Annegret Paul, Stephen Miller, Marc van Leeuwen, and David
  Vogan.
\newblock Atlas of {L}ie groups and representations software.

\bibitem[AR22]{arancibia_characteristic}
N.~Arancibia~Robert.
\newblock Characteristic cycles, micro local packets and packets with
  cohomology.
\newblock {\em Trans. Amer. Math. Soc.}, 375(2):997--1049, 2022.

\bibitem[Art84]{Arthur84}
James Arthur.
\newblock On some problems suggested by the trace formula.
\newblock In {\em Lie group representations, {II} ({C}ollege {P}ark, {M}d.,
  1982/1983)}, volume 1041 of {\em Lecture Notes in Math.}, pages 1--49.
  Springer, Berlin, 1984.

\bibitem[Art89]{Arthur89}
James Arthur.
\newblock Unipotent automorphic representations: conjectures.
\newblock {\em Ast\'erisque}, (171-172):13--71, 1989.
\newblock Orbites unipotentes et repr{\'e}sentations, II.

\bibitem[AV15]{AVParameters}
Jeffrey Adams and David~A. Vogan, Jr.
\newblock Parameters for twisted representations.
\newblock In {\em Representations of reductive groups}, volume 312 of {\em
  Progr. Math.}, pages 51--116. Birkh\"{a}user/Springer, Cham, 2015.

\bibitem[AvLTV20]{AvLTV}
Jeffrey~D. Adams, Marc A.~A. van Leeuwen, Peter~E. Trapa, and David~A. Vogan,
  Jr.
\newblock Unitary representations of real reductive groups.
\newblock {\em Ast\'{e}risque}, (417):viii + 188, 2020.

\bibitem[BB85]{BoBrIII}
W.~Borho and J.-L. Brylinski.
\newblock Differential operators on homogeneous spaces. {III}. {C}haracteristic
  varieties of {H}arish-{C}handra modules and of primitive ideals.
\newblock {\em Invent. Math.}, 80(1):1--68, 1985.

\bibitem[BGK{\etalchar{+}}87]{Boreletal}
A.~Borel, P.-P. Grivel, B.~Kaup, A.~Haefliger, B.~Malgrange, and F.~Ehlers.
\newblock {\em Algebraic {$D$}-modules}, volume~2 of {\em Perspectives in
  Mathematics}.
\newblock Academic Press, Inc., Boston, MA, 1987.

\bibitem[BL94]{Lunts}
Joseph Bernstein and Valery Lunts.
\newblock {\em Equivariant sheaves and functors}, volume 1578 of {\em Lecture
  Notes in Mathematics}.
\newblock Springer-Verlag, Berlin, 1994.

\bibitem[BV83]{Barbasch-Vogan-Trombi}
Dan Barbasch and David Vogan.
\newblock Weyl group representations and nilpotent orbits.
\newblock In {\em Representation theory of reductive groups ({P}ark {C}ity,
  {U}tah, 1982)}, volume~40 of {\em Progr. Math.}, pages 21--33. Birkh\"auser
  Boston, Boston, MA, 1983.

\bibitem[BZ08]{Barchini-Zierau}
L.~Barchini and R.~Zierau.
\newblock Certain components of {S}pringer fibers and associated cycles for
  discrete series representations of {${\rm SU}(p,q)$}.
\newblock {\em Represent. Theory}, 12:403--434, 2008.
\newblock With an appendix by Peter E. Trapa.

\bibitem[Car85]{Carter85}
Roger~W. Carter.
\newblock {\em Finite groups of {L}ie type}.
\newblock Pure and Applied Mathematics (New York). John Wiley \& Sons, Inc.,
  New York, 1985.
\newblock Conjugacy classes and complex characters, A Wiley-Interscience
  Publication.

\bibitem[CFZ21]{Cunningham21}
Clifton Cunningham, Andrew Fiori, and Qing Zhang.
\newblock Toward the endoscopic classification of unipotent representations of
  $p$-adic $g_2$, 2021.

\bibitem[CFZ22]{Cunningham22}
Clifton Cunningham, Andrew Fiori, and Qing Zhang.
\newblock Arthur packets for {$G_2$} and perverse sheaves on cubics.
\newblock {\em Adv. Math.}, 395:Paper No. 108074, 74, 2022.

\bibitem[CG10]{Chriss-Ginzburg}
Neil Chriss and Victor Ginzburg.
\newblock {\em Representation theory and complex geometry}.
\newblock Modern Birkh\"auser Classics. Birkh\"auser Boston, Ltd., Boston, MA,
  2010.
\newblock Reprint of the 1997 edition.

\bibitem[CM93]{Collingwood-McGovern}
David~H. Collingwood and William~M. McGovern.
\newblock {\em Nilpotent orbits in semisimple {L}ie algebras}.
\newblock Van Nostrand Reinhold Mathematics Series. Van Nostrand Reinhold Co.,
  New York, 1993.

\bibitem[CNT12]{CNT}
Dan Ciubotaru, Kyo Nishiyama, and Peter~E. Trapa.
\newblock Regular orbits of symmetric subgroups on partial flag varieties.
\newblock In {\em Representation theory, complex analysis, and integral
  geometry}, pages 61--86. Birkh\"auser/Springer, New York, 2012.

\bibitem[Ful98]{fulton}
William Fulton.
\newblock {\em Intersection theory}, volume~2 of {\em Ergebnisse der Mathematik
  und ihrer Grenzgebiete. 3. Folge. A Series of Modern Surveys in Mathematics
  [Results in Mathematics and Related Areas. 3rd Series. A Series of Modern
  Surveys in Mathematics]}.
\newblock Springer-Verlag, Berlin, second edition, 1998.

\bibitem[Gin86]{Ginsburg86}
V.~Ginsburg.
\newblock Characteristic varieties and vanishing cycles.
\newblock {\em Invent. Math.}, 84(2):327--402, 1986.

\bibitem[GS81]{Gonzalez-Sprinberg}
Gerardo Gonz\'alez-Sprinberg.
\newblock L'obstruction locale d'{E}uler et le th\'eor\`eme de {M}ac{P}herson.
\newblock In {\em The {E}uler-{P}oincar\'e{} characteristic ({F}rench)}, volume
  82-83 of {\em Ast\'erisque}, pages 7--32. Soc. Math. France, Paris, 1981.

\bibitem[Hot84]{HottaR}
Ryoshi Hotta.
\newblock On {J}oseph's construction of {W}eyl group representations.
\newblock {\em Tohoku Math. J. (2)}, 36(1):49--74, 1984.

\bibitem[Hot85]{Hotta85}
Ryoshi Hotta.
\newblock A local formula for {S}pringer's representation.
\newblock In {\em Algebraic groups and related topics ({K}yoto/{N}agoya,
  1983)}, volume~6 of {\em Adv. Stud. Pure Math.}, pages 127--138.
  North-Holland, Amsterdam, 1985.

\bibitem[HTT08]{Hotta}
Ryoshi Hotta, Kiyoshi Takeuchi, and Toshiyuki Tanisaki.
\newblock {\em {$D$}-modules, perverse sheaves, and representation theory},
  volume 236 of {\em Progress in Mathematics}.
\newblock Birkh\"{a}user Boston, Inc., Boston, MA, 2008.
\newblock Translated from the 1995 Japanese edition by Takeuchi.

\bibitem[Hum75]{HumphLAG}
James~E. Humphreys.
\newblock {\em Linear algebraic groups}, volume No. 21 of {\em Graduate Texts
  in Mathematics}.
\newblock Springer-Verlag, New York-Heidelberg, 1975.

\bibitem[Hum78]{Humphreys}
James~E. Humphreys.
\newblock {\em Introduction to {L}ie algebras and representation theory},
  volume~9 of {\em Graduate Texts in Mathematics}.
\newblock Springer-Verlag, New York-Berlin, 1978.
\newblock Second printing, revised.

\bibitem[Hum95]{HumphCC}
James~E. Humphreys.
\newblock {\em Conjugacy classes in semisimple algebraic groups}, volume~43 of
  {\em Mathematical Surveys and Monographs}.
\newblock American Mathematical Society, Providence, RI, 1995.

\bibitem[Jan79]{Jantzen}
Jens~Carsten Jantzen.
\newblock {\em Moduln mit einem h\"{o}chsten {G}ewicht}, volume 750 of {\em
  Lecture Notes in Mathematics}.
\newblock Springer, Berlin, 1979.

\bibitem[Kas73]{Kashiwara73}
Masaki Kashiwara.
\newblock Index theorem for a maximally overdetermined system of linear
  differential equations.
\newblock {\em Proc. Japan Acad.}, 49:803--804, 1973.

\bibitem[Kna86]{Knapp}
Anthony~W. Knapp.
\newblock {\em Representation theory of semisimple groups}, volume~36 of {\em
  Princeton Mathematical Series}.
\newblock Princeton University Press, Princeton, NJ, 1986.
\newblock An overview based on examples.

\bibitem[Kna02]{Beyond}
Anthony~W. Knapp.
\newblock {\em Lie groups beyond an introduction}, volume 140 of {\em Progress
  in Mathematics}.
\newblock Birkh\"{a}user Boston, Inc., Boston, MA, second edition, 2002.

\bibitem[KR71]{Kostant-Rallis}
B.~Kostant and S.~Rallis.
\newblock Orbits and representations associated with symmetric spaces.
\newblock {\em Amer. J. Math.}, 93:753--809, 1971.

\bibitem[Lan89]{Langlands}
R.~P. Langlands.
\newblock On the classification of irreducible representations of real
  algebraic groups.
\newblock In {\em Representation theory and harmonic analysis on semisimple
  {L}ie groups}, volume~31 of {\em Math. Surveys Monogr.}, pages 101--170.
  Amer. Math. Soc., Providence, RI, 1989.

\bibitem[Mac74]{MacPherson74}
R.~D. MacPherson.
\newblock Chern classes for singular algebraic varieties.
\newblock {\em Ann. of Math. (2)}, 100:423--432, 1974.

\bibitem[Max19]{Maxim}
Lauren\c tiu~G. Maxim.
\newblock {\em Intersection homology \& perverse sheaves---with applications to
  singularities}, volume 281 of {\em Graduate Texts in Mathematics}.
\newblock Springer, Cham, [2019] \copyright 2019.

\bibitem[McG23]{McGovern2023}
William~M. McGovern.
\newblock {\em Representation theory and geometry of the flag variety},
  volume~90 of {\em De Gruyter Studies in Mathematics}.
\newblock De Gruyter, Berlin, [2023] \copyright 2023.

\bibitem[{\DJ}ok00]{Dokovic}
Dragomir~\v{Z}. {\DJ}okovi\'c.
\newblock The closure diagrams for nilpotent orbits of real forms of {$F_4$}
  and {$G_2$}.
\newblock {\em J. Lie Theory}, 10(2):491--510, 2000.

\bibitem[Ros91]{Rossmann91}
W.~Rossmann.
\newblock Invariant eigendistributions on a semisimple {L}ie algebra and
  homology classes on the conormal variety. {II}. {R}epresentations of {W}eyl
  groups.
\newblock {\em J. Funct. Anal.}, 96(1):155--193, 1991.

\bibitem[Ros95]{Rossmann95}
W.~Rossmann.
\newblock Picard-{L}efschetz theory for the coadjoint quotient of a semisimple
  {L}ie algebra.
\newblock {\em Invent. Math.}, 121(3):531--578, 1995.

\bibitem[RS90]{RS90}
R.~W. Richardson and T.~A. Springer.
\newblock The {B}ruhat order on symmetric varieties.
\newblock {\em Geom. Dedicata}, 35(1-3):389--436, 1990.

\bibitem[Sam14]{Samples}
Brandon Samples.
\newblock Components of {S}pringer fibers for the exceptional groups {$G_2$}
  and {$F_4$}.
\newblock {\em J. Algebra}, 400:219--248, 2014.

\bibitem[She82]{Shel82}
D.~Shelstad.
\newblock {$L$}-indistinguishability for real groups.
\newblock {\em Math. Ann.}, 259(3):385--430, 1982.

\bibitem[Spa82]{Spaltenstein}
Nicolas Spaltenstein.
\newblock {\em Classes unipotentes et sous-groupes de {B}orel}, volume 946 of
  {\em Lecture Notes in Mathematics}.
\newblock Springer-Verlag, Berlin-New York, 1982.

\bibitem[Tan85]{Tanisaki}
Toshiyuki Tanisaki.
\newblock Holonomic systems on a flag variety associated to {H}arish-{C}handra
  modules and representations of a {W}eyl group.
\newblock In {\em Algebraic groups and related topics ({K}yoto/{N}agoya,
  1983)}, volume~6 of {\em Adv. Stud. Pure Math.}, pages 139--154.
  North-Holland, Amsterdam, 1985.

\bibitem[Tit66]{Tits}
J.~Tits.
\newblock Classification of algebraic semisimple groups.
\newblock In {\em Algebraic {G}roups and {D}iscontinuous {S}ubgroups ({P}roc.
  {S}ympos. {P}ure {M}ath., {B}oulder, {C}olo., 1965)}, pages 33--62. Amer.
  Math. Soc., Providence, RI, 1966.

\bibitem[Tra05]{Trapa2005}
Peter~E. Trapa.
\newblock Symplectic and orthogonal {R}obinson-{S}chensted algorithms.
\newblock {\em J. Algebra}, 286(2):386--404, 2005.

\bibitem[Vog81]{greenbook}
David~A. Vogan, Jr.
\newblock {\em Representations of real reductive {L}ie groups}, volume~15 of
  {\em Progress in Mathematics}.
\newblock Birkh\"{a}user, Boston, Mass., 1981.

\bibitem[Vog83]{ICIII}
David~A. Vogan.
\newblock Irreducible characters of semisimple {L}ie groups. {III}. {P}roof of
  {K}azhdan-{L}usztig conjecture in the integral case.
\newblock {\em Invent. Math.}, 71(2):381--417, 1983.

\bibitem[Vog91]{Vogan-Avariety}
David~A. Vogan, Jr.
\newblock Associated varieties and unipotent representations.
\newblock In {\em Harmonic analysis on reductive groups ({B}runswick, {ME},
  1989)}, volume 101 of {\em Progr. Math.}, pages 315--388. Birkh\"auser
  Boston, Boston, MA, 1991.

\bibitem[Vog94]{VoganG2}
David~A. Vogan, Jr.
\newblock The unitary dual of {$G_2$}.
\newblock {\em Invent. Math.}, 116(1-3):677--791, 1994.

\bibitem[VZ84]{Vogan-Zuckerman}
David~A. Vogan, Jr. and Gregg~J. Zuckerman.
\newblock Unitary representations with nonzero cohomology.
\newblock {\em Compositio Math.}, 53(1):51--90, 1984.

\bibitem[Yam97]{Yamamoto}
Atsuko Yamamoto.
\newblock Orbits in the flag variety and images of the moment map for classical
  groups. {I}.
\newblock {\em Represent. Theory}, 1:329--404, 1997.

\end{thebibliography}

\end{document}